\documentclass[11pt]{scrartcl}

\usepackage[utf8]{inputenc}

\usepackage{config}
\usepackage{titling}

\usepackage[a4paper, top=3cm, bottom=2cm, left=2.5cm, right=2.5cm, includefoot]{geometry}

\usepackage[inline]{enumitem}

\setlist[itemize]{topsep=0pt,partopsep=0pt,itemsep=0pt,parsep=0pt}
\setlist[itemize,1]{label={\small\textbullet}}
\setlist[itemize,2]{label={\tiny\textbullet}}
\setlist[itemize,3]{label=$\cdot$}
\setlist[enumerate]{topsep=0pt,partopsep=0pt,itemsep=0pt,parsep=0pt}
\setlist[enumerate,1]{label=\roman*)}
\setlist[enumerate,2]{label=\alph*)}
\setlist[enumerate,3]{label=\arabic*)}

\hypersetup{
	colorlinks=true,
	linkcolor=AO!65!black,
	citecolor=AO!65!black,
	urlcolor=AppleGreen!65!black,
	bookmarksopen=true,
	bookmarksnumbered,
	bookmarksopenlevel=2,
	bookmarksdepth=3
}

\title{Excluding a Planar Matching Minor in Bipartite Graphs\thanks{This research has been supported by the ERC consolidator grant DISTRUCT-648527.}}
\predate{}
\date{}
\postdate{}

\preauthor{}
\DeclareRobustCommand{\authorthing}{
	\begin{center}
		\begin{tabular}{p{0.30\textwidth}p{.19\textwidth}p{.26\textwidth}}
			Archontia C. Giannopoulou\thanks{archontia.giannopoulou@ gmail.com} & Stephan Kreutzer\thanks{stephan.kreutzer@tu-berlin.de} & Sebastian Wiederrecht\thanks{sebastian.wiederrecht@gmail.com}\\
			\emph{National and Kapodistrian University of Athens} & \multicolumn{2}{c}{\emph{Technische Universität Berlin}}
		\end{tabular}
\end{center}}
\author{\authorthing}
\postauthor{}

\begin{document}
\maketitle

\begin{abstract}
Matching minors are a specialisation of minors fit for the study of graph with perfect matchings.
The notion of matching minors has been used to give a structural description of bipartite graphs on which the number of perfect matchings can be computed efficiently, based on a result of Little, by McCuaig et al.\@ in 1999.

In this paper we generalise basic ideas from the graph minor series by Robertson and Seymour to the setting of bipartite graphs with perfect matchings.
We introduce a version of Erd\H{o}s-P{\'o}sa property for matching minors and find a direct link between this property and planarity.
From this, it follows that a class of bipartite graphs with perfect matchings has bounded perfect matching width if and only if it excludes a planar matching minor.
We also present algorithms for bipartite graphs of bounded perfect matching width for a matching version of the disjoint paths problem, matching minor containment, and for counting the number of perfect matchings.
From our structural results, we obtain that recognising whether a bipartite graph $G$ contains a fixed planar graph $H$ as a matching minor, and that counting the number of perfect matchings of a bipartite graph that excludes a fixed planar graph as a matching minor are both polynomial time solvable.

\noindent \textbf{Keywords}: Disjoint Paths, Matching Minor, Perfect Matching, Erd\H{o}s-Pósa property, Counting Perfect Matchings, Digraphs, Butterfly Minor
\end{abstract}

\section{Introduction}

Graph minors are a generalisation of subgraphs which preserve, among other attributes, the topological properties of their host graph.
In particular this means that a graph $G$ embeds on a surface of genus $g$ if and only if all of its minors do so.
The classical theorem of Kuratowski and Wagner \cite{wagner1937eigenschaft,kuratowski1930probleme} uses this link between embeddability and the exclusion of certain minors by providing a compellingly short list of excluded minors which characterise exactly those graphs embeddable in the plane, namely $K_{3,3}$ and $K_5$.
The fact that the number of excluded minors for this famous theorem is finite lead to the question whether this is just a coincidence, or if there is a bigger rule behind it.
This question became known as Wagner's Conjecture, which was made a theorem almost 70 years later by Robertson and Seymour as the final result of their Graph Minors Project \cite{robertson2004graph}.
The Graph Minors Project revealed several deep connections between graph compositions, excluded minors and graphs that embed on surfaces of bounded genus.
In particular, the main tools and findings of the Graph Minors Project can be seen as a natural generalisation of Wagner's characterisation of $K_5$-minor free graphs as those which can be built from planar graphs and a single non-planer graph $W_8$ by means of small clique sums \cite{wagner1937eigenschaft}.

Roughly speaking, the Graph Minors Project can be broken down into the following steps or tools:
\begin{enumerate}
	\item The introduction of a complexity measure that describes the structure of a graph and allows to rapidly simplify graphs where the measure is small.
	In the Graph Minors Project this role is played by \emph{treewidth} and the idea of \emph{tree-decompositions} \cite{halin1976s,robertson1986graphB}.
	\item A (rough) characterisation of minor closed classes of graphs where the complexity measure is small by linking the measure to a topological property.
	In the case of graph minors this was done in form of the \emph{Grid Theorem} and the resulting corollary that any proper minor closed class of graphs has bounded treewidth if and only if it excludes a planar minor \cite{robertson1986graph}.
	\item An extension of the second step to fully describe highly connected graphs where the complexity measure is large, but which exclude some non-planar minor \cite{robertson1995graph,kawarabayashi2020quickly}.
%	For graph minors this was done in form of the \emph{Flat Wall Theorem} and the so called \emph{Local Structure Theorem} which uses tangles \cite{robertson1995graph,kawarabayashi2020quickly}.
	\item Finally, a combination of all previous steps which results in a rough (but global) description of all graphs excluding a fixed minor \cite{robertson2003graph,kawarabayashi2020quickly}.
%	This is known as the \emph{Global Structure Theorem}, or the \emph{Graph Minors Theorem}\cite{robertson2003graph,kawarabayashi2020quickly}.
\end{enumerate} 

An equivalent formulation of P{\'o}lya's Permanent Problem, the problem of recognising those $0$-$1$-matrices whose permanent can be computed efficiently using a specific method, which is known as the bipartite Pfaffian Recognition Problem was shown in 1975 to correspond to the containment of $K_{3,3}$ as a so called \emph{matching minor}\footnote{A version of graph minors which also preserves the structure of perfect matchings in a graph. We give formal definitions of these concepts in \cref{subsec:preliminaries}.} in bipartite graphs with perfect matchings \cite{little1975characterization}.
A structural description of bipartite graphs excluding $K_{3,3}$ as a matching minor was later found by Seymour et al.\@ and McCuaig independently \cite{robertson1999permanents,mccuaig2004polya}.
Interestingly, this description is somewhat similar to Wanger's Theorem on $K_5$-minor free graphs in the following sense:
The theorem states that every brace that excludes $K_{3,3}$ as a matching minor can be created from planar braces and a single non-planar brace by means of a matching theoretic analogue of clique sums.
Inspired by this observation, we pursue an extension of the Graph Minors Project to matching minors in bipartite graphs.
%
%An algorithm that decides in polynomial time whether a given bipartite graph with a perfect matchings contains $K_{3,3}$ as a matching minor was later found by Seymour et al.\@ and McCuaig independently \cite{seymour1987characterization,mccuaig2004polya}\footnote{For a definition of P{\'o}lya's Permanent Problem and Pfaffian Orientation, as well as an overview on the topic we recommend \cite{mccuaig2004polya} for further reading.}.
%A somewhat surprising detail of the structural result of McCuaig et al.\@ which made the algorithm possible is its similarity to Wagner's Theorem of $K_5$-minor free graphs and to the general structure theorem of Robertson and Seymour on $H$-minor free graphs.
%The theorem states that every brace that excludes $K_{3,3}$ as a matching minor can be created from planar braces and a single non-planar brace by means of a matching theoretic analogue of clique sums.
%Inspired by this observation, we pursue an extension of the Graph Minors Project to matching minors in bipartite graphs.

\paragraph{The project so far}

This paper is part of the larger project of extending the graph minors theory of Robertson and Seymour to bipartite graphs with perfect matchings.

A matching theoretic analogue of treewidth, called \emph{perfect matching width}, was introduced by Norine \cite{norine2005matching}.
Together with Hatzel and Rabinovich, the third author also derived a grid theorem for bipartite graphs with perfect matchings and perfect matching width from the related Directed Grid Theorem \cite{hatzel2019cyclewidth,rabinovich2019cyclewidth}.

Roughly speaking, \hyperref[def:pmw]{perfect matching width} is a branch decomposition over the set of vertices of a graph and the weight of the edges of the corresponding cubic decomposition tree depends on the perfect matchings of the graph.
By $\pmw{G}$ we denote the perfect matching width of a graph.

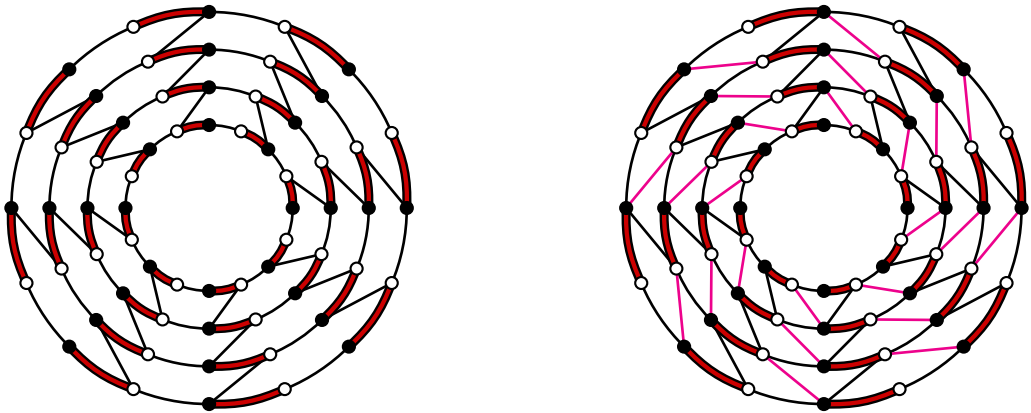
\begin{figure}[!h]
	\begin{subfigure}{0.5\textwidth}
		\centering
		\begin{tikzpicture}

			\pgfdeclarelayer{background}
			\pgfdeclarelayer{foreground}
			\pgfsetlayers{background,main,foreground}

			\draw[e:main] (0,0) circle (11mm);
			\draw[e:main] (0,0) circle (16mm);
			\draw[e:main] (0,0) circle (21mm);
			\draw[e:main] (0,0) circle (26mm);
			
			\foreach \x in {1,...,4}
			{
				\draw[e:main] (\x*90:16mm) -- (\x*90+22.5:11mm);
				\draw[e:main] (\x*90:21mm) -- (\x*90+22.5:16mm);
				\draw[e:main] (\x*90:26mm) -- (\x*90+22.5:21mm);
			}
			
			\foreach \x in {1,...,4}
			{
				\draw[e:main] (\x*90-22.5:16mm) -- (\x*90-45:11mm);
				\draw[e:main] (\x*90-22.5:21mm) -- (\x*90-45:16mm);
				\draw[e:main] (\x*90-22.5:26mm) -- (\x*90-45:21mm);
				
			}
			
			\foreach \x in {1,...,8}
			{
				\draw[e:coloredthin,color=BostonUniversityRed,bend right=13] (\x*45:11mm) to (\x*45+22.5:11mm);
				\draw[e:coloredthin,color=BostonUniversityRed,bend right=13] (\x*45:16mm) to (\x*45+22.5:16mm);
				\draw[e:coloredthin,color=BostonUniversityRed,bend right=13] (\x*45:21mm) to (\x*45+22.5:21mm);
				\draw[e:coloredthin,color=BostonUniversityRed,bend right=13] (\x*45:26mm) to (\x*45+22.5:26mm);
			}
			
			\foreach \x in {1,...,8}
			{
				\node[v:main] () at (\x*45:11mm){};
				\node[v:main] () at (\x*45:16mm){};
				\node[v:main] () at (\x*45:21mm){};
				\node[v:main] () at (\x*45:26mm){};
				\node[v:mainempty] () at (\x*45+22.5:11mm){};
				\node[v:mainempty] () at (\x*45+22.5:16mm){};
				\node[v:mainempty] () at (\x*45+22.5:21mm){};
				\node[v:mainempty] () at (\x*45+22.5:26mm){};
			}

			\begin{pgfonlayer}{background}
				\foreach \x in {1,...,8}
				{
					\draw[e:coloredborder,bend right=13] (\x*45:11mm) to (\x*45+22.5:11mm);
					\draw[e:coloredborder,bend right=13] (\x*45:16mm) to (\x*45+22.5:16mm);
					\draw[e:coloredborder,bend right=13] (\x*45:21mm) to (\x*45+22.5:21mm);
					\draw[e:coloredborder,bend right=13] (\x*45:26mm) to (\x*45+22.5:26mm);
				}
			\end{pgfonlayer}
			
		\end{tikzpicture}
	\end{subfigure}
	\begin{subfigure}{0.5\textwidth}
		\centering
		\begin{tikzpicture}

			\pgfdeclarelayer{background}
			\pgfdeclarelayer{foreground}
			\pgfsetlayers{background,main,foreground}
			\tikzstyle{vertex}=[shape=circle, fill=black, draw, inner sep=.4mm]
			\tikzstyle{emptyvertex}=[shape=circle, fill=white, draw, inner sep=.45mm]

			\draw[e:main] (0,0) circle (11mm);
			\draw[e:main] (0,0) circle (16mm);
			\draw[e:main] (0,0) circle (21mm);
			\draw[e:main] (0,0) circle (26mm);
			
			\foreach \x in {1,...,4}
			{
				\draw[e:main] (\x*90:16mm) -- (\x*90+22.5:11mm);
				\draw[e:main] (\x*90:21mm) -- (\x*90+22.5:16mm);
				\draw[e:main] (\x*90:26mm) -- (\x*90+22.5:21mm);
			}
			
			\foreach \x in {1,...,4}
			{
				\draw[e:main] (\x*90-22.5:16mm) -- (\x*90-45:11mm);
				\draw[e:main] (\x*90-22.5:21mm) -- (\x*90-45:16mm);
				\draw[e:main] (\x*90-22.5:26mm) -- (\x*90-45:21mm);
			}
			
			\foreach \x in {1,...,8}
			{
				\draw[magenta,e:main] (\x*45:16mm) -- (\x*45-22.5:11mm);
				\draw[magenta,e:main] (\x*45:21mm) -- (\x*45-22.5:16mm);
				\draw[magenta,e:main] (\x*45:26mm) -- (\x*45-22.5:21mm);
				
			}
			
			\foreach \x in {1,...,8}
			{
				\draw[e:coloredthin,color=BostonUniversityRed,bend right=13] (\x*45:11mm) to (\x*45+22.5:11mm);
				\draw[e:coloredthin,color=BostonUniversityRed,bend right=13] (\x*45:16mm) to (\x*45+22.5:16mm);
				\draw[e:coloredthin,color=BostonUniversityRed,bend right=13] (\x*45:21mm) to (\x*45+22.5:21mm);
				\draw[e:coloredthin,color=BostonUniversityRed,bend right=13] (\x*45:26mm) to (\x*45+22.5:26mm);
			}
			
			\foreach \x in {1,...,8}
			{
				\node[v:main] () at (\x*45:11mm){};
				\node[v:main] () at (\x*45:16mm){};
				\node[v:main] () at (\x*45:21mm){};
				\node[v:main] () at (\x*45:26mm){};
				\node[v:mainempty] () at (\x*45+22.5:11mm){};
				\node[v:mainempty] () at (\x*45+22.5:16mm){};
				\node[v:mainempty] () at (\x*45+22.5:21mm){};
				\node[v:mainempty] () at (\x*45+22.5:26mm){};
			}
			
			\begin{pgfonlayer}{background}
				\foreach \x in {1,...,8}
				{
					\draw[e:coloredborder,bend right=13] (\x*45:11mm) to (\x*45+22.5:11mm);
					\draw[e:coloredborder,bend right=13] (\x*45:16mm) to (\x*45+22.5:16mm);
					\draw[e:coloredborder,bend right=13] (\x*45:21mm) to (\x*45+22.5:21mm);
					\draw[e:coloredborder,bend right=13] (\x*45:26mm) to (\x*45+22.5:26mm);
				}
			\end{pgfonlayer}
		\end{tikzpicture}
	\end{subfigure}
	\caption{The \hyperref[def:matchinggrid]{cylindrical matching grid} of order $4$ with the canonical matching on the left and an internal quadrangulation on the right.}
	\label{fig:cylindricalgrid}
\end{figure}

\begin{theorem}[\cite{hatzel2019cyclewidth,rabinovich2019cyclewidth}]\label{thm:matchinggrid}
	There exists a function $\MatchingCylinder\colon\N\rightarrow\N$ such that for every $k\in\N$ and every bipartite graph $B$ with a perfect matching $M$ either $\pmw{B}\leq\Fkt{\MatchingCylinder}{k}$ or $B$ contains \hyperref[def:matchinggrid]{$CG_k$} as an \hyperref[def:matchingminor]{$M$-minor} such that $M$ contains the \hyperref[def:matchinggrid]{canonical matching} of $CG_k$.
\end{theorem}

\subsection{Our Contribution}\label{subsec:contribution}

Towards the greater goal of extending the Graph Minors Project to bipartite graphs with perfect matchings and matching minors, this paper focusses on \cite{robertson1986graph} and the completion of step \emph{ii)} of the Graph Minors Project.
Since many of the notions used to state our results are relatively technical we postpone a formal introduction of the definitions to \cref{subsec:preliminaries}.
For the readers convenience each statement contains references to the necessary definitions.

\begin{theorem}\label{thm:boundedclasses}
A proper \hyperref[def:matchingminor]{matching minor} closed class $\mathcal{B}$ of bipartite graphs has bounded \hyperref[def:pmw]{perfect matching width} if and only if it excludes a planar bipartite \hyperref[def:matchingcovered]{matching covered} graph.
\end{theorem}

To prove this theorem we first show that the cylindrical grid, which is guaranteed by \cref{thm:matchinggrid}, contains a square grid as a matching minor.
Then we use the theory of ear decompositions of matching covered graphs to construct a matching minor model of any fixed planar and matching covered graph within an appropriately sized square grid.

\begin{theorem}\label{thm:matchinggridminors}
	For every planar bipartite \hyperref[def:matchingcovered]{matching covered} graph $H$ there exists a number $\omega_H\in\N$ such that $H$ is a \hyperref[def:matchingminor]{matching minor} of the \hyperref[def:matchinggrid]{cylindrical matching grid} of order $\omega_H$.
\end{theorem}

Besides the characterisation of classes of bounded perfect matching width, we find an extension of the Erd\H{o}s-P\'osa property for minors to bipartite graphs with perfect matchings.

\begin{definition}[(Bipartite) Erd\H{o}s-P\'osa Property for Matching Minors]\label{def:matchingEP}
	A (bipartite) matching covered graph $H$ has the (bipartite) \emph{Erd\H{o}s-P{\'o}sa property for matching minors} if there exists a function $\varepsilon_H\colon\mathbb{N}\rightarrow\mathbb{N}$ such that for every $k\in\N$ any given (bipartite) \hyperref[def:matchingcovered]{matching covered} graph $G$ with a perfect matching $M$ has $k$-pairwise disjoint \hyperref[def:conformal]{$M$-conformal} subgraphs, all of which contain $H$ as a \hyperref[def:matchingminor]{matching minor}, or there exists an $M$-conformal set $S_H\subseteq\V{G}$ with $\Abs{S_H}\leq\Fkt{\varepsilon_H}{k}$ such that $G-S_H$ does not have $H$ as a matching minor.
\end{definition}

\begin{theorem}\label{thm:matchingEP}
	A bipartite \hyperref[def:matchingcovered]{matching covered} graph $H$ has the bipartite Erd\H{o}s-P\'osa property for \hyperref[def:matchingminor]{matching minors} if and only if it is planar. 
\end{theorem}

Towards establishing \cref{thm:matchingEP} we overcome two particular challenges:
\begin{itemize}
	\item We need to establish that for any perfect matching $M$ and any edge cut $\CutG{}{X}$ there exists a set $F\subseteq M$ of size bounded in a function of the \hyperref[def:matchingporosity]{matching porosity} of $\CutG{}{X}$ such that $F$ meets all \hyperref[def:alternating]{$M$-alternating cycles} with edges in $\CutG{}{X}$.
	The existence of such sets is not guaranteed by the definition of perfect matching width.
	By using these sets we are also able to improve on several results of \cite{hatzel2019cyclewidth} regarding directed treewidth.
	This is the foundation of our algorithmic applications of perfect matching width.
	
	\item By deleting a conformal set it is possible to drastically reduce the number of perfect matchings in a bipartite graph.
	In particular, a subgraph $H$ of $B$ might be \hyperref[def:conformal]{conformal} in $B$, but if $F$ is subset of edges of some perfect matching $M$ of $B$ such that $H$ is not $M$-conformal, then $F$ might not be a conformal subgraph of $B-\V{F}$.
	This means that the set $S_H$ from \cref{def:matchingEP} is not necessarily a hitting set for all matching minor models of $H$ within $B$, but still makes sure that $B-S_H$ does not contain $H$ as a matching minor.
	For this reason the method to prove that a non-planar graph cannot have the Erd\H{os}-P\'osa property for minors cannot be extended to bipartite graphs with perfect matchings in a straightforward fashion.
\end{itemize}

Our resolution of the first challenge also solves an Erd\H{o}s-P\'osa type problem on directed cycles through a given set of vertices and it is probably the deepest result in this paper.
The second challenge is resolved by describing what excluding a matching minor means for digraphs, and establishing a \hyperref[def:stronggenus]{strong} version of genus and a new notion of the Erd\H{o}s-P\'osa property for digraphs.
These new concepts are built upon the fact that whole anti-chains of \hyperref[def:butterflyminor]{butterfly minors} can, in some sense, be identified with a single bipartite graph with a perfect matching.
By using these anti-chains and the insight gained from the matching theoretic context, we are able to reformulate and resolve some problems from the world of structural digraph theory.
How the setting of digraphs relates to bipartite graphs with perfect matchings is briefly described in \cref{subsec:digraphs}, where we also discuss our contributions to digraph theory.

\paragraph{Algorithmic Applications of Perfect Matching Width}

In their seminal work on directed treewidth \cite{johnson2001directed} Johnson et al.\@ listed three main points which, in their eyes, made treewidth a successful parameter\footnote{They actually list four points, but the fourth is the successful use of treewidth in practical application which is unlikely to be replicable for perfect matching width at the time of writing.}.
These points are
\begin{itemize}
	\item It served as a cornerstone of the Graph Minors Project,
	\item it can be used to prove structural theorems, and
	\item it has algorithmic applications due to the fact that many $\NP$-hard problems become tractable on classes of bounded width.
\end{itemize}
The current state of research, in particular \cref{thm:matchinggrid} and \cref{thm:boundedclasses}, can be seen as evidence that perfect matching width can take the place of treewidth in the context of matching minors in bipartite graphs for the first point.
For the second point we can consider \cref{thm:matchingEP} to be a nice first result made possible through perfect matching width.
Hence, to further strengthen our claim for the usefulness of perfect matching width, we introduce some algorithmic applications of perfect matching width.
When considering bipartite graphs from the view of Matching Theory, two particular problems appear naturally, namely
\begin{enumerate}
\item what is the computational complexity of recognising whether a given bipartite graph contains a fixed bipartite and matching covered graph $H$ as a matching minor, and
\item what is the complexity of counting the number of perfect matchings in bipartite graphs excluding a fixed graph $H$ as a matching minor? 
\end{enumerate}
To be able to use perfect matching width for any of these two questions, we must first show that we can compute a decomposition of bounded width in polynomial time.

\begin{theorem}\label{thm:approximatepmw}
Let $B$ be a bipartite graph with a perfect matching and $\pmw{B}\leq w$.
There exists a constant $c_{\operatorname{pmw}}\in\N$ and an algorithm with running time $2^{\mathcal{O}(w\log w)}n^{\mathcal{O}(1)}$ that computes a \hyperref[def:pmw]{perfect matching decomposition} of width at most $c_{\operatorname{pmw}}\cdot w^2$ for $B$.
\end{theorem}

Towards the first question we show that perfect matching width can be used for recognising the existence of a fixed matching minor within a bipartite graph.

\begin{theorem}\label{thm:matchingminors}
	Let $H$ be a fixed bipartite \hyperref[def:matchingcovered]{matching covered} graph and $B$ a bipartite graph with a perfect matching.
	There exists an algorithm with running time $\Abs{\V{B}}^{\Fkt{\mathcal{O}}{\Abs{\V{H}}^2+\pmw{B}^2}}$ that decides whether $B$ contains $H$ as a matching minor.
\end{theorem}

The algorithm  from \cref{thm:matchingminors} is achieved by solving a more general problem on bipartite graphs of bounded perfect matching width, namely a matching version of the so-called $t$-Linkage Problem, or $t$-Disjoint Paths Problem. 

Additionally, counting the number of perfect matchings on bipartite graphs of bounded perfect matching width can also be solved efficiently.

\begin{theorem}\label{thm:countmatchings}
	Let $B$ be a bipartite graph with a perfect matching.
	There exists an algorithm with running time $\Abs{\V{B}}^{\Fkt{\mathcal{O}}{\pmw{B}^2}}$ that computes the number of perfect matchings in $B$.
\end{theorem}

By combining \cref{thm:boundedclasses} and \cref{thm:matchinggridminors} with these algorithms, we obtain the following results for classes of bipartite graphs with perfect matchings that exclude a planar and matching covered matching minor.

\begin{corollary}\label{cor:planarmatchingminors}
Let $H$ be a fixed planar and bipartite \hyperref[def:matchingcovered]{matching covered} graph and let $\omega_H$ be the number from \cref{thm:matchinggridminors}.
There exists an algorithm with running time $\Abs{\V{B}}^{\mathcal{O}(\omega_H^2)}$ that decides whether a given bipartite graph $B$ with a perfect matching contains $H$ as a \hyperref[def:matchingminor]{matching minor}.
\end{corollary}

\begin{corollary}\label{cor:countpmsexcludingplanar}
Let $H$ be a fixed planar and bipartite \hyperref[def:matchingcovered]{matching covered} graph, let $\omega_H$ be the number from \cref{thm:matchinggridminors}, and let $B$ be a bipartite graph with a perfect matching that does not contain $H$ as a \hyperref[def:matchingminor]{matching minor}.
There exists an algorithm with running time $\Abs{\V{B}}^{\mathcal{O}(\omega_H^2)}$ that computes the number of perfect matchings of $B$.
\end{corollary}

\subsection{A Comparison to the Directed Case}\label{subsec:digraphs}

Structural Digraph Theory, especially the study of \hyperref[def:butterflyminor]{butterfly minors}, and Bipartite Matching Theory are closely linked via a powerful construction that translates every digraph into a bipartite graph with a perfect matching and vice versa.
See \cref{fig:Mdirection} for an example.

\begin{definition}[$M$-Direction]\label{def:Mdirection}
	Let $B=\Brace{V_1\cup V_2, E}$ be a bipartite graph and let $M\in\Perf{G}$ be a perfect matching of $B$. 
	The \emph{$M$-direction} $\DirM{B}{M}$ of $B$ is defined as follows.
	\begin{enumerate}
		\item $\Fkt{V}{\DirM{G}{M}}\coloneqq M$ and
		
		\item $\Fkt{E}{\DirM{G}{M}}\coloneqq\left\{\Brace{e,f}\in\Choose{M}{2}\text{there is $g\in\E{B}$ such that $\emptyset\neq e\cap g\subseteq V_1$ and}\right.$\\
		\phantom{x} ~~~~~~~~~~~~~~ $\left.\phantom{\Choose{M}{2}}\text{$\emptyset\neq f\cap g\subseteq V_2$}\right\}$.	
	\end{enumerate}
\end{definition}

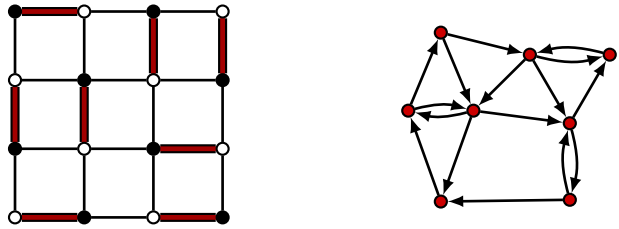
\begin{figure}[h!]
	\begin{center}
		\begin{tikzpicture}[scale=0.7]
			
			\pgfdeclarelayer{background}
			\pgfdeclarelayer{foreground}
			
			\pgfsetlayers{background,main,foreground}

			\begin{pgfonlayer}{main}
				
				%%%%% Centered Ghost Vertices %%%%%
				\node (C) [] {};
				
				%%%%% Left Center %%%%%
				\node (C1) [v:ghost, position=180:40mm from C] {};
				%\node (U1) [v:ghost, position=90:50mm from C1] {};
				%		\node (L1) [v:ghost, position=270:32mm from C1,align=center] {$R$};
				
				%%%%% Center %%%%%
				\node (C2) [v:ghost, position=0:0mm from C] {};
				%\node (U2) [v:ghost, position=90:50mm from C2] {};
				%		\node (L2) [v:ghost, position=270:27mm from C2,align=center] {};
				
				%%%%% Right Center %%%%%
				\node (C3) [v:ghost, position=0:40mm from C] {};
				%\node (U3) [v:ghost, position=90:50mm from C3] {};
				%		\node (L3) [v:ghost, position=270:32mm from C3,align=center] {$\Bidirected{C_5}$};
				%%%%% %%%%% %%%%%

				%%%%% Vertices %%%%%
				
				%%%%% Left Center %%%%%
				
				\node (a) [v:main,position=0:0mm from C1] {};
				\node (ain) [v:ghost,position=0:2.25mm from a] {};
				
				\node (b) [v:main,position=0:13mm from a,fill=white] {};
				\node (e) [v:main,position=270:13mm from a,fill=white] {};
				
				\node (c) [v:main,position=0:13mm from b] {};
				\node (cin) [v:ghost,position=270:2.25mm from c] {};
				\node (f) [v:main,position=270:13mm from b] {};
				\node (fin) [v:ghost,position=270:2.25mm from f] {};
				\node (i) [v:main,position=270:13mm from e] {};
				\node (iin) [v:ghost,position=90:2.25mm from i] {};
				
				\node (d) [v:main,position=0:13mm from c,fill=white] {};
				\node (g) [v:main,position=270:13mm from c,fill=white] {};
				\node (j) [v:main,position=0:13mm from i,fill=white] {};
				\node (m) [v:main,position=270:13mm from i,fill=white] {};
				
				\node (h) [v:main,position=0:13mm from g] {};
				\node (hin) [v:ghost,position=90:2.25mm from h] {};
				\node (k) [v:main,position=270:13mm from g] {};
				\node (kin) [v:ghost,position=0:2.25mm from k] {};
				\node (n) [v:main,position=270:13mm from j] {};
				\node (nin) [v:ghost,position=180:2.25mm from n] {};
				
				\node (l) [v:main,position=0:13mm from k,fill=white] {};
				\node (o) [v:main,position=270:13mm from k,fill=white] {};
				
				\node (p) [v:main,position=0:13mm from o] {};
				\node (pin) [v:ghost,position=180:2.25mm from p] {};
				
				%%%%% %%%%% %%%%%
				
				%%%%% Center %%%%%

				%%%%% %%%%% %%%%%
				
				%%%%% Right Center %%%%%
				
				\node (ab) [v:main,fill=BostonUniversityRed,position=270:4mm from C3] {};
				\node (ei) [v:main,fill=BostonUniversityRed,position=247.5:16mm from ab] {};
				\node (fj) [v:main,fill=BostonUniversityRed,position=292.5:16mm from ab] {};
				\node (mn) [v:main,fill=BostonUniversityRed,position=270:32mm from ab] {};
				\node (cg) [v:main,fill=BostonUniversityRed,position=45:15mm from fj] {};
				\node (dh) [v:main,fill=BostonUniversityRed,position=0:15mm from cg] {};
				\node (kl) [v:main,fill=BostonUniversityRed,position=300:15mm from cg] {};
				\node (op) [v:main,fill=BostonUniversityRed,position=270:14.5mm from kl] {};
				
				%%%%% %%%%% %%%%%
				
				%%%%% %%%%% %%%%%

				%%%%% Edges %%%%%
				
				%%%%% Left Center %%%%%
				
				\draw (b) [e:main] to (c);
				\draw (b) [e:main] to (f);
				
				\draw (e) [e:main] to (a);
				\draw (e) [e:main] to (f);
				
				\draw (d) [e:main] to (c);
				
				\draw (g) [e:main] to (f);
				\draw (g) [e:main] to (h);
				\draw (g) [e:main] to (k);
				
				\draw (j) [e:main] to (i);
				\draw (j) [e:main] to (k);
				\draw (j) [e:main] to (n);
				
				\draw (m) [e:main] to (i);
				
				\draw (l) [e:main] to (h);
				\draw (l) [e:main] to (p);
				
				\draw (o) [e:main] to (k);
				\draw (o) [e:main] to (n);

				%%%%% %%%%% %%%%%
				
				%%%%% Center %%%%%
				
				%%%%% %%%%% %%%%%
				
				%%%%% Right Center %%%%%
				
				\draw (ab) [e:main,->] to (fj);
				\draw (ab) [e:main,->] to (cg);
				
				\draw (ei) [e:main,->,bend left=15] to (fj);
				\draw (ei) [e:main,->] to (ab);
				
				\draw (fj) [e:main,->,bend left=15] to (ei);
				\draw (fj) [e:main,->] to (mn);
				\draw (fj) [e:main,->] to (kl);
				
				\draw (mn) [e:main,->] to (ei);
				
				\draw (cg) [e:main,->,bend right=15] to (dh);
				\draw (cg) [e:main,->] to (fj);
				\draw (cg) [e:main,->] to (kl);
				
				\draw (dh) [e:main,->,bend right=15] to (cg);
				
				\draw (kl) [e:main,->] to (dh);
				\draw (kl) [e:main,->,bend left=15] to (op);
				
				\draw (op) [e:main,->] to (mn);
				\draw (op) [e:main,->,bend left=15] to (kl);
				
				%%%%% %%%%% %%%%%
				
				%%%%% %%%%% %%%%%
				
			\end{pgfonlayer}
			
			%%%%% %%%%% %%%%%

			%%%%% Background %%%%%
			\begin{pgfonlayer}{background}
				
				\draw (b) [e:coloredborder] to (a);
				\draw (e) [e:coloredborder] to (i);
				\draw (d) [e:coloredborder] to (h);
				\draw (g) [e:coloredborder] to (c);
				\draw (j) [e:coloredborder] to (f);
				\draw (m) [e:coloredborder] to (n);
				\draw (l) [e:coloredborder] to (k);
				\draw (o) [e:coloredborder] to (p);
				
				\draw (b) [e:colored,color=BostonUniversityRed] to (a);
				\draw (e) [e:colored,color=BostonUniversityRed] to (i);
				\draw (d) [e:colored,color=BostonUniversityRed] to (h);
				\draw (g) [e:colored,color=BostonUniversityRed] to (c);
				\draw (j) [e:colored,color=BostonUniversityRed] to (f);
				\draw (m) [e:colored,color=BostonUniversityRed] to (n);
				\draw (l) [e:colored,color=BostonUniversityRed] to (k);
				\draw (o) [e:colored,color=BostonUniversityRed] to (p);
				
			\end{pgfonlayer}	
			%%%%% %%%%% %%%%%
			
			%%%%% Foreground %%%%%
			\begin{pgfonlayer}{foreground}

			\end{pgfonlayer}
			%%%%% %%%%% %%%%%
		\end{tikzpicture}
	\end{center}
	\caption{Left: A bipartite graph $B$ with a perfect matching $M$. Right: The arising $M$-direction $\DirM{B}{M}$.}
	\label{fig:Mdirection}
\end{figure}

Similarly, every undirected graph can be turned into a digraph by simply replacing every undirected edge $uv$ with the directed edges $\Brace{u,v}$ and $\Brace{v,u}$.

Every digraph $D$ can be made into an undirected graph by simply `forgetting' the orientation of the edges.
There is also a way to interpret any undirected graph as a digraph.

\begin{definition}[Biorientation]\label{def:biorientation}
	Let $G$ be a graph.
	The digraph
	\begin{align*}
		\Bidirected{G}\coloneqq\Brace{\V{G},\CondSet{\Brace{u,v},\Brace{v,u}}{uv\in\E{G}}}
	\end{align*}
	is called the \emph{biorientation} of $G$.
	A digraph $D$ for which a graph $G$ exists with $D=\Bidirected{G}$ is called a \emph{bioriented graph} or \emph{symmetric digraph}.
\end{definition}

Several properties of matching covered bipartite graphs naturally correspond to properties of digraphs.
In particular this is the case for strong connectivity, as one can easily observe that the $M$-alternating cycles of a bipartite graph $B$ with a perfect matching $M$ are in bijection with the directed cycles of its $M$-direction.
The following statement is folklore (a proof can be found in \cite{zhang2010bipartite}, but the result was already known by \cite{robertson1999permanents}).

\begin{theorem}\label{thm:exttoconn}
	Let $B$ be a bipartite graph with a perfect matching $M$ and $k\in\N$ be a positive integer.
	Then $B$ is \hyperref[def:extendibility]{$k$-extendable} if and only if $\DirM{B}{M}$ is strongly $k$-connected.
\end{theorem}

Even more important, and part of the main motivation behind this research is the following relation between butterfly minors and matching minors.

\begin{lemma}[\cite{mccuaig2000even}]\label{lemma:mcguigmatminors}
	Let $B$ and $H$ be bipartite matching covered graphs.
	Then $H$ is a \hyperref[def:matchingminor]{matching minor} of $B$ if and only if there exist perfect matchings $M\in\Perf{B}$ and $M'\in\Perf{H}$ such that $\DirM{H}{M'}$ is a \hyperref[def:butterflyminor]{butterfly minor} of $\DirM{G}{M}$.
\end{lemma}

Indeed, to the best of our knowledge, all exact characterisations of classes of digraphs in form of excluded families of butterfly minors are due to, or at least closely related to, \cref{lemma:mcguigmatminors} \cite{seymour1987characterization,guenin2011packing,wiederrecht2020digraphs}.

Building on the introduction of \hyperref[def:dtw]{directed treewidth} in \cite{johnson2001directed}, Kawarabayashi et al.\@ \cite{kawarabayashi2015directed,amiri2016erdos,giannopoulou2020directed} have started the project of extending the Graph Minors Theory of Robertson and Seymour to digraphs.
However, when dealing with excluding a planar butterfly minor, especially with extensions of \cite{robertson1986graph}, some problems arise.
A key observation is that, while the cylindrical grid\footnote{The cylindrical grid is the \hyperref[def:matchinggrid]{$M$-direction} of the \hyperref[def:matchinggrid]{cylindrical matching grid}, where $M$ is its canonical matching.} as guaranteed by the Directed Grid Theorem from \cite{kawarabayashi2015directed} is indeed a planar digraph, it does not contain every planar digraph as a butterfly minor.
For an example of such a digraph see \cref{fig:nongridminorplanardigraph}.
Because of this, the strongly connected digraphs which have the Erd\H{o}s-P\'osa property are exactly those which are butterfly minors of the cylindrical grid \cite{amiri2016erdos}.

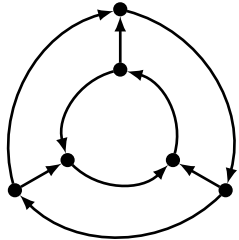
\begin{figure}[!h]
	\centering
	\begin{tikzpicture}
		\pgfdeclarelayer{background}
		\pgfdeclarelayer{foreground}
		\pgfsetlayers{background,main,foreground}
		
		\node (mid) [v:ghost] {};
		
		\node (v1) [v:main,position=90:8mm from mid] {};
		\node (v2) [v:main,position=210:8mm from mid] {};
		\node (v3) [v:main,position=330:8mm from mid] {};
		
		\node (v4) [v:main,position=90:16mm from mid] {};
		\node (v5) [v:main,position=210:16mm from mid] {};
		\node (v6) [v:main,position=330:16mm from mid] {};
		
		\draw [e:main,->,bend right=45] (v1) to (v2);
		\draw [e:main,->,bend right=45] (v2) to (v3);
		\draw [e:main,->,bend right=45] (v3) to (v1);
		
		\draw [e:main,->,bend left=45] (v4) to (v6);
		\draw [e:main,->,bend left=45] (v6) to (v5);
		\draw [e:main,->,bend left=45] (v5) to (v4);
		
		\draw [e:main,->] (v1) to (v4);
		\draw [e:main,->] (v5) to (v2);
		\draw [e:main,->] (v6) to (v3);
		
	\end{tikzpicture}
	\caption{A strongly connected and planar digraph that is \textbf{not} a butterfly minor of the cylindrical grid.}
	\label{fig:nongridminorplanardigraph}
\end{figure}

As pointed out in \cite{robertson1999permanents,mccuaig2004polya}, the class of digraphs excluding all $\Bidirected{C_n}$, where $n\geq 3$ is odd, corresponds exactly to the class of bipartite graphs without $K_{3,3}$ as a matching minor via the \hyperref[def:Mdirection]{$M$-direction} and \cref{lemma:mcguigmatminors}.
So excluding a single matching minor in a bipartite graph $B$ corresponds to excluding a, possible infinite, anti-chain of butterfly minors in the $M$-directions of $B$.

By using \cref{lemma:mcguigmatminors} we formalise these anti-chains based on a single bipartite graph with a perfect matching in \cref{subsec:antichains}.
This leads to a new way of handling butterfly minor closed classes of digraphs in terms of excluded anti-chains instead of individual digraphs.
By doing so we are able to give digraphic analogues of \cref{thm:boundedclasses,thm:matchinggridminors,thm:matchingEP} by replacing a single digraph with a whole anti-chain and planarity with a stronger version, more suited for the theory of butterfly minors.

\subsection{Organisation}\label{subsec:oragnisation}

In \cref{subsec:preliminaries} we present a collection of all necessary definitions.
The remainder of this article is organised as follows:
\begin{enumerate}
	\item \Cref{sec:separation} deals with finding a bounded size cover for all $M$-alternating cycles crossing over an edge cut of bounded \hyperref[def:matchingporosity]{matching porosity}.
	\item We then use these bounded size covers in \cref{sec:pmwanddtw} to give explicit bounds for the relation of the \hyperref[def:pmw]{perfect matching width} of a bipartite graph with a perfect matching and the \hyperref[def:dtw]{directed treewidth} of its \hyperref[def:Mdirection]{$M$-directions}, thereby substantially improving on the previous result from \cite{hatzel2019cyclewidth}.
	\item In \cref{sec:minors} we introduce a notion of models for matching minors based on previous definitions from \cite{norine2007generating} and formalise the link between the exclusion of a single matching minor and a whole anti-chain of butterfly minors.
	\item \Cref{sec:EP} is dedicated to the proofs of \cref{thm:boundedclasses,thm:matchinggridminors,thm:matchingEP} and their analogues for digraphs.
	\item Finally, in \cref{sec:algorithms} we describe the algorithms and present the necessary proofs towards \cref{thm:approximatepmw,thm:matchingminors,thm:countmatchings}.
\end{enumerate}

\subsection{Preliminaries}\label{subsec:preliminaries}

In the following we introduce basic terminology as well as all concepts necessary for the statements of our main results as presented above.
All graphs and digraphs in this article are considered simple, that is we do now allow multiple edges or loops and wherever such objects would arise from contraction, we identify multiple edges and remove loops.
For a deeper introduction to Matching Theory the reader may consult \cite{lovasz2009matching}, while for Digraph Theory we recommend \cite{bang2018classes}.

For integers $i,j\in\Z$ we use the notation $[i,j]$ for the set $\Set{i,i+1,\dots,j}$, where $[i,j]=\emptyset$ if $i>j$.

Since the majority of our research is focussed on bipartite graphs we fix the following convention.
Wherever possible we use $B$ as the standard name for a bipartite graph and $G$ for arbitrary graphs if not stated explicitly otherwise.
Moreover, we assume every bipartite graph to come with a bipartition into the \emph{colour classes} $V_1$ and $V_2$, where in our figures $V_1$ is represented by black vertices and the vertices in $V_2$ are depicted white.
In case ambiguity arises we either treat $V_i$ as the placeholder for all possible vertices of colour $i\in[1,2]$, or we write $\Vi{i}{B}$ to specify which graph we are talking about.

If $X$ and $Y$ are two finite sets, we denote the \emph{symmetric difference} by $X\Delta Y\coloneqq \Brace{X\setminus Y}\cup \Brace{Y\setminus X}$.

\paragraph{Matching Theory}

\begin{definition}[Perfect Matching and Matching Covered]\label{def:matchingcovered}
	Let $G$ be a graph.
	A \emph{matching} is a set $F\subseteq\E{G}$ of pairwise disjoint edges, by $\V{F}$ we denote the set $\bigcup_{e\in F}e$, and we say that a vertex $v\in\V{G}$ is \emph{covered} by $F$ if $v\in\V{F}$.
	A matching $M\subseteq\V{G}$ is \emph{perfect} if $\V{M}=\V{G}$, an edge $e\in\E{G}$ which is contained in a perfect matching of $G$ is called \emph{admissible}.
	We denote by $\Perf{G}$ the set of all perfect matchings of $G$.
	The graph $G$ is called \emph{matching covered} if it is connected, and every edge of $G$ is contained in a perfect matching.
\end{definition}

Since we are interested in graphs with perfect matchings, we need a restricted notion of subgraphs which preserves, at least in parts, the matching structure within our graphs.

\begin{definition}[Conformal Sets and Subgraphs]\label{def:conformal}
	Let $G$ be a graph with a perfect matching $M$.
	A set $X\subseteq\V{G}$ is called \emph{conformal} if $G-X$ has a perfect matching, it is \emph{$M$-conformal}, if $M$ contains a perfect matching of $G-X$.
	Similarly, a subgraph $H\subseteq G$ is conformal if $\V{H}$ is conformal\footnote{At first glance, this definition might seem unintuitive. However, for our applications we only ever consider conformal subgraphs $H$ which themselves also have a perfect matching. With $\V{H}$ being conformal that means we can combine any perfect matching of $G-H$ with any perfect matching of $H$ to obtain a perfect matching of $G$ as a whole. Hence it suffices to only require $G-H$ to have a perfect matching.}.
	Moreover, $H$ is \emph{$M$-conformal} if $\V{H}$ is $M$-conformal and $M$ contains a perfect matching of $H$.
\end{definition}

For a theory of connectivity and routing with respect to the perfect matchings within a graph we need specialised versions of paths and cycles.

\begin{definition}[Alternating Paths and Cycles]\label{def:alternating}
	Let $G$ be a graph and $F$ a matching in $G$.
	A path $P$ is said to be \emph{$F$-alternating}, if there exists a subset $S$ of the endpoints of $P$ such that $F$ contains a perfect matching of $P-S$.
	The path $P$ is \emph{alternating} if there is a maximum matching $M$ of $G$ such that $P$ is $M$-alternating.
	
	A cycle $C$ is said to be \emph{$F$-alternating} if $F$ contains a perfect matching of $C$.
	The cycle $C$ is said to be \emph{alternating} if $G$ has a maximum matching $M$ for which $C$ is $M$-alternating.
\end{definition}

A key notion, especially for bipartite graphs, in Matching Theory is the property of being able to extend small matchings to perfect matchings of the whole graph.
As illustrated by \cref{thm:exttoconn} this property can be seen as the matching theoretic version of strong connectivity.

\begin{definition}[Extendability]\label{def:extendibility}
	Let $G$ be a graph with a perfect matching and $F\subseteq\E{G}$ a matching.
	We say that $F$ is \emph{extendable} if there exists $M\in\Perf{G}$ such that $F\subseteq M$.
	
	For any positive integer $k\in\N$, $G$ is said to be \emph{$k$-extendable} if it is connected, has at least $2k+2$ vertices, and every matching of size $k$ in $G$ is extendable.
\end{definition}

We also need a specialisation of minors which also preserves the existence of perfect matchings.
In particular this means that every time we perform one atomic instance of contraction we must identify an odd number of vertices, since the total number of vertices must stay even.

Let $G$ be a graph with a perfect matching and $v\in\V{G}$ a vertex of degree two with $\NeighboursG{G}{v}=\Set{v_1,v_2}$.
Let 
\begin{align*}
	G'\coloneqq G-\Set{v_1,v,v_2}+u+\CondSet{uw}{w\in\NeighboursG{G-v}{v_1}\cup\NeighboursG{G-v}{v_2}},
\end{align*}
where $u\notin\V{G}$.
We say that $G'$ is obtained from $G$ by \emph{bicontracting $v$}.

\begin{definition}[Matching Minor]\label{def:matchingminor}
	Let $G$ and $H$ be graphs with perfect matchings and let $M$ be a perfect matching of $G$.
	We say that $H$ is a \emph{matching minor} of $G$ if $H$ can be obtained from a conformal subgraph of $G$ by a sequence of bicontractions.
	If $H$ can be obtained from an $M$-conformal subgraph of $G$ by a sequence of bicontractions we say that $H$ is an \emph{$M$-minor} of $G$.
\end{definition}

A class $\mathcal{B}$ of bipartite graphs is called a \emph{proper matching minor closed class} if every member of $\mathcal{B}$ has a perfect matching, for every $B\in\mathcal{B}$ if $H$ is a matching minor of $B$, then $H\in\mathcal{B}$, and $\mathcal{B}$ does not contain all bipartite graphs with perfect matchings.

\paragraph{Perfect Matching Width}

In this paragraph we introduce the matching theoretic version of treewidth our theory is based on: perfect matching width.

Let $G$ be a graph and $X\subseteq\V{G}$ a set of vertices.
The \emph{edge cut around $X$} in $G$ is defined as the set $\CutG{G}{X}\coloneqq\CondSet{uv\in\E{G}}{e\in X\text{ and }v\notin X}$.
The sets $X$ and $\Complement{X}\coloneqq\V{G}\setminus X$ are called the \emph{shores} of $\CutG{G}{X}$.

\begin{definition}[Matching Porosity]\label{def:matchingporosity}
The \emph{matching porosity} of an edge cut in a graph $G$ with a perfect matching is the value
\begin{align*}
	\MatPor{\CutG{G}{X}}\coloneqq\max_{M\in\Perf{G}}\Abs{M\cap\CutG{G}{X}}.
\end{align*}
\end{definition}

Let $T$ be a tree and $t_1t_2\in\E{T}$, then $T-t_1t_2$ consists of two components, each containing exactly one endpoint of $t_1t_2$. For $i\in\Set{1,2}$ we denote by $T_{t_i}$ the component of $T-t_1t_2$ containing $t_i$.
Moreover, if $T$ is rooted and $t\in\V{T}$ a vertex different from the root, then there exists a unique vertex $d\in\V{T}$ such that $dt$ is the edge on the path from $t$ to the root.
In this case, we denote by $T_t$ the component of $T-dt$ containing $t$. Finally, we denote by $\Leaves{T}$ the vertices of $T$ of degree one called the \emph{leaves} of $T$.

\begin{definition}[Perfect Matching Width]\label{def:pmw}
	Let $G$ be a graph with a perfect matching.
	A \emph{perfect matching decomposition} of $G$ is a tuple $\Brace{T,\delta}$, where $T$ is a cubic tree and $\delta\colon \Leaves{T} \to \Fkt{V}{G}$ a bijection.
	Let $t_1t_2$ be an edge in $T$, then the partition of the tree into $T_{t_1}$ and $T_{t_2}$ also yields a partition of the vertices in $G$ that are mapped to the leaves of $T$.
	Let
	\begin{align*}
		X_i &\coloneqq \bigcup_{t \in \Leaves{T} \cap \V{T_{t_i}}} \Set{\Fkt{\delta}{t}}
	\end{align*}
	be the two classes of the partition.
	Note that $\CutG{G}{X_1} = \CutG{G}{X_2}$ defines an edge cut in $G$, we refer to it by $\CutG{G}{t_1t_2}$.
	The width of a perfect matching decomposition $\Brace{T,\delta}$ is given by $\Width{T,\delta}\coloneqq\max_{t_1t_2 \in \E{T}} \MatPor{\CutG{G}{t_1t_2}}$ and the \emph{perfect matching width} of $G$ is then defined as 
	\begin{align*}
		\pmw{G} \coloneqq \min_{\substack{\Brace{T,\delta} \text{ perfect matching}\\\text{decomposition of } G}} \quad \max_{t_1t_2 \in \E{T}} \MatPor{\CutG{G}{t_1t_2}}.
	\end{align*}
\end{definition}

A class $\mathcal{B}$ of bipartite graphs with perfect matchings is said to be of bounded \emph{perfect matching width} if there exists a constant $c_{\mathcal{B}}\in\N$ such that $\pmw{B}\leq c_{\mathcal{B}}$ for all $B\in\mathcal{B}$.

\begin{definition}[Cylindrical Matching Grid]\label{def:matchinggrid}
	The \emph{cylindrical matching grid} $CG_k$ of order $k$ is defined as follows.
	Let $C_1,\dots,C_k$ be $k$ vertex disjoint cycles of length $4k$.
	For every $i\in[1,k]$ let $C_i=\Brace{v_1^i,v_2^i,\dots,v_{4k}^i}$, $V_1^i\coloneqq\CondSet{v_j^i}{j\in\Set{1,3,5,\dots,4k-1}}$, $V_2^i\coloneqq\Fkt{V}{C_i}\setminus V_1^i$, and $M_i\coloneqq\CondSet{v_j^iv_{j+1}^i}{v_j^i\in V_1^i}$.
	Then $CG_k$ is the graph obtained from the union of the $C_i$ by adding
	\begin{align*}
		\CondSet{v_j^iv_{j+1}^{i+1}}{i\in[1,k-1]~\text{and}~j\in\Set{1,5,9,\dots,4k-3}}&\text{, and}\\
		\CondSet{v_j^iv_{j+1}^{i-1}}{i\in[2,k]~\text{and}~j\in\Set{3,7,11,\dots,4k-1}}&
	\end{align*}
	to the edge set.
	We call $M\coloneqq\bigcup_{i=1}^kM_i$ the \emph{canonical matching} of $CG_k$.
	See \cref{fig:cylindricalgrid} for an illustration.
\end{definition}

\paragraph{Digraphs}

\begin{definition}[Butterfly Minor]\label{def:butterflyminor}
	Let $D$ be a digraph and $\Brace{u,v}\in\E{D}$.
	The edge $\Brace{u,v}$ is \emph{butterfly contractible} if $\OutNeighbours{D}{u}=\Set{v}$, or $\InNeighbours{D}{v}=\Set{u}$.
	
	Suppose $\Brace{u,v}$ is butterfly contractible and let
	\begin{align*}
		D'\coloneqq D-u-v+x&+\CondSet{\Brace{w,x}}{\Brace{w,u}\in\E{D}\text{ or }\Brace{w,v}\in\E{D}}\\&+\CondSet{\Brace{x,w}}{\Brace{u,w}\in\E{D}\text{ or }\Brace{v,w}\in\E{D}},
	\end{align*}
	where $x\notin\V{D}$.
	We say that $D'$ is obtained from $D$ by \emph{butterfly contraction} of $\Brace{u,v}$.
	
	A digraph $H$ is a \emph{butterfly minor} of $D$ if it can be obtained from $D$ by a sequence of edge-deletions, vertex-deletions, and butterfly contractions.
\end{definition}

Let $D$ be a digraph and $X\subseteq\V{D}$.
A directed walk $W$ is a \emph{directed $X$-walk} if it starts and ends in $X$, and contains a vertex of $\V{D-X}$.
We say that $Y\subseteq\V{D}$ \emph{strongly guards} $X$ if every directed $X$-walk in $D$ contains a vertex of $Y$.

An \emph{arborescence} is a digraph $\vec{T}$ obtained from a tree $T$ by selecting a \emph{root} $r\in\V{T}$ and orienting all edges of $T$ away from $r$.
If $e$ is a directed edge and $v$ is an endpoint of $e$ we write $v\sim e$.

\begin{definition}[Directed Treewidth]\label{def:dtw}
	Let $D$ be a digraph.
	A \emph{directed tree decomposition} for $D$ is a tuple $\Brace{T,\beta,\gamma}$ where $T$ is an arborescence, $\beta\colon\Fkt{V}{T}\rightarrow 2^{V(D)}$ is a function that partitions $\Fkt{V}{D}$, into sets called the \emph{bags}\footnote{This means $\CondSet{\Fkt{\beta}{t}}{t\in\V{T}}$ is a partition of $\V{D}$ into non-empty sets.}, and $\gamma\colon\Fkt{E}{T}\rightarrow 2^{V(D)}$ is a function, giving us sets called the \emph{guards}, satisfying the following requirement:
	\begin{enumerate}
		\item[] For every $\Brace{d,t}\in\Fkt{E}{T}$, $\Fkt{\gamma}{d,t}$ strongly guards $\Fkt{\beta}{T_t}\coloneqq\bigcup_{t'\in V(T_t)}\Fkt{\beta}{t'}$.
	\end{enumerate}
	Here $T_t$ denotes the subarboresence of $T$ with root $t$.
	For every $t\in\Fkt{V}{T}$ let $\Fkt{\Gamma}{t}\coloneqq\Fkt{\beta}{t}\cup\bigcup_{t\sim e}\Fkt{\gamma}{e}$.
	The \emph{width} of $\Brace{T,\beta,\gamma}$ is defined as
	\begin{align*}
		\Width{T,\beta,\gamma}\coloneqq\max_{t\in V(T)}\Abs{\Fkt{\Gamma}{t}}-1.
	\end{align*}
	The \emph{directed treewidth} of $D$, denoted by $\dtw{D}$, is the minimum width over all directed tree decompositions for $D$.
\end{definition}

\section{Matching Porosity, Directed Cycles, and Separation}\label{sec:separation}

Inspecting the guard sets of a \hyperref[def:dtw]{directed tree decomposition} more closely reveals that guards are not supposed to block all directed paths that go in one direction.
Instead, the guards are meant to make sure that no strong component of $D$ avoids the guards and contains vertices from below and above the guarded edge in the directed tree decomposition.
Indeed, that means that the guard set of any given edge $e$ in a directed tree decomposition is a hitting set for all directed cycles that contain vertices from the bags below the head of $e$, but are not fully contained in their union.
Since, through the eyes of the $M$-direction, the vertices of a digraph are in fact the edges of a perfect matching in a bipartite graph we are interested in a similar property for subsets of perfect matchings on cuts of bounded matching porosity.
Explicitly, given a cut $\CutG{B}{X}$ in a bipartite graph $B$ with a perfect matching $M$ we are interested in a set $F\subseteq M$ that meets all $M$-conformal cycles with vertices in $X$ and $\Complement{X}$, and whose size is upper bounded by some function of $\MatPor{\CutG{B}{X}}$.

\begin{definition}[Guarding Set]\label{def:guardingset}
	Let $B$ be a bipartite graph with a perfect matching $M$, and $X\subseteq\V{B}$.
	An \hyperref[def:conformal]{$M$-conformal} cycle $C$ is said to \emph{cross} the cut $\CutG{B}{X}$ if $\E{C}\cap\CutG{B}{X}\neq\emptyset$.
	A set $F\subseteq M$ is a \emph{hitting set} for a family $\mathcal{C}$ of $M$-conformal cycles if $F\cap\E{C}\neq\emptyset$ for all $C\in\mathcal{C}$.
	Moreover, $F$ is called a \emph{guard} for $\CutG{B}{X}$ if $\CutG{B}{X}\cap M\subseteq F$ and $F$ is a hitting set for the family of all $M$-conformal cycles that cross $\CutG{B}{X}$.	
\end{definition}

This section is dedicated to the proof of the following theorem and its consequences for digraphs.

\begin{theorem}\label{lemma:guardingsseps}
	Let $B$ be a bipartite graph with a perfect matching $M$, and $X\subseteq\V{B}$ a set of vertices with $\MatPor{\CutG{B}{X}}=k$, then there exists a guarding set $F\subseteq M$ of $\CutG{B}{X}$ with $\Abs{F}\leq 2k+k^2$.
\end{theorem}

Let $G$ be a graph with a perfect matching.
We call the graph induced by the \hyperref[def:matchingcovered]{admissible} edges of $G$ the \emph{cover graph} of $G$.
A component of the cover graph of $G$ is an \emph{elementary component} of $G$ and the set of all elementary components of $G$ is denoted by $\Elementary{G}$.
With respect to elementarity, guarding sets for cuts can be seen as a matching analogue of strong separators.

\begin{definition}[Dulmage-Mendelsohn Decomposition]\label{def:DMordering}
	Let $B$ be a bipartite graph with a perfect matching, and $i\in[1,2]$.
	
	For any two elementary components $K^1,K^2\in\Elementary{B}$ we set $K^1\leq^{\circ}_i K^2$ if $K^1=K^2$, or there exists an edge with one endpoint in $V_i\cap\E{K^2}$ and the other one in $\E{K^1}\setminus V_i$.
	
	We then write $K^1\leq_i K^2$ for any two elementary components of $B$ if there exist $H_1,\dots,H_k\in\Elementary{B}$, $k\geq 1$, such that $H_1=K^1$, $H_k=K^2$, and $H_j\leq_i^{\circ}H_{j+1}$ for all $j\in[1,k-1]$.
\end{definition}

In particular, this means $K^1\leq_1 K^2$ if and only if $K^2\leq_2 K^1$.
This relation, in a way, resembles the topological ordering of strong components of digraphs, as in every step we leave a component going from $V_i$ to $V_j$, inside the component we pass over to $V_i$ again and may now move to the next component.

\begin{theorem}[\cite{dulmage1958coverings,dulmage1959structure,dulmage1963two}]\label{thm:DMordering}
	Let $B$ be a bipartite graph with a perfect matching.
	Then for any $i\in[1,2]$, the binary relation $\leq_i$ is a partial order over $\Elementary{B}$.
\end{theorem}

\begin{lemma}\label{lemma:matchingseparation}
	Let $B$ be a bipartite graph with a perfect matching $M$, and $X\subseteq\V{B}$.
	If $F\subseteq M$ is a \hyperref[def:guardingset]{guard} for $\CutG{B}{X}$, then no elementary component of $B-\V{F}$ can contain vertices of both $X\setminus\V{F}$ and of $\Complement{X}\setminus\V{F}$.
\end{lemma}

\begin{proof}
	Suppose there exists an elementary component $K$ with vertices in both $X\setminus\V{F}$ and $\Complement{X}\setminus\V{F}$.
	Then $K$ must have at least four vertices since otherwise, $K$ would be isomorphic to $K_2$ and its single edge would have to be an edge of $M\cap\CutG{B}{X}\subseteq F$.
	Now let $e\in\E{K}\cap\CutG{B}{X}$ be an edge of $K$ (and observe that, by definition, $e\notin M$). Also, let $M'$ be a perfect matching of $K$ containing $e$.
	Since $F\subseteq M$ and $K$ is an elementary component of $B-\V{F}$, $M_K\coloneqq M\cap\E{K}$ is a perfect matching of $K$.
	Moreover, $e\notin M_{K}$.
	If we consider the subgraph $K'$ of $K$ consisting solely of the edges of $M'$ and $M_K$, every component either is an $M$-conformal cycle, or isomorphic to $K_2$.
	Since $e\notin M_K$, the endpoints of $e$ are covered by distinct edges of $M$, and thus the component of $K'$ containing $e$ must be an $M$-conformal cycle that crosses $\CutG{B}{X}$ and avoids $F$.
	However, such a cycle cannot exist by definition.
\end{proof}

An important observation one can make in the proof of the lemma above is that, if one were to delete the vertices of a set $F\subseteq M$ and there still is an elementary component with vertices on both sides of $\CutG{B}{X}$, there exists a cycle $C$ that is $M$-conformal in $B$, which still has edges in $\CutG{B}{X}$.
That means, if $F=\CutG{B}{X}\cap M$, then there exists the perfect matching $M'\coloneqq M\Delta\E{C}$ with $\Abs{\CutG{B}{X}\cap M'}\geq\Abs{\CutG{B}{X}\cap M}+2$.
We make this observation more exact in the following lemma.

\begin{lemma}\label{lemma:maximisingguards}
	Let $B$ be a bipartite graph with a perfect matching $M$, $X\subseteq\V{B}$ such that $\MatPor{\CutG{B}{X}}=k$, and $\Abs{M\cap\CutG{B}{X}}=k$.
	Then $M\cap\CutG{B}{X}$ is a \hyperref[def:guardingset]{guard} for $\CutG{B}{X}$.
\end{lemma}

\begin{proof}
	Suppose there is an $M$-conformal cycle $C$ avoiding $F\coloneqq M\cap\CutG{B}{X}$ but crossing $\CutG{B}{X}$.
	Let $W\coloneqq \E{C}\cap\CutG{B}{X}$, then clearly $W\cap M=\emptyset$.
	However, since $C$ avoids $F$ we can define a new perfect matching $M'$ of $B$ as $M'\coloneqq M\Delta\E{C}$.
	Then $W\cup F\subseteq M'$ and thus $\Abs{M'\cap\CutG{B}{X}}>k=\MatPor{\CutG{B}{X}}$ which is a contradiction.
	Hence no $M$-alternating cycle that crosses $\CutG{B}{X}$ can avoid $F$.
\end{proof}

For the next part, we need some additional notation.
Let $B$ be a bipartite graph with a perfect matching $M$, $X\subseteq\V{B}$ be an $M$-conformal set of vertices and $M'\neq M$ another perfect matching of $B$.
Then let us denote for every $W\subseteq M'$ by $\MCover{M'}{M}{W}$ the set of edges of $M$ that match the vertices in $\V{W}$.
Note that $\Abs{\MCover{M'}{M}{W}}\leq2\Abs{W}$.
Let $H$ be an elementary component of $B-W$, the \emph{M-box} of $H$ is the set $\HBox{H}{M}\coloneqq\V{H}\setminus\V{\MCover{M'}{M}{W}}$.

The following observation is an immediate consequence of the definition of $\leq_2$, \cref{thm:DMordering}, and $\V{W}\subseteq\V{\MCover{M'}{M}{W}}$.

\begin{observation}\label{obs:boxespreserveDM}
	Let $B$ be a bipartite graph with perfect matchings $M$ and $M'$, let $W\subseteq M'$ and $H_1$ and $H_2$ be two distinct elementary components of $B-\V{W}$.
	Then, if $H_1\leq_2 H_2$, there is no internally $M$-conformal path $P$ in $B-\V{\MCover{M'}{M}{W}}$ such that $P$ starts in a vertex of $V_2\cap\HBox{H_1}{M}$, ends in a vertex of $V_1\cap\HBox{H_2}{M}$, and is otherwise disjoint from $\HBox{H_1}{M}\cup\HBox{H_2}{M}$.
\end{observation}

We fix the following for the upcoming lemmata.

Let $B$ be a bipartite graph with a perfect matching $M$.
Let $X\subseteq\V{B}$ be an $M$-conformal set and let $M'$ be a perfect matching with $\Abs{\CutG{B}{X}\cap M'}=\MatPor{\CutG{B}{X}}=k$ as well as $W\coloneqq \CutG{B}{X}\cap M'$.
Let $\lambda$ be a linearisation of the partial order $\leq_2$ of elementary components of $B-\V{W}$ and let us number the elementary components $H_1,\dots,H_{\ell}$ of $B-\V{W}$ such that $\Fkt{\lambda}{H_i}=i$ for all $i\in[1,\ell]$.

A set $I\subseteq[1,\ell]$ is a \emph{dangerous configuration} if there exists an $M$-conformal cycle $C$ of $B-\V{\MCover{M'}{M}{W}}$ such that
\begin{enumerate}
	\item $\V{C}\subseteq \bigcup_{i\in I}\HBox{H_i}{M}$,
	
	\item $\HBox{H_i}{M}\cap\V{C}\neq\emptyset$ for all $i\in I$, and
	
	\item there are $i,j\in I$ such that $\HBox{H_i}{M}
	\subseteq X$ and $\HBox{H_j}{M}\subseteq\Complement{X}$.
\end{enumerate}
If $I$ is a dangerous configuration, we call $i_I\coloneqq \max I$ the \emph{endpoint} of $I$, the cycle $C$ is a \emph{base cycle} of $I$.

\begin{lemma}\label{lemma:DCexistence}
	Let $B'\coloneqq B-\V{\MCover{M'}{M}{W}}$.
	There exists an $M$-conformal cycle $C$ in $B'$ that crosses $\CutG{B'}{X\setminus\V{\MCover{M'}{M}{W}}}$ if and only if there exists a dangerous configuration $I$ with $C$ as a base cycle.
\end{lemma}

\begin{proof}
	The reverse direction follows immediately from the definition of dangerous configurations.
	If $I$ is dangerous with base cycle $C$, then $C$ contains vertices of both $X$ and $\Complement{X}$ and thus crosses $\CutG{B'}{X\setminus\V{\MCover{M'}{M}{W}}}$.
	Hence it suffices to prove the forward direction.
	So let $C$ be an $M$-conformal cycle in $B'$ that crosses $\CutG{B'}{X\setminus\V{\MCover{M'}{M}{W}}}$.
	Now let $I\coloneqq\CondSet{i\in[1,\ell]}{\HBox{H_i}{M}\cap\V{C}\neq\emptyset}$.
	Then $\V{C}\subseteq\bigcup_{i\in I}\HBox{H_i}{M}$ and clearly, the second requirement is met by the definition of $I$.
	At last we know that $C$ crosses $\CutG{B'}{X\setminus\V{\MCover{M'}{M}{W}}}$.
	By \cref{lemma:matchingseparation,lemma:maximisingguards} there cannot exist $j\in[1,\ell]$ such that $\V{H_j}\cap X\neq\emptyset$ and $\V{H_j}\cap\Complement{X}\neq\emptyset$ at the same time.
	Hence there must be $i,j\in I$ such that $\HBox{H_i}{M}\subseteq\V{H_i}\subseteq X$ and $\HBox{H_j}{M}\subseteq\V{H_j}\subseteq\Complement{X}$.
\end{proof}

In the fixed setting we are working on, let $i\in[1,\ell-1]$ be any number.
We associate a specific edge cut in $B$ with $i$ and $\lambda$ as follows:
\begin{align*}
	\CutG{\lambda}{H_i}\coloneqq\CutG{B}{\bigcup_{j=1}^i\V{H_j}\cup\Vi{1}{W}}.
\end{align*}

\begin{lemma}\label{lemma:lambdacuts}
	For all $i\in[1,\ell-1]$ and all perfect matchings $M''$ of $B$, we have $\Abs{\CutG{\lambda}{H_i}\cap M''}=k$.	
\end{lemma}

\begin{proof}
	Let $i\in[1,\ell-1]$ be arbitrary.
	By definition of the partial order $\leq_2$ of the $H_j$ no vertex $v\in \bigcup_{j=1}^i\Vi{2}{H_j}$ can have a neighbour in $\bigcup_{j=i+1}^{\ell}\V{H_j}$.
	So the only neighbours $v$ can have outside of $\bigcup_{j=1}^i\V{H_j}$ must be vertices of $\Vi{1}{W}$.
	Hence every perfect matching $M''$ of $B$ must match $v$ to a vertex within $\bigcup_{j=1}^i\V{H_j}\cup\Vi{1}{W}$ and thus it must have exactly 
	\begin{align*}
		\Abs{\bigcup_{j=1}^i\Vi{1}{H_j}\cup\Vi{1}{W}}-\Abs{\bigcup_{j=1}^i\Vi{2}{H_j}}=\Abs{\Vi{1}{W}}=k
	\end{align*}
	many edges in $\CutG{\lambda}{H_i}$.
\end{proof}

\begin{lemma}\label{lemma:meetingDC}
	Let $I\subseteq[1,\ell]$ be a dangerous configuration and let $C$ be a base cycle of $I$.
	Moreover, let $i\coloneqq\min I$ and $i\leq j< i_I$, then $\CutG{\lambda}{H_j}\cap\E{C}\cap M\neq\emptyset$.
\end{lemma}

\begin{proof}
	Clearly, with $C$ being a cycle, $\Abs{\E{C}\cap\CutG{\lambda}{H_j}}\geq 2$.
	Now suppose $\CutG{\lambda}{H_j}\cap\E{C}\cap M=\emptyset$ and let $e_1,e_2\in\E{C}\cap\CutG{\lambda}{H_j}$ be two distinct edges with $e_p=u_pv_p$ such that $v_p\in\bigcup_{j'=1}^j\V{H_j'}\cup\Vi{1}{W}$ for each $p\in[1,2]$.
	Let us further choose $e_1$ and $e_2$ such that there is a subpath $P$ of $C$ from $v_1$ to $v_2$ that avoids $u_1$ and $u_2$ and does not contain an edge of $\CutG{\lambda}{H_j}$.
	To find $P$ move along $C$ starting in $v_1$ and away from $u_1$ until the first time we reach an endpoint of another edge in $\CutG{\lambda}{H_j}$, this will be $v_2$.
	By our assumption $\Set{e_1,e_2}\cap M=\emptyset$ and thus, with $C$ being $M$-conformal, $\E{P}\cap M$ must be a perfect matching of $P$.
	Hence $P$ must have an even number of vertices in particular and thus is of odd length.
	But by \cref{obs:boxespreserveDM} $v_1,v_2\in V_1$ and thus $P$ is an odd length path joining two vertices of $V_1$.
	With $B$ being bipartite this is impossible and our claim follows.
\end{proof}

We are finally ready to prove the main result of this section.

\begin{proof}[Proof of \Cref{lemma:guardingsseps}]
	First let $F_{-1}\coloneqq M\cap\CutG{B}{X}$, $B_0\coloneqq B-\V{F_{-1}}$, $X_0\coloneqq X\setminus\V{F_{-1}}$, and $M_0\coloneqq M\setminus F_{-1}$.
	Clearly, every $M_0$ conformal cycle in $B_0$ is also an $M$-conformal cycle in $B$ that avoids $F_{-1}$.
	Each such cycle that crosses $\CutG{B_0}{X_0}$ also crosses $\CutG{B}{X}$, and every conformal set in $B_0$ is also conformal in $B$.
	Moreover, $\MatPor{\CutG{B_0}{X_0}}=k_0\coloneqq k-\Abs{F_{-1}}$ and $X_0$ is $M_0$-conformal.
	Now let $M'$ be a perfect matching of $B_0$ with $\Abs{M'\cap\CutG{B_0}{X_0}}=k_0$ and let $F_0\coloneqq\MCover{M'}{M_0}{\CutG{B_0}{X_0}\cap M'}$.
	Then $\Abs{F_0}\leq 2k_0$ since every edge of $F_0$ covers an endpoint of an edge in $\CutG{B_0}{X_0}\cap M'$ and there are $2k_0$ such endpoints.
	
	Let $\lambda$ be a linearisation of the partial order $\leq_2$ of the elementary components of $B_0-\V{\CutG{B_0}{X_0}\cap M'}$.
	Let us number the elementary components of $B_0-\V{\CutG{B_0}{X_0}\cap M'}$ $H_1,\dots,H_{\ell}$ such that $\Fkt{\lambda}{H_i}=i$ for all elementary components.
	We can now choose $i_1\in[1,\ell]$ to be the smallest number such that there is a dangerous configuration $I_1$ with $i_1=i_{I_1}$.
	Then for every dangerous configuration $I$ with the smallest element $i\leq i_1$ we must have $i_I\geq i_1$.
	Hence each base cycle of such a configuration must have an edge in $F_1\coloneqq\Brace{\CutG{\lambda}{H_{i_1-1}}\cap M_0}\cup\Brace{\CutG{\lambda}{H_{i_1}}\cap M_0}$ by \cref{lemma:meetingDC}.
	Indeed, every $M_0$-conformal cycle that crosses $\CutG{B_0}{X_0}$, avoids $F_0$, and has vertices in $\bigcup_{j=1}^{i_1}\HBox{H_j}{M_0}$ is met by $F_1$ by \cref{lemma:DCexistence}.
	Moreover, by \cref{lemma:lambdacuts} we have $\Abs{F_1}\leq 2k_0$.
	
	Now suppose the sets $F_1,\dots,F_{p-1}\subseteq M_0$ with $\Abs{F_j}\leq 2k_0$ for all $j\in[1,p-1]$ have already been constructed together with pairwise disjoint dangerous configurations $I_1,\dots,I_{p-1}$.
	Moreover let us assume that $1\leq j<j'\leq p-1$ implies $i_{I_j}<h$ where $h$ is the smallest member of $I_{j'}$ and $\bigcup_{j=1}^{p-1}F_j$ meets all base cycles of dangerous configurations $I$ for which some $i''\in I$ exists with  $i''< i_{I_{p-1}}$.
	Let $i_p\in[i_{p-1}+1,\ell]$ be the smallest number such that there is a dangerous configuration $I_p$ with base cycle $C$ that avoids $\bigcup_{j=1}^{p-1}F_j$.
	This means $I_p$ must be disjoint from $\bigcup_{j=1}^{p-1}I_j$.
	Let $F_p\coloneqq\Brace{\CutG{\lambda}{H_{i_p-1}}\cap M_0}\cup\Brace{\CutG{\lambda}{H_{i_p}}\cap M_0}$.
	By \cref{lemma:meetingDC,lemma:DCexistence}, $\bigcup_{j=1}^{p}F_j$ meets all $\CutG{B_0}{X_0}$ crossing $M_0$-conformal cycles that avoid $F_0$ and have a vertex in $\bigcup_{j=1}^{i_p}\HBox{H_j}{M_0}$.
	With \cref{lemma:lambdacuts} we also have $\Abs{F_p}\leq 2k_0$.
	
	With $B$ being finite and thus $\ell$ being a natural number there must be some $q$ such that we cannot find an $i_{q+1}$ as above.
	Suppose $q>\frac{k_0}{2}$.
	Clearly every $I_j$, $j\in[1,q]$ has a base cycle $C_j$ that is $M_0$ conformal and crosses $\CutG{B_0}{X_0}$.
	However, with $X_0$ being $M_0$-conformal, $C_j$ must have at least two edges in $\CutG{B_0}{X_0}$ that do not belong to $M_0$.
	Since $I_1,\dots,I_q$ are pairwise disjoint, also the $C_1,\dots,C_q$ are also pairwise disjoint.
	So we construct the following perfect matching of $B_0$:
	\begin{align*}
		M''\coloneqq M_0\Delta\bigcup_{j=1}^q\E{C_j}.
	\end{align*}
	Then $\Abs{\CutG{B_0}{X_0}\cap M''}\geq \sum_{j=1}^q\Abs{\CutG{B_0}{X_0}\cap\E{C_j}\setminus M_0}\geq 2q>2\frac{k_0}{2}=\MatPor{\CutG{B_0}{X_0}}$ which is impossible.
	Hence our process must stop after $q\leq\Floor{\frac{k_0}{2}}$ many steps.
	In total we get a set $F\coloneqq F_{-1}\cup F_0\cup\bigcup_{j=1}^q F_j$ that meets all $M$-conformal cycles crossing $\CutG{B}{X}$ and satisfying $\CutG{B}{X}\cap M\subseteq F$.
	So $F$ is a guard of $\CutG{B}{X}$.
	Moreover, we have
	\begin{align*}
		\Abs{F} &\leq \Abs{F_{-1}}+\Abs{F_0}+\sum_{j=1}^q\Abs{F_j}\\
		&\leq k-k_0+2k_0+\frac{k_0}{2}2k_0\\
		&\leq k+k_0+k_0^2\\
		& \leq 2k+k^2.
	\end{align*}
\end{proof}

\subsection{Cycle Porosity and Strong Separators}

How can we translate results on matching porosity into the setting of digraphs using our notion of $M$-directions?
To answer this question let us consider some graph $G$ with a perfect matching $M$.
Note that $G$ is not necessarily bipartite as this general idea can be applied to any graph with a perfect matching.
Now let $X\subseteq\V{G}$ be an $M$-conformal set with $\MatPor{\CutG{G}{X}}=k$ for some $k\in\N$.
Since $X$ is $M$-conformal, $\CutG{G}{X}\cap M=\emptyset$.
Let $M'\in\Perf{G}$ be a perfect matching with $\Abs{M'\cap\CutG{G}{X}}=k$ and consider the graph $G'\coloneqq\InducedSubgraph{G}{M'\cup M}$ that only consists of edges from $M'\cup M$.
That is, $G'$ is the subgraph of $G$ \emph{induced} by the edge set $M'\cap M$, so its vertex set is $\V{M'\cup M}$ and its edge set is $M'\cup M$.
Observe that, since $M$ and $M'$ are perfect matchings, every component of $G'$ either is an $M$-$M'$-conformal cycle or isomorphic to $K_2$.
Moreover, no edge of $M'\cap\CutG{G}{X}$ can belong to a component isomorphic to $K_2$ in $G'$ and thus each of these edges must be contained in an $M$-$M'$-conformal cycle.
Let $\mathcal{C}$ be the collection of all components of $G'$ that have an edge in $\CutG{G}{X}$.
Then we have $\bigcup_{C\in\mathcal{C}}\E{C}\cap \CutG{G}{X}\subseteq M'$ and in particular $\Abs{\bigcup_{C\in\mathcal{C}}\E{C}\cap \CutG{G}{X}}=k$.
Hence if the matching porosity of $\CutG{G}{X}$ is $k$ we find a family of pairwise disjoint $M$-conformal cycles in $G$ that share $k$ edges with $\CutG{G}{X}$ in total.
Now let $Y\subseteq\V{G}$ be another $M$-conformal set, and let $\mathcal{C}$ be a family of pairwise disjoint $M$-conformal cycles in $G$.
Suppose $\Abs{\bigcup_{C\in\mathcal{C}}\E{C}\cap \CutG{G}{Y}}=k'$ for some $k'\in\N$.
Let us denote by $\E{\mathcal{C}}$ the set $\bigcup_{C\in\mathcal{C}}\E{C}$, and let $M''\coloneqq M\Delta\E{\mathcal{C}}$.
Then $M''$ is a perfect matching of $G$ and $\Abs{\CutG{G}{Y}\cap M''}=k'$ which implies $\MatPor{\CutG{G}{Y}}\geq k'$.
So if we can find a family of pairwise disjoint $M$-conformal cycles in $G$, then the number of edges this family has in our cut is a lower bound on its matching porosity.
Suppose $G$ is bipartite, then there is a bijection between the $M$-conformal cycles in $G$ and the directed cycles in $\DirM{G}{M}$.
This leads us to the following definitions and observation.

\begin{definition}[Cycle Porosity]\label{def:cycleporosity}
	Let $D$ be a digraph, and $X\subseteq\V{D}$.
	The \emph{cycle porosity} of the cut $\CutG{D}{X}$ is defined as
	\begin{align*}
		\CycPor{\CutG{D}{X}}\coloneqq \max_{\substack{\mathcal{C}\text{ family of}\\\text{pairwise disjoint}\\\text{directed cycles}}}\Abs{\E{\mathcal{C}}\cap\CutG{D}{X}}.
	\end{align*}
\end{definition}

Let $G$ be a graph with a perfect matching $M$, and $X\subseteq\V{G}$ be an $M$-conformal set.
We denote by $\Fkt{M}{X}$ the set of edges of $M$ with both endpoints in $X$.

\begin{observation}\label{obs:matporandcycpor}
	Let $B$ be a bipartite graph with a perfect matching $M$ and $X\subseteq\V{B}$ be an $M$-conformal set.
	Then $\MatPor{\CutG{B}{X}}=\CycPor{\CutG{\DirM{B}{M}}{\Fkt{M}{X}}}$.	
\end{observation}

\begin{corollary}\label{thm:cycportoseparator}
	Let $D$ be a digraph and $X\subseteq\Fkt{V}{D}$.
	If $\CycPor{\CutG{D}{X}}=k$, then there is a hitting set of size at most $k^2+2k$ for all directed cycles crossing $\CutG{D}{X}$.
\end{corollary}

\begin{proof}
	With $D$ being a digraph there exists a bipartite graph $B$ and a perfect matching $M$ such that $D=\DirM{B}{M}$.
	For every $v\in\V{D}$ let us identify the edge $e_v\in M$ that uniquely corresponds to $v$ in $B$.
	Then $X$ naturally corresponds to the set $Y\coloneqq\V{\CondSet{e_v}{v\in X}}\subseteq\V{B}$.
	By calling upon \cref{lemma:guardingsseps}, it now suffices to show $\MatPor{\CutG{B}{Y}}\leq k$ in order to prove our claim.
	In fact $\MatPor{\CutG{B}{Y}}= k$ follows immediately from \cref{obs:matporandcycpor} and thus we are done.
\end{proof}

Please note that the proofs of \cref{lemma:guardingsseps,thm:cycportoseparator} are constructive in the sense that the construction of the separators and the digraphs $D_i$ can be done in polynomial time.
This yields the following algorithmic result.

\begin{corollary}
	Let $D$ be a digraph and $X\subseteq\Fkt{V}{D}$.
	There exists a polynomial time algorithm that finds a hitting set $S$ for all directed cycles that cross $\CutG{D}{X}$ of size at most $\CycPor{\CutG{D}{X}}^2+2\CycPor{\CutG{D}{X}}$.
\end{corollary}

\section{Perfect Matching Width and Directed Treewidth}\label{sec:pmwanddtw}

We are now ready to apply the findings of \cref{sec:separation} to prove a close relation between the perfect matching width of any bipartite graph $B$ with a perfect matching $M$ and its $M$-direction.
To achieve this we use the analogue of perfect matching width for digraphs induced by \hyperref[def:cycleporosity]{cycle porosity} from \cite{hatzel2019cyclewidth} called \emph{cycle width}.
We then use the relation between a version of the \emph{cops \& robber game} and directed treewidth and \cref{thm:cycportoseparator} to show that bounded cycle width implies bounded directed treewidth.

\begin{definition}[Cycle Width]
	Let $D$ be a directed graph.
	A \emph{cycle decomposition} of $D$ is a tuple $\Brace{T,\delta}$ where $T$ is a cubic tree and $\delta\colon \Leaves{T}\rightarrow\V{D}$ a bijection.
	Let $t_1t_2$ be an edge in $T$, then the partition of the tree into $T_{t_1}$ and $T_{t_2}$ also yields a partition of the vertices in $D$ that are mapped to the leaves of $T$.
	Let
	\begin{align*}
		X_i &\coloneqq \bigcup_{t \in \Leaves{T} \cap \V{T_{t_i}}} \Set{\Fkt{\delta}{t}}
	\end{align*}
	be the two classes of the partition.
	Note that $\CutG{D}{X_1} = \CutG{D}{X_2}$ defines an edge cut in $D$, we refer to it by $\CutG{D}{t_1t_2}$.
	The \emph{width} of a cycle decomposition $\Brace{T,\delta}$ is defined as \textbf{half}\footnote{Note that the cycle porosity of any cut is even, we add the factor $\frac{1}{2}$ to the definition to be more in line with other width parameters by allowing any value from $\N$.} of the value $\max_{e\in\E{T}}\CycPor{\CutG{D}{e}}$.
	Finally, the \emph{cycle width} of $D$, denoted by $\cycw{D}$, is the minimum width over all cycle decompositions for $D$.
\end{definition}

The close relation between the perfect matching width of a bipartite graph $B$ with a perfect matching $M$ and the cycle width of its $M$-direction has been established in \cite{hatzel2019cyclewidth}.

\begin{theorem}[\cite{hatzel2019cyclewidth}]\label{thm:cycwandpmw}
	Let $B$ be a bipartite graph with a perfect matching $M$.
	Then $\frac{1}{2}\pmw{B}\leq\cycw{\DirM{B}{M}}\leq \pmw{B}$.
\end{theorem}

In light of \cref{thm:cycwandpmw} it is clear that it suffices to relate cycle width and directed treewidth in order to also obtain a close relationship between directed treewidth and perfect matching width.
The remainder of this section is dedicated to proving the following result.

\begin{proposition}\label{thm:cycwdtw}
	Let $D$ be a digraph.
	Then $\cycw{D}-1\leq\dtw{D}\leq 18\cycw{D}^2+36\cycw{D}-2$.
\end{proposition}

While the first inequality has already been proven in \cite{hatzel2019cyclewidth}, the second inequality is a major improvement over the former bound which depended on the function of the Directed Grid Theorem.

\begin{lemma}[\cite{hatzel2019cyclewidth}]
Let $D$ be a digraph.
Then $\cycw{D}\leq\dtw{D}+1$.	
\end{lemma}

Let us now define the \emph{cops \& robber} game for digraphs.
Let $D$ be a digraph, the cops \& robber game on $D$ is played as follows:
\begin{enumerate}
	\item Initially the cops announce a starting position $C_0\subseteq\V{D}$ and the robber chooses her starting position $R_0$ which is a strong component of $D-C_0$.
	\item Suppose the game has been played for $i\in\N$ rounds, $C_i\subseteq\V{D}$ is the current position of the cops and $R_i$ is the current position of the robber.
	Then the cops announce a new position $C_{i+1}\subseteq\V{D}$ and the robber announces a new position $R_{i+1}$ such that $R_{i+1}$ is a strong component of $D-C_{i+1}$, and there exists a strong component $R$ in $D-C_i\cap C_{i+1}$ such that $R_i\cup R_{i+1}\subseteq R$.
	In case no such $R_{i+1}$ can be chosen, the robber is \emph{caught after $i+1$ rounds}.
\end{enumerate}
If the cops can guarantee that, no matter how the robber plays, she is always caught after finitely many rounds we say that the cops have a \emph{winning strategy}.
If the cops have a winning strategy and can also guarantee that any cop position $C_i$ they choose during any possible play of the game satisfies $\Abs{C_i}\leq k$, we say that \emph{$k$ cops can catch the robber} on $D$.

\begin{theorem}[\cite{johnson2001directed}]\label{thm:dtwandcopsandrobber}
	Let $D$ be a digraph and $k\in\N$.
	If $k$ cops can catch the robber on $D$, then $\dtw{D}\leq 3k-2$.
\end{theorem}

So all we have to do is to show that a cycle decomposition of bounded width for a digraph $D$ certifies that a bounded number of cops can catch the robber on $D$.

\begin{lemma}\label{lemma:catchtherobber}
Let $D$ be a digraph and $\Brace{T,\delta}$ be a cycle decomposition of width $k\in\N$ for $D$.
Then $6k^2+12k$ cops can catch the robber on $D$.
\end{lemma}

\begin{proof}
For every edge $e\in\E{T}$ let us denote by $S_e$ a minimum sized set of vertices of $D$ such that $S_e$ meets all directed cycles crossing $\CutG{D}{e}$.
Since the width of $\Brace{T,\delta}$ is $k$ we know $\CycPor{\CutG{D}{e}}\leq 2k$ and thus, by \cref{thm:cycportoseparator}, we know $\Abs{S_e}\geq 2k^2+4k$ for all $e\in\E{T}$.

Choose any vertex $v\in\V{D}$ and select $C_0\coloneqq v$.
Let $\ell t_0\in\E{T}$ be the unique edge of $T$ with $\Fkt{\delta}{\ell}=v$.
Independently of the choice of $R_0$ of the robber let $e_1^0,e_2^0\in\E{T}$ be the other two edges of $T$ incident with $t_0$ and let $C_1\coloneqq \Set{v}\cup S_{e_1}\cup S_{e_2}$.
Note that $\Abs{C-1}\leq6k^2+12k$.
For $i\in[1,2]$ let $T_i$ be the component of $T-t_0$ containing the other endpoint of $e_i$.
Let $r\in[1,2]$ be chosen such that $T_r$ is the unique component of $T-t_0$ which contains vertices of $R_1$ under $\delta$.
To see that $T_i$ is unique observe that $T-t_0$ consists of exactly three components, one of them only consisting of $\ell$.
Moreover, by definition of $S_{e_1}$ and $S_{e_2}$ no strong component of $D-S_{e_1}-S_{e_2}$ can contain vertices which are mapped to leaves of both $T_1$ and $T_2$ by $\delta$.
Hence we can now set $C_3\coloneqq S_{e_r}$ and can still be sure that the robber cannot leave $T_r$.

Now suppose we are in the following situation in round $i\geq 3$:
For all $j\in[0,i-1]$ we have $\Abs{C_j}\leq 6k^2+12k$ and there exists an edge $dt\in\E{T}$ such that $C_i=S_{dt}$ and the vertices of $R_i$ are all mapped to the leaves of $T_t$ which is the component of $T-dt$ that contains $t$.

In case $T_t$ has only a single vertex, namely $t$, $t$ must be a leaf of $T$ and so let $C_{i+1}\coloneqq C_i\cup\Set{\Fkt{\delta}{t}}$.
By our assumptions we have $\Abs{C_{i+1}}\leq 2k^2+4k+1\leq 6k^2+12k$, and we have also caught the robber as she cannot choose a new component.
Hence in this case we are done. 

So let $e_1'$ and $e_2'$ be the other edges of $T$ incident with $t$ and set $C_{i+1}\coloneqq S_{dt}\cup S_{e_1'}\cup S_{e_2'}$.
Again we have $\Abs{C_{i+1}}\leq 6k^2+12k$.
For each $j\in[1,2]$ let $T'_j$ be the component of $T-t$ that contains the other endpoint of $e'_j$.
Now robber must choose $R_{i+1}$ such that all vertices of $R_{i+1}$ are either completely mapped to the vertices of $T'_1$ or completely mapped to the vertices of $T'_2$.
Let $j\in[1,2]$ be chosen such that $\delta$ maps all vertices of $R_{i+1}$ to leaves of $T'_j$.
Then we can set $C_{i+2}\coloneqq S_{e_j'}$ and $R_{i+2}$ must still be chosen to only contain vertices which are mapped to the leaves of $T'_j$. 

As $D$ is finite so is $T$ and therefore, the subtree hosting the vertices of $R_{i+2}$ shrinks with every iteration.
Hence after finitely many rounds we are sure to have caught the robber using at most $6k^2+12k$ cops.
\end{proof}

Combining \cref{thm:dtwandcopsandrobber,lemma:catchtherobber} immediately yields the second inequality of \cref{thm:cycwdtw} and thus we are done.

\section{Matching Minors and Fundamental Anti-Chains in Digraphs}\label{sec:minors}

To be able to precisely describe \hyperref[def:matchingminor]{matching minors} in terms of alternating paths which can be found by an algorithmic solution to our matching theoretic analogue of the $t$-Disjoint Paths Problem\footnote{See \cref{subsec:linkages} for more information.}, we need a more precise and formal description of matching minors.

In the second part of this section we introduce our notion of fundamental anti-chains for butterfly minors in digraphs based on the definition of matching minors.

\subsection{Matching Minor Models}

To speak about matching minors in a more formal way, we introduce the concept of models for matching minors.
Models, or embeddings, for matching minors have already been used and discussed in \cite{robertson1999permanents} and \cite{norine2007generating} and the definitions we give here are similar to those of Norine and Thomas.
Some parts of these definitions, however, have been changed to better suit our needs in the sections to come and therefore we provide the necessary proofs.

Let $T'$ be a tree and let $T$ be obtained from $T'$ by subdividing every edge an odd number of times.
Then $\Fkt{V}{T'} \subseteq \Fkt{V}{T}$.
The vertices of $T$ that belong to $T'$ are called \emph{old}, and the vertices in $\Fkt{V}{T} \setminus \Fkt{V}{T'}$ are called \emph{new}.
We say that $T$ is a \emph{barycentric tree}.

\begin{definition}[Matching Minor Model]\label{def:matchingminormodel}
	Let $G$ and $H$ be graphs with perfect matchings.
	An \emph{embedding} or \emph{matching minor model} of $H$ in $G$ is a mapping 
	\begin{equation*}
		\mu \colon \Fkt{V}{H} \cup \Fkt{E}{H} \to \CondSet{F}{F\subseteq G},
	\end{equation*}
	such that the following requirements are met for all $v,v' \in \Fkt{V}{H}$ and $e,e' \in \Fkt{E}{H}$:
	\begin{enumerate}
		\item $\Fkt{\mu}{v}$ is a barycentric subtree in $G$,
		
		\item if $v \neq v'$, then $\Fkt{\mu}{v}$ and $\Fkt{\mu}{v'}$ are vertex disjoint,
		
		\item $\Fkt{\mu}{e}$ is an odd path with no internal vertex in any $\Fkt{\mu}{v}$, and if $e' \neq e$, then $\Fkt{\mu}{e}$ and $\Fkt{\mu}{e'}$ are internally vertex disjoint,
		
		\item if $e=u_1u_2$, then the ends of $\Fkt{\mu}{e}$ can be labelled by $x_1,x_2$ such that $x_i$ is an old vertex of $\Fkt{\mu}{u_i}$,
		
		\item if $v$ has degree one, then $\Fkt{\mu}{v}$ is exactly one vertex, and
		
		\item $G-\Fkt{\mu}{H}$ has a perfect matching, where $\Fkt{\mu}{H'} \coloneqq \bigcup_{x\in \Fkt{V}{H'} \cup \Fkt{E}{H'}}\Fkt{\mu}{x}$ for every subgraph $H'$ of $H$.
	\end{enumerate}	
	If $\mu$ is a matching minor model of $H$ in $G$ we write $\mu\colon H\rightarrow G$.
\end{definition}

While we slightly changed the definition here, the next lemma follows immediately from a result of \cite{norine2007generating} and thus we omit the proof.

\begin{lemma}[\cite{norine2007generating}]\label{lemma:matmodel}
	Let $G$ and $H$ be graphs with perfect matchings.
	There exists a matching minor model $\mu\colon H\rightarrow G$ if and only if $H$ is isomorphic to a matching minor of $G$.
\end{lemma}

\begin{lemma}\label{lemma:minormatching}
	Let $H$ and $G$ be graphs and $\mu\colon H\rightarrow G$ be an embedding of $H$ into $G$. Then $H$ has a perfect matching if and only if $\Fkt{\mu}{H}$ has a perfect matching.	
\end{lemma}

\begin{proof}
	Suppose $H$ has a perfect matching.
	We prove our claim by induction on the number $c$ of bicontractions	that have to be applied to $\Fkt{\mu}{H}$ in order to obtain a graph isomorphic to $H$.
	For $c=0$ this implies $\Fkt{\mu}{H}=H$ and $H$ has a perfect matching.
	
	So let $c\geq 1$.
	Starting with $\Fkt{\mu}{H}$ let $b_1,\dots,b_c$ be the bicontractions that need to be applied and $H_i$ be the graph obtained from $\Fkt{\mu}{H}$ by only applying the bicontractions $b_1,\dots,b_i$.
	Furthermore, let those bicontractions be ordered in such a way that $H_c=H$ and $H_0=\Fkt{\mu}{H}$, where $H_0$ is the uncontracted graph and moreover $H_i$ is obtained from $H_{i-1}$ by applying exactly one bicontraction.
	
	Hence $H_1$ is a matching minor of $G$ that also contains $H$ as a matching minor and $H$ can be obtained from $H_1$ by applying $b_2,\dots,b_c$, which are $c-1$ bicontractions, let $\mu_1$ be a corresponding matching model of $H$ in $H_1$, then $\Fkt{\mu_1}{H}=H_1$.
	By our induction hypothesis, $H_1$ has a perfect matching.
	The transition from $\Fkt{\mu}{H}$ to $H_1$ is done by applying $b_1$ to $\Fkt{\mu}{H}$.
	Let $v_0$ be the vertex in $\Fkt{\mu}{H}$ that is to be bicontracted by $b_1$, and let $v_1,v_2$ be its two unique neighbours, and let $v$ be the new vertex in $H_1$ after the bicontraction.
	Since $H_1$ has a perfect matching, there is some vertex $x\in\Fkt{V}{H_1}\cap\Fkt{V}{\Fkt{\mu}{H}}$ such that $xv$ is a perfect matching edge in $H_1$.
	Therefore, there must be $v_i$ with $i\in\Set{1,2}$, say $i=1$, such that $xv_1$ is an edge of $\Fkt{\mu}{H}$ by the definition of matching models.
	Let $M$ be some perfect matching of $H_1$ containing $xv$, then $M\setminus\Set{xv}$ is a perfect matching of $\Fkt{\mu}{H}-x-v_0-v_1-v_2$.
	Let $M'\coloneqq \Brace{M\cup\Set{xv_1,v_0v_2}}\setminus\Set{xv}$, then $M'$ is a perfect matching of $\Fkt{\mu}{H}$.
	
	The reverse direction follows along similar lines and is therefore omitted.
\end{proof}

If $\mu\colon H\rightarrow G$ is a matching minor model of a matching covered graph $H$ in $G$, then both $G-\Fkt{\mu}{H}$ and $\Fkt{\mu}{H}$ have a perfect matching.
Let $M$ be a perfect matching of $G$ such that $M\cap \Fkt{E}{G-\Fkt{\mu}{H}}$ is a perfect matching of $G-\Fkt{\mu}{H}$ and $M'\coloneqq M\cap\Fkt{E}{\Fkt{\mu}{H}}$ is a perfect matching of $\Fkt{\mu}{H}$.
Then there is a perfect matching of $H$ that `mimics' the structure of $M'$ in $\Fkt{\mu}{H}$.
In the following, we explain what we mean with the word `mimics'.

\begin{lemma}\label{lemma:matchingsofmodels}
	Let $G$ and $H$ be graphs with perfect matchings, $\mu\colon H\rightarrow G$, and $M$ a perfect matching of $\Fkt{\mu}{H}$.
	Then for every $u\in\Fkt{V}{H}$, there is a unique vertex $v\in\Fkt{N_H}{u}$ such that $\Fkt{\mu}{uv}$ is an $M$-conformal path and for all other edges $e\in\Fkt{E}{H}$ incident with $u$ their respective model $\Fkt{\mu}{e}$ is internally $M$-conformal.
\end{lemma}

\begin{proof}
	Let $e\in\Fkt{E}{H}$ be any edge, then $\Fkt{\mu}{e}$ is a path of odd length where every inner vertex has degree two in $\Fkt{\mu}{H}$.
	Thus for every perfect matching $M'$ of $\Fkt{\mu}{H}$, $\Fkt{\mu}{e}$ either is an internally $M'$-conformal path, i.e.,\@ $M'$ contains a perfect matching of $\Fkt{\mu}{e}$ without its endpoints, or $M'$ contains a perfect matching of $\Fkt{\mu}{e}$.
	
	For any vertex $u\in\Fkt{V}{H}$ let us call $t\in\Fkt{V}{\Fkt{\mu}{u}}$ \emph{exposed} if the edge of $M$ covering $t$ is not an edge of the barycentric tree $\Fkt{\mu}{u}$.
	Please note that for every exposed vertex $t$ of $\Fkt{\mu}{u}$ there must be an edge $e\in\Fkt{E}{H}$ such that $t$ is an endpoint of $\Fkt{\mu}{e}$ and the edge of $M$ covering $t$ is an edge of $\Fkt{\mu}{e}$.
	Moreover, in this case $\Fkt{\mu}{e}$ cannot be internally $M$-conformal and thus must be $M$-conformal by the observation above.
	Hence the other endpoint of $\Fkt{\mu}{e}$, which is a vertex of $\Fkt{\mu}{v}$ for some $v\in\Fkt{V}{H}$, must also be exposed.
	These observations immediately imply that any exposed vertex in $\Fkt{\mu}{u}$ must be an old vertex.
	
	Next, observe that every path $P$ in $\Fkt{\mu}{u}$ that connects two old vertices and otherwise consists only of new vertices is of even length.
	Similar to our observation for $\Fkt{\mu}{e}$, every inner vertex of $P$ must be covered by an edge of $\Fkt{E}{P}\cap M$.
	Hence there exists exactly one vertex of $P$ that is not covered by an edge of $\Fkt{E}{P}\cap M$.
	
	So in order to prove our claim, we have to show that for every $u\in\Fkt{V}{H}$ there is exactly one exposed vertex in $\Fkt{\mu}{u}$.
	To do this, we generalise the observation on the even paths within $\Fkt{\mu}{u}$ we made above.
	Let $T=\Fkt{\mu}{u}$ be a barycentric tree and let $O$ be the set of old vertices of $T$. Moreover, let $T'$ be the tree with $\Fkt{V}{T'}=O$ from which $T$ was constructed by subdividing every edge an odd number of times.
	Then any two old vertices that are adjacent in $T'$ are linked by a path of even length in $T$.
	Hence in a proper $2$-colouring of $T$, all vertices of $O$ receive the same colour.
	Now let $e\in\Fkt{E}{T'}$ be any edge and $P_e$ the corresponding path in $T$, moreover let $P\coloneqq P_e-O$.
	Then $P$ is a path of even length as well and thus in a proper $2$-colouring of $T$ its endpoints receive the same colour.
	With $P$ being of even length, it has an odd number of vertices, say $2k+1$, and thus in a proper $2$-colouring of $T$, $P$ has $k+1$ vertices whose colour is different from the colour of the old vertices and $k$ vertices with the same colour as the old vertices, combining this with our observation above that in each $P_e$ exactly one vertex is not covered by an edge of $\Fkt{E}{P_e}\cap M$.
	This yields that in total $\Abs{O}-\Abs{\Fkt{E}{T'}}=1$ vertices of $T$ must be exposed by every perfect matching of $\Fkt{\mu}{u}$.
\end{proof}

In the situation of \cref{lemma:matchingsofmodels} let $M$ be a perfect matching of $\Fkt{\mu}{H}$, then for every $u\in\Fkt{V}{H}$ there is a unique vertex $v\in\Fkt{V}{H}$ such that $\Fkt{\mu}{uv}$ is $M$-conformal.
Let 
\begin{align*}
	\Reduct{M}{H}\coloneqq\CondSet{uv\in\Fkt{E}{H}}{\Fkt{\mu}{uv}~\text{is}~M\text{-conformal}},
\end{align*}
then $\Reduct{M}{H}$ is a perfect matching of $H$.
In a slight extension of our definition of residual matching we call $\Reduct{M}{H}$ the \emph{$M$-residual} matching of $H$.
Moreover, we call a matching minor model $\mu\colon H\rightarrow G$ an \emph{$M$-model} of $H$ in $G$ if $\Fkt{\mu}{H}$ is $M$-conformal.
With this, we obtain the following corollary.

\begin{corollary}\label{cor:Mmodels}
	Let $G$ and $H$ be graphs with perfect matchings and $M$ a perfect matching of $G$.
	Then $H$ is isomorphic to an $M$-minor of $G$ if and only if there exists an $M$-model $\mu_M\colon H\rightarrow G$ in $G$.
\end{corollary}

\subsection{Fundamental Anti-Chains of Butterfly Minors}\label{subsec:antichains}

Recall the definition of \hyperref[def:Mdirection]{$M$-directions} of bipartite graphs and that we argued that the operation can be reversed to yield a unique bipartite graph with a perfect matching when given a digraph $D$ as input.
Let us formalise this inverse operation.

\begin{definition}[Split]\label{def:split}
	Let $D$ be a directed graph.
	We define $\Split{D}$ to be the bipartite graph $B$ for which a perfect matching $M$ exists such that $\DirM{B}{M}=D$.
\end{definition}

The digraph $J$ is a \emph{proper butterfly minor} of the digraph $D$ if $J$ is a butterfly minor of $D$ and $J\not\cong D$.
We say that $D$ is \emph{$J$-minimal} if $\Split{D}$ contains $\Split{J}$ as a matching minor, but for every proper butterfly minor $D'$ of $D$, $\Split{D'}$ is $\Split{J}$-free.

\begin{definition}[Fundamental Anti-Chain]
	Let $D$ be a digraph.
	The family
	\begin{align*}
		\Antichain{D}\coloneqq\CondSet{D'}{\text{$D'$ is a $D$-minimal digraph}}
	\end{align*}
	is called the \emph{fundamental anti-chain based on $D$}.
\end{definition}

\begin{lemma}\label{lemma:canonicalantichainisantichain}
	Let $D$ be a digraph.
	Then $\Antichain{D}$ is an anti-chain for the butterfly minor relation.
\end{lemma}

\begin{proof}
	Suppose $\Antichain{D}$ is not an anti-chain for the butterfly minor relation.
	Then there must exist $D_1$ and $D_2$ in $\Antichain{D}$ such that $D_1$ is a butterfly minor of $D_2$.
	Indeed, $D_1$ must be a proper butterfly minor of $D_2$, as otherwise the two digraphs would be isomorphic.
	Since by definition $\Split{D_1}$ contains $\Split{D}$ as a matching minor, from Lemma~\ref{lemma:mcguigmatminors}, $D_2$ cannot be $D$-minimal, which contradicts $D_2\in\Antichain{D}$.
	Hence $\Antichain{D}$ must be an anti-chain for the butterfly minor relation.
\end{proof}

\begin{definition}[Matching Equivalent]
	Two digraphs $D_1$ and $D_2$ are said to be \emph{matching equivalent} if $\Split{D_1}$ and $\Split{D_2}$ are isomorphic.
	Given a digraph $D$, we denote by $\MatchingEquivalent{D}$ the family of all digraphs that are matching equivalent to $D$.
\end{definition}

\begin{lemma}\label{lemma:matchingequivalentantichain}
	Let $D$ be a digraph.
	Then $\MatchingEquivalent{D}\subseteq\Antichain{D}$.
\end{lemma}

\begin{proof}
	First note that there cannot be a pair of distinct digraphs $D_1,D_2\in\MatchingEquivalent{D}$ such that $D_1$ is a proper butterfly minor of $D_2$.
	If this were the case, then $\Split{D_1}$ would be a proper matching minor of $\Split{D_2}$, but by definition $\Split{D_1}$ and $\Split{D_2}$ must be isomorphic.
	Hence $\MatchingEquivalent{D}$ forms an anti-chain for the butterfly minor relation.
	Moreover, any proper butterfly minor $D'$ of some digraph in $\MatchingEquivalent{D}$ must satisfy that $\Split{D'}$ is $\Split{D}$-free and thus the claim follows.
\end{proof}

\begin{figure}[h!]
	\begin{center}
		\begin{tikzpicture}[scale=0.8]
			
			\node (D1) [v:ghost] {};
			\node (D2) [v:ghost,position=0:30mm from D1] {};
			\node (D3) [v:ghost,position=0:35mm from D2] {};
			\node (D4) [v:ghost,position=0:25mm from D3] {$\dots$};

			\node (C1) [v:ghost,position=0:0mm from D1] {
				
				\begin{tikzpicture}[scale=0.8]
					
					\pgfdeclarelayer{background}
					\pgfdeclarelayer{foreground}
					
					\pgfsetlayers{background,main,foreground}
					
					\begin{pgfonlayer}{main}
						
						%%%%% Centered Ghost Vertices %%%%%
						\node (C) [] {};

						%%%%% Vertices %%%%%
						
						%%%%% Left Center %%%%%

						%%%%% %%%%% %%%%%
						
						%%%%% Center %%%%%
						
						\node (v1) [v:main,position=90:10mm from C] {};
						\node (v2) [v:main,position=210:10mm from C] {};
						\node (v3) [v:main,position=330:10mm from C] {};
						
						%%%%% %%%%% %%%%%
						
						%%%%% Right Center %%%%%
						
						%%%%% %%%%% %%%%%
						
						%%%%% %%%%% %%%%%

						%%%%% Edges %%%%%
						
						%%%%% Left Center %%%%%

						%%%%% %%%%% %%%%%
						
						%%%%% Center %%%%%
						
						\draw (v1) [e:main,->,bend right=15] to (v2);
						\draw (v2) [e:main,->,bend right=15] to (v3);
						\draw (v3) [e:main,->,bend right=15] to (v1);
						
						\draw (v1) [e:main,->,bend right=15] to (v3);
						\draw (v2) [e:main,->,bend right=15] to (v1);
						\draw (v3) [e:main,->,bend right=15] to (v2);

						%%%%% %%%%% %%%%%
						
						%%%%% Right Center %%%%%

						%%%%% %%%%% %%%%%
						
						%%%%% %%%%% %%%%%
						
					\end{pgfonlayer}
					
					%%%%% %%%%% %%%%%

					%%%%% Background %%%%%
					\begin{pgfonlayer}{background}
						
					\end{pgfonlayer}	
					%%%%% %%%%% %%%%%
					
					%%%%% Foreground %%%%%
					\begin{pgfonlayer}{foreground}

					\end{pgfonlayer}
					%%%%% %%%%% %%%%%
				\end{tikzpicture}
				
			};
			
			\node (C2) [v:ghost,position=0:0mm from D2] {
				
				\begin{tikzpicture}[scale=0.8]
					
					\pgfdeclarelayer{background}
					\pgfdeclarelayer{foreground}
					
					\pgfsetlayers{background,main,foreground}
					
					\begin{pgfonlayer}{main}
						
						%%%%% Centered Ghost Vertices %%%%%
						\node (C) [] {};

						%%%%% Vertices %%%%%
						
						%%%%% Left Center %%%%%

						%%%%% %%%%% %%%%%
						
						%%%%% Center %%%%%
						
						\node (v1) [v:main,position=90:12mm from C] {};
						\node (v2) [v:main,position=162:12mm from C] {};
						\node (v3) [v:main,position=234:12mm from C] {};
						\node (v4) [v:main,position=306:12mm from C] {};
						\node (v5) [v:main,position=18:12mm from C] {};
						
						%%%%% %%%%% %%%%%
						
						%%%%% Right Center %%%%%
						
						%%%%% %%%%% %%%%%
						
						%%%%% %%%%% %%%%%

						%%%%% Edges %%%%%
						
						%%%%% Left Center %%%%%

						%%%%% %%%%% %%%%%
						
						%%%%% Center %%%%%
						
						\draw (v1) [e:main,->,bend right=15] to (v2);
						\draw (v2) [e:main,->,bend right=15] to (v3);
						\draw (v3) [e:main,->,bend right=15] to (v4);
						\draw (v4) [e:main,->,bend right=15] to (v5);
						\draw (v5) [e:main,->,bend right=15] to (v1);
						
						\draw (v1) [e:main,->,bend right=15] to (v5);
						\draw (v2) [e:main,->,bend right=15] to (v1);
						\draw (v3) [e:main,->,bend right=15] to (v2);
						\draw (v4) [e:main,->,bend right=15] to (v3);
						\draw (v5) [e:main,->,bend right=15] to (v4);
						
						%%%%% %%%%% %%%%%
						
						%%%%% Right Center %%%%%

						%%%%% %%%%% %%%%%
						
						%%%%% %%%%% %%%%%
						
					\end{pgfonlayer}
					
					%%%%% %%%%% %%%%%

					%%%%% Background %%%%%
					\begin{pgfonlayer}{background}
						
					\end{pgfonlayer}	
					%%%%% %%%%% %%%%%
					
					%%%%% Foreground %%%%%
					\begin{pgfonlayer}{foreground}

					\end{pgfonlayer}
					%%%%% %%%%% %%%%%
				\end{tikzpicture}
				
			};
			
			\node (C3) [v:ghost,position=0:0mm from D3] {
				
				\begin{tikzpicture}[scale=0.8]
					
					\pgfdeclarelayer{background}
					\pgfdeclarelayer{foreground}
					
					\pgfsetlayers{background,main,foreground}
					
					\begin{pgfonlayer}{main}
						
						%%%%% Centered Ghost Vertices %%%%%
						\node (C) [] {};

						%%%%% Vertices %%%%%
						
						%%%%% Left Center %%%%%

						%%%%% %%%%% %%%%%
						
						%%%%% Center %%%%%
						
						\node (v1) [v:main,position=90:15mm from C] {};
						\node (v2) [v:main,position=141.4:15mm from C] {};
						\node (v3) [v:main,position=192.8:15mm from C] {};
						\node (v4) [v:main,position=243.2:15mm from C] {};
						\node (v5) [v:main,position=294.6:15mm from C] {};
						\node (v6) [v:main,position=346.1:15mm from C] {};
						\node (v7) [v:main,position=37.5:15mm from C] {};
						
						%%%%% %%%%% %%%%%
						
						%%%%% Right Center %%%%%
						
						%%%%% %%%%% %%%%%
						
						%%%%% %%%%% %%%%%

						%%%%% Edges %%%%%
						
						%%%%% Left Center %%%%%

						%%%%% %%%%% %%%%%
						
						%%%%% Center %%%%%
						
						\draw (v1) [e:main,->,bend right=15] to (v2);
						\draw (v2) [e:main,->,bend right=15] to (v3);
						\draw (v3) [e:main,->,bend right=15] to (v4);
						\draw (v4) [e:main,->,bend right=15] to (v5);
						\draw (v5) [e:main,->,bend right=15] to (v6);
						\draw (v6) [e:main,->,bend right=15] to (v7);
						\draw (v7) [e:main,->,bend right=15] to (v1);
						
						\draw (v2) [e:main,->,bend right=15] to (v1);
						\draw (v3) [e:main,->,bend right=15] to (v2);
						\draw (v4) [e:main,->,bend right=15] to (v3);
						\draw (v5) [e:main,->,bend right=15] to (v4);
						\draw (v6) [e:main,->,bend right=15] to (v5);
						\draw (v7) [e:main,->,bend right=15] to (v6);
						\draw (v1) [e:main,->,bend right=15] to (v7);

						%%%%% %%%%% %%%%%
						
						%%%%% Right Center %%%%%

						%%%%% %%%%% %%%%%
						
						%%%%% %%%%% %%%%%
						
					\end{pgfonlayer}
					
					%%%%% %%%%% %%%%%

					%%%%% Background %%%%%
					\begin{pgfonlayer}{background}
						
					\end{pgfonlayer}	
					%%%%% %%%%% %%%%%
					
					%%%%% Foreground %%%%%
					\begin{pgfonlayer}{foreground}

					\end{pgfonlayer}
					%%%%% %%%%% %%%%%
				\end{tikzpicture}
				
			};
			
		\end{tikzpicture}
	\end{center}
	\caption{The anti-chain $\Antichain{\Bidirected{K_3}}$.}
	\label{fig:oddbicycles}
\end{figure}

At last let us prove that excluding $\Split{D}$ as a matching minor is the same as excluding every digraph in $\Antichain{D}$ as a butterfly minor.

\begin{lemma}\label{lemma:excludingantichains}
	Let $H$ and $D$ be digraphs.
	Then $D$ contains a butterfly minor from $\Antichain{H}$ if and only if $\Split{D}$ contains $\Split{H}$ as a matching minor.
\end{lemma}

\begin{proof}
	Suppose $D$ contains a butterfly minor from $\Antichain{H}$, say $J$,
	Then by \cref{lemma:mcguigmatminors} $\Split{D}$ must contain $\Split{J}$ as a matching minor, and by definition of $\Antichain{H}$, $\Split{J}$ must contain $\Split{H}$ as a matching minor.
	For the reverse, assume $\Split{D}$ contains $\Split{H}$ as a matching minor.
	Then let $J$ be an $H$-minimal butterfly minor of $D$.
	Clearly, $J$ must exist and $J\in\Antichain{H}$.
\end{proof}

\subsection{The Strong Genus of Digraphs}\label{subsec:stronggenus}

Recall our discussion of the Erd\H{o}s-P\'osa property for butterfly minors in \cref{subsec:digraphs} and consider the cylindrical grid $D$ of any order together with a planar embedding.
Now zoom in on any vertex $v$ and inspect an open disc $\zeta$ with $v$ at its centre such that $\zeta$ does not contain any other vertex of $D$.
Note that we can draw a curve $\gamma$ through $v$ connecting two points of the boundary of $\zeta$ such that every incoming edge of $D$ incident with $v$ lies on one side of $\gamma$, while every edge emanating from $v$ lies on the other side of $\gamma$.
Moreover, note that, by the definition of butterfly minors, every butterfly minor of $D$ must also have a plane embedding with this property.
With this we may rule out any planar digraph which does not have such an embedding as a candidate for being a butterfly minor of the cylindrical grid\footnote{In \cref{fig:nongridminorplanardigraph} however, one can see a digraph that has such an embedding and still is not a butterfly minor of the cylindrical grid.}.
See \cref{fig:nonstronglyplanargigraph} for a strongly connected planar digraph which does not have such an embedding.
Moreover, notice that this particular digraph has exactly two butterfly contractible edges and by contracting both of them one obtains $\Bidirected{K_3}$.

\begin{figure}[!h]
	\centering
	\begin{tikzpicture}
		\pgfdeclarelayer{background}
		\pgfdeclarelayer{foreground}
		\pgfsetlayers{background,main,foreground}
		
		\node (mid) [v:ghost] {};
		
		\node (v1) [v:main,position=90:12mm from mid] {};
		\node (v2) [v:main,position=210:8mm from mid] {};
		\node (v3) [v:main,position=330:8mm from mid] {};
		
		\node (v4) [v:main,position=90:12mm from mid] {};
		\node (v5) [v:main,position=210:16mm from mid] {};
		\node (v6) [v:main,position=330:16mm from mid] {};
		
		\draw [e:main,->,bend right=25] (v1) to (v2);
		\draw [e:main,<-,bend right=45] (v2) to (v3);
		\draw [e:main,->,bend right=25] (v3) to (v1);
		
		\draw [e:main,->,bend left=45] (v4) to (v6);
		\draw [e:main,<-,bend left=45] (v6) to (v5);
		\draw [e:main,->,bend left=45] (v5) to (v4);
		
		%	\draw [e:main,->] (v1) to (v4);
		\draw [e:main,<-] (v5) to (v2);
		\draw [e:main,->] (v6) to (v3);
		
	\end{tikzpicture}
	\caption{A strongly connected and planar digraph that has no strong embedding.}
	\label{fig:nonstronglyplanargigraph}
\end{figure}
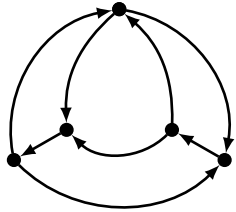

Let us formally introduce this concept.
The definitions given here only scratch the surface of topological graph theory, see \cite{stahl1978embeddings,archdeacon1996topological} for broader introduction and an overview of the topic.

Let $G$ be a graph or digraph.
Then $G$ corresponds to a topological space called the \emph{geometric realisation} of $G$.
In this space the vertices are distinct points and the edges are subspaces homeomorphic to the closed interval $[0,1]$ over the real numbers\footnote{In this instance we do \textbf{not} use our definition of $[0,1]$ as the set $\Set{0,1}$.} joining their endpoints.
An \emph{embedding} of $G$ into some topological space $X$ is a homeomorphism between the geometric realisation of $G$ and a subspace of $X$.
In a slight abuse of notation we use $G$ for both the graph $G$ and its geometric realisation.
A \emph{surface} is a compact Hausdorff topological space which is locally isomorphic to $\R^2$.
There are two ways to construct these surfaces; either take a sphere and attach $n\in\N$ handles to it, or take a sphere and attach $m\in\N$ crosscaps.
Let us denote by $\Sigma_n$ the surfaces of the first kind and by $\tilde{\Sigma}_m$ the surfaces of the second kind.

\begin{theorem}[\cite{brahana1921systems}]\label{thm:surfaces}
	The surfaces in $\CondSet{\Sigma_n}{n\in\N}$ and $\CondSet{\tilde{\Sigma}_M}{m\in\N,~m\geq 1}$ are pairwise non-homeomorphic and every surface is homeomorphic to a member of one of these two families.
\end{theorem}

A surface $\Sigma$ is \emph{orientable} and of \emph{orientable genus $n$} if it is homeomorphic to $\Sigma_n$. Similarly, $\Sigma$ is \emph{non-orientable} and of \emph{non orientable genus $m$} if it is homeomorphic to $\tilde{\Sigma}_m$.

A \emph{$2$-cell embedding} or \emph{map} of a (di)graph $G$ is an embedding in which every face is homeomorphic to an open disk.
The \emph{genus} of a (di)graph $G$ is the smallest integer $g\in\N$ such that $G$ can be embedded in $\Sigma_g$, and its \emph{non-orientable genus} is the smallest integer $g'\in\N$ such that $G$ can be embedded in $\tilde{\Sigma}_{g'}$.
The \emph{Euler genus} of $G$, denoted by $\Genus{G}$, is the smallest integer $h\in\N$ such that $G$ can be embedded in $\Sigma_{\frac{h}{2}}$ or $\tilde{\Sigma}_h$.

Let $G$ be a (di)graph embedded in a surface $\Sigma$, $v\in\V{G}$ a vertex and $\zeta\subseteq\Sigma$ an open disc centred at $v$ such that every edge of $G$ incident with $v$ contains exactly one point from the boundary $\beta$ of $\zeta$.
Let $F\subseteq\E{G}$ be the edges incident with $v$ and $\Set{F_1,F_2}$ be a bipartition of $F$.
For each $f\in F$ let $p_f\in \beta$ be the point that $f$ has on the boundary of $\zeta$.
We say that $\Brace{F_1,F_2}$ is a \emph{butterfly in $\zeta$} if there exists a curve $\gamma$ through $v$ in $\Sigma$ with both endpoints, $x$ and $y$ on $\beta$ such that we can number the two internally disjoint curves $\beta_1\subseteq\beta$ and $\beta_2\subseteq\beta$ with endpoints $x$ and $y$ to obtain $\CondSet{p_f}{f\in F_i}\subseteq \beta_i$ for both $i\in[1,2]$.

\begin{definition}[Strong Embedding]\label{def:stronggenus}
	Let $D$ be a digraph and $\Sigma$ be a surface.
	An embedding $\mu\colon D\rightarrow\Sigma$ of $D$ into $\Sigma$ is \emph{strong} if for every vertex $v\in\V{D}$ there exist $r_v\in\R$ and an open disc $\zeta\subseteq\Sigma$ of radius $r_v$ centred at $v$ such that
	\begin{align*}
		\Brace{\CondSet{\Brace{u,v}}{\Brace{u,v}\in\E{D}},\CondSet{\Brace{v,u}}{\Brace{v,u}\in\E{D}}}
	\end{align*}
	is a butterfly in $\zeta$.
	
	The smallest integer $h\in\N$ such that $D$ can be strongly embedded in $\Sigma_{\frac{h}{2}}$ or $\tilde{\Sigma}_h$ is called the \emph{strong genus} of $D$.
	We denote the strong genus of $D$ by $\StrongGenus{D}$.
	If $\StrongGenus{D}=0$, $D$ is said to be \emph{strongly planar}.
\end{definition}

Note that the strong genus of a digraph $D$ is closed under vertex and edge deletion.
Moreover, let $e=\Brace{u,v}$ be a butterfly contractible edge of $D$ and assume $D$ is strongly embedded into some surface $\Sigma$.
By definition of butterfly minors $\Brace{u,v}$ is the only outgoing edge of $u$, or the only incoming edge at $v$.
In both cases, after adjusting the embedding of $D$ into $\Sigma$ for the digraph $D'$ obtained from $D$ by contracting $e$, the incoming and outgoing edges of the contraction vertex $w$ still form a butterfly in some open disc in $\Sigma$ centred at $w$.
Hence we have the following observation.

\begin{observation}
	Let $D$ be a digraph and $D'$ be a butterfly minor of $D$, then $\StrongGenus{D'}\leq\StrongGenus{D}$.	
\end{observation}

The strong genus of digraphs is closely linked to the Euler genus of their splits.

\begin{proposition}\label{pro:translategenus}
	Let $B$ be a bipartite graph with a perfect matching $M$ and $D\coloneqq\DirM{G}{M}$.
	Then $\Genus{B}=\StrongGenus{D}$.
\end{proposition}

\begin{proof}
	First let $g\coloneqq\Genus{B}$ and consider an embedding of $B$ into a surface $\Sigma$ such that $\Sigma=\Sigma_{2g}$ if $B$ has an embedding in $\Sigma_{2g}$, and $\Sigma=\tilde{\Sigma}_g$ otherwise.
	Now contract the edges of $M$ and let $ab\in M$ be any edge.
	Note that we may find an open disc $\zeta\subseteq\Sigma$ and a curve $\gamma$ through $a$ such that $\Set{\Set{ab},\CondSet{ax}{ax\in \E{B-b}}}$ is a butterfly in $\zeta$.
	Indeed, the same holds true if we swap $a$ and $b$.
	Hence after contracting $ab$ into the vertex $v_{ab}$, $\Set{\CondSet{v_{ab}x}{ax\in \E{B-b}},\CondSet{v_{ab}x}{bx\in \E{B-a}}}$ is a butterfly in $\zeta$ and thus $D$ has a strong embedding in $\Sigma$.
	
	For the reverse let $h\coloneqq\StrongGenus{B}$ and consider a strong embedding of $D$ into a surface $\Sigma$ such that $\Sigma=\Sigma_{2h}$ if $D$ has a strong embedding in $\Sigma_{2h}$, and $\Sigma=\tilde{\Sigma}_h$ otherwise.
	Consider $\Split{D}$ and let $M$ be the perfect matching of $\Split{D}$ such that $D$ is the $M$-direction of $\Split{D}$.
	Adapt the embedding of $D$ in $\Sigma$ for $\Split{D}$ by placing the two endpoints of each edge in $M$ as close together as possible.
	Let $v\in\V{D}$ be any vertex and $e_v=ab\in M$ the corresponding matching edge in $\Split{D}$ with $a\in V_1$.
	Since, in our embedding of $D$ in $\Sigma$, the out- and incoming edges at  every vertex$v\in\V{D}$ form a butterfly, this butterfly induces a bipartition of the edges of $\Split{D}$ incident with the endpoints of $e_v$ that resembles this butterfly.
	Hence the edge $ab$ can be added to the embedding without producing a crossing.
\end{proof}

Hence we obtain the following immediate corollary which was implicitly stated in \cite{robertson1999permanents,guenin2011packing}.

\begin{corollary}[\cite{{robertson1999permanents,guenin2011packing}}]\label{cor:stronglyplanar1}
	A digraph $D$ is strongly planar if and only if $\Split{D}$ is planar.
\end{corollary}

\paragraph{Strongly Planar Digraphs and Butterfly Minors of the Cylindrical Grid}

It can be observed, as described above, that the cylindrical grid is strongly planar.
Still strong planarity does not seem to be enough as the digraph in \cref{fig:nongridminorplanardigraph} is also strongly planar\footnote{In fact every subcubic digraph that is planar is necessarily strongly planar.} but not a butterfly minor of the cylindrical grid.
So for fixed digraphs there probably is no analogue of \cref{thm:matchinggridminors}.
If, however, we consider fundamental anti-chains, we can produce very similar results.

\begin{proposition}\label{thm:directedgridminors}
	Let $D$ be a strongly connected digraph.
	Then $D$ is strongly planar if and only if $\Antichain{D}$ contains a butterfly minor of the cylindrical grid.
\end{proposition}

\begin{proof}
	Let us assume $D$ to be strongly planar.
	Then $B\coloneqq\Split{D}$ is planar and matching covered.
	By \cref{thm:matchinggridminors} $B$ is a matching minor of the $\omega_B\times\omega_B$-grid.
	The $\omega_B\times\omega_B$-grid however is a matching minor of $CG_{3\omega_B}$ by \Cref{lemma:quadrangulate,lemma:obtainsquaregrid}.
	Let $G$ be the cylindrical grid of order $3\omega_B$, then $\Split{G}=CG_{3\omega_B}$ and thus, by \cref{lemma:excludingantichains} $G$ must contain a butterfly minor $H$ which is a member of $\Antichain{D}$.
	
	For the reverse direction let us assume there is $H\in\Antichain{D}$ such that $H$ is a butterfly minor of the cylindrical grid.
	That means for some $k\in\N$, the cylindrical grid of order $k$, let us call it $G$, contains $H$ as a butterfly minor.
	By \cref{lemma:mcguigmatminors} this means that $\Split{G}=CG_k$ contains $\Split{H}$ as a matching minor.
	As $\Split{D}$ is a matching minor of $\Split{H}$ and $\Split{H}$ is a matching minor of a planar graph, $\Split{D}$ must be planar and therefore $D$ is strongly planar. 
\end{proof}

There is an immediate consequence of \cref{thm:directedgridminors} which we state without proof.
A proof for the matching theoretic analogue can be found in \cref{sec:EP}.

\begin{corollary}\label{cor:boundeddirectedtreewidth}
	Let $\mathcal{D}$ be a proper butterfly minor closed class of digraphs.
	Then $\mathcal{D}$ has bounded directed treewidth if and only if there exists a strongly connected strongly planar digraph $H$ such that no member of $\mathcal{D}$ contains a digraph from $\Antichain{H}$ as a butterfly minor.
\end{corollary}

\section{The Erd\H{o}s-P{\'o}sa Property for Matching Minors}\label{sec:EP}

In \cref{subsec:planarityandgrids} we present a proof of \cref{thm:boundedclasses} based on \cref{thm:matchinggridminors}.
We then continue to show that every bipartite and planar matching covered graph has the matching Erd\H{o}s-P\'osa property for matching minors by adapting techniques from the original proofs to the setting of matching covered bipartite graphs.
For the reverse however, the nature of the \textbf{matching} Erd\H{o}s-P\'osa property prevents us from doing the same.
Instead, in \cref{subsec:generalisedEP} we use our insight on fundamental anti-chains of butterfly minors to present a version of the Erd\H{o}s-P\'osa property that interacts with anti-chains instead of a single graph.
A nice pay off from this approach allows us to replace `butterfly minor of the cylindrical grid' by the purely topological condition of being \hyperref[def:stronggenus]{strongly planar}.

\subsection{Planarity and Grids}\label{subsec:planarityandgrids}

The most important step towards \cref{thm:boundedclasses} after the grid theorem itself is \cref{thm:matchinggridminors}.
To achieve this goal, we make use of the iterative construction for bipartite matching covered graphs in the form of \hyperref[thm:bipartiteeardecomposition]{ear decompositions}.
For this any \hyperref[def:singleear]{ear} we add to our graph will in fact be an internally $M$-conformal path for some perfect matching $M$.
Moreover, one can observe that any bipartite matching covered graph $B$ has an ear decomposition and a perfect matching $M$, such that the conformal cycle $B_2$ obtained from $K_2$ by adding the first ear is $M$-conformal in $B$, and every $B_i$ obtained from adding an additional ear $P$ has the property that $P$ is internally $M$-conformal.
Additionally, in case $B$ is bipartite, matching covered, and planar, we can choose an ear decomposition as above in such a way that $B_{i+1}$ can be drawn in the plane and the newly added ear is part of the boundary of a face.

Let us quickly introduce the necessary notions and results.

\begin{definition}[Ear]\label{def:singleear}
	Let $B$ be a bipartite graph with a perfect matching.
	A \emph{ear} is a path $P$ of odd length such that all internal vertices, if there are any, of $P$ have degree two in $B$.
	Let $M$ be a perfect matching of $G$.
	The path $P$ is a \emph{$M$-ear} if it is internally $M$-conformal and an ear.
\end{definition}

\begin{theorem}[Theorem 4.1.6 in \cite{lovasz2009matching}]\label{thm:bipartiteeardecomposition}
	Given any bipartite matching covered graph $B$, there exists a sequence $B_1\subset B_2\subset \dots \subset B_t$ 	of matching covered conformal subgraphs of $B$, such that $B_1=K_2$, $B_t=B$, and $B_{i+1}$ is obtained from $B_i$ by adding an ear of $B_{i+1}$ for all $i\in[1,t-1]$.
\end{theorem}

A sequence as in the above theorem is called a \emph{bipartite ear-decomposition}.

Towards \cref{thm:matchinggridminors} we first need a refined version of \cref{thm:matchinggrid}.

Let $CG_k$ be the cylindrical matching grid of order $k$.
The canonical \emph{internal quadrangulation} $CG_k^{\square}$ of $CG_k$ is defined as the graph obtained from the cylindrical grid by adding the following edges.
\begin{align*}
	\CondSet{v^i_{j}v^{i+1}_{j+1}}{i\in[1,k-1]\text{ and }j\in\Set{2,4,\dots,4k}}
\end{align*}
See \cref{fig:cylindricalgrid} for an illustration.

\begin{figure}
	\centering
	\begin{tikzpicture}[scale=0.75]

		\pgfdeclarelayer{background}
		\pgfdeclarelayer{foreground}
		\pgfsetlayers{background,main,foreground}

		\draw[e:main] (0,0) circle (11mm);
		\draw[e:main] (0,0) circle (16mm);
		\draw[e:main] (0,0) circle (21mm);
		\draw[e:main] (0,0) circle (26mm);
		\draw[e:main] (0,0) circle (31mm);
		\draw[e:main] (0,0) circle (36mm);
		\draw[e:main] (0,0) circle (41mm);
		\draw[e:main] (0,0) circle (46mm);
		\draw[e:main] (0,0) circle (51mm);
		\draw[e:main] (0,0) circle (56mm);
		\draw[e:main] (0,0) circle (61mm);
		\draw[e:main] (0,0) circle (66mm);

		\foreach \x in {1,...,12}
		{
			\draw[e:main] (\x*30:16mm) -- (\x*30-7.5:11mm);
			\draw[e:main] (\x*30:21mm) -- (\x*30-7.5:16mm);
			\draw[e:main] (\x*30:26mm) -- (\x*30-7.5:21mm);
			\draw[e:main] (\x*30:31mm) -- (\x*30-7.5:26mm);
			\draw[e:main] (\x*30:36mm) -- (\x*30-7.5:31mm);
			\draw[e:main] (\x*30:41mm) -- (\x*30-7.5:36mm);
			\draw[e:main] (\x*30:46mm) -- (\x*30-7.5:41mm);
			\draw[e:main] (\x*30:51mm) -- (\x*30-7.5:46mm);
			\draw[e:main] (\x*30:56mm) -- (\x*30-7.5:51mm);
			\draw[e:main] (\x*30:61mm) -- (\x*30-7.5:56mm);
			\draw[e:main] (\x*30:66mm) -- (\x*30-7.5:61mm);
		}
		
		\foreach \x in {1,...,12}
		{
			\draw[e:main] (\x*30+7.5:16mm) -- (\x*30+15:11mm);
			\draw[e:main] (\x*30+7.5:21mm) -- (\x*30+15:16mm);
			\draw[e:main] (\x*30+7.5:26mm) -- (\x*30+15:21mm);
			\draw[e:main] (\x*30+7.5:31mm) -- (\x*30+15:26mm);
			\draw[e:main] (\x*30+7.5:36mm) -- (\x*30+15:31mm);
			\draw[e:main] (\x*30+7.5:41mm) -- (\x*30+15:36mm);
			\draw[e:main] (\x*30+7.5:46mm) -- (\x*30+15:41mm);
			\draw[e:main] (\x*30+7.5:51mm) -- (\x*30+15:46mm);
			\draw[e:main] (\x*30+7.5:56mm) -- (\x*30+15:51mm);
			\draw[e:main] (\x*30+7.5:61mm) -- (\x*30+15:56mm);
			\draw[e:main] (\x*30+7.5:66mm) -- (\x*30+15:61mm);
		}
		
		\foreach \x in {1,...,8}
		{
			\draw[e:marker,color=BrightUbe,bend left =17.5] (\x*45:61mm) to (\x*45-37.5:61mm);
			\draw[e:marker,color=BrightUbe,bend left =17.5] (\x*45:46mm) to (\x*45-37.5:46mm);
			\draw[e:marker,color=BrightUbe,bend left =17.5] (\x*45:31mm) to (\x*45-37.5:31mm);
			\draw[e:marker,color=BrightUbe,bend left =17.5] (\x*45:16mm) to (\x*45-37.5:16mm);
			
		}
		
		\foreach \x in {1,...,4}
		{
			\draw[e:marker,color=BrightUbe] (\x*90-22.5:61mm) to (\x*90-15:56mm);
			\draw[e:marker,color=BrightUbe] (\x*90-15:56mm) to (\x*90-22.5:56mm);
			\draw[e:marker,color=BrightUbe] (\x*90+30:61mm) to (\x*90+22.5:56mm);
			\draw[e:marker,color=BrightUbe] (\x*90+22.5:56mm) to (\x*90+30:56mm);

			\draw[e:marker,color=BrightUbe] (\x*90+30:51mm) to (\x*90+22.5:46mm);
			\draw[e:marker,color=BrightUbe] (\x*90+22.5:51mm) to (\x*90+30:51mm);
			\draw[e:marker,color=BrightUbe] (\x*90-22.5:51mm) to (\x*90-15:46mm);
			\draw[e:marker,color=BrightUbe] (\x*90-15:51mm) to (\x*90-22.5:51mm);
			\draw[e:marker,color=BrightUbe] (\x*90-22.5:46mm) to (\x*90-15:41mm);
			\draw[e:marker,color=BrightUbe] (\x*90-15:41mm) to (\x*90-22.5:41mm);
			\draw[e:marker,color=BrightUbe] (\x*90+30:46mm) to (\x*90+22.5:41mm);
			\draw[e:marker,color=BrightUbe] (\x*90+22.5:41mm) to (\x*90+30:41mm);

			\draw[e:marker,color=BrightUbe] (\x*90+30:36mm) to (\x*90+22.5:31mm);
			\draw[e:marker,color=BrightUbe] (\x*90+22.5:36mm) to (\x*90+30:36mm);
			\draw[e:marker,color=BrightUbe] (\x*90-22.5:36mm) to (\x*90-15:31mm);
			\draw[e:marker,color=BrightUbe] (\x*90-15:36mm) to (\x*90-22.5:36mm);
			\draw[e:marker,color=BrightUbe] (\x*90-22.5:31mm) to (\x*90-15:26mm);
			\draw[e:marker,color=BrightUbe] (\x*90-15:26mm) to (\x*90-22.5:26mm);
			\draw[e:marker,color=BrightUbe] (\x*90+30:31mm) to (\x*90+22.5:26mm);
			\draw[e:marker,color=BrightUbe] (\x*90+22.5:26mm) to (\x*90+30:26mm);
			
			\draw[e:marker,color=BrightUbe] (\x*90+30:21mm) to (\x*90+22.5:16mm);
			\draw[e:marker,color=BrightUbe] (\x*90+22.5:21mm) to (\x*90+30:21mm);
			\draw[e:marker,color=BrightUbe] (\x*90-22.5:21mm) to (\x*90-15:16mm);
			\draw[e:marker,color=BrightUbe] (\x*90-15:21mm) to (\x*90-22.5:21mm);
		}
		
		\foreach \x in {1,...,4}
		{
			\draw[e:marker,color=AO,bend left=12.5,opacity=0.2] (\x*90-22.5:56mm) to (\x*90-52.5:56mm);
			\draw[e:marker,color=AO,opacity=0.2] (\x*90-52.5:56mm) to (\x*90-45:51mm);
			\draw[e:marker,color=AO,opacity=0.2] (\x*90-45:51mm) to (\x*90-52.5:51mm);
			\draw[e:marker,color=AO,opacity=0.2] (\x*90-52.5:51mm) to (\x*90-45:46mm);

			\draw[e:marker,color=AO,bend left=12.5,opacity=0.2] (\x*90-22.5:41mm) to (\x*90-52.5:41mm);
			\draw[e:marker,color=AO,opacity=0.2] (\x*90-52.5:41mm) to (\x*90-45:36mm);
			\draw[e:marker,color=AO,opacity=0.2] (\x*90-45:36mm) to (\x*90-52.5:36mm);
			\draw[e:marker,color=AO,opacity=0.2] (\x*90-52.5:36mm) to (\x*90-45:31mm);
			
			\draw[e:marker,color=AO,bend left=12.5,opacity=0.2] (\x*90-22.5:26mm) to (\x*90-52.5:26mm);
			\draw[e:marker,color=AO,opacity=0.2] (\x*90-52.5:26mm) to (\x*90-45:21mm);
			\draw[e:marker,color=AO,opacity=0.2] (\x*90-45:21mm) to (\x*90-52.5:21mm);
			\draw[e:marker,color=AO,opacity=0.2] (\x*90-52.5:21mm) to (\x*90-45:16mm);

			\draw[e:marker,color=AO,bend left=12.5,opacity=0.2] (\x*90+15:51mm) to (\x*90-15:51mm);
			\draw[e:marker,color=AO,opacity=0.2] (\x*90+15:51mm) to (\x*90+7.5:56mm);
			\draw[e:marker,color=AO,opacity=0.2] (\x*90+7.5:56mm) to (\x*90+15:56mm);
			\draw[e:marker,color=AO,opacity=0.2] (\x*90+15:56mm) to (\x*90+7.5:61mm);

			\draw[e:marker,color=AO,bend left=12.5,opacity=0.2] (\x*90+15:36mm) to (\x*90-15:36mm);
			\draw[e:marker,color=AO,opacity=0.2] (\x*90+15:36mm) to (\x*90+7.5:41mm);
			\draw[e:marker,color=AO,opacity=0.2] (\x*90+7.5:41mm) to (\x*90+15:41mm);
			\draw[e:marker,color=AO,opacity=0.2] (\x*90+15:41mm) to (\x*90+7.5:46mm);

			\draw[e:marker,color=AO,bend left=12.5,opacity=0.2] (\x*90+15:21mm) to (\x*90-15:21mm);
			\draw[e:marker,color=AO,opacity=0.2] (\x*90+15:21mm) to (\x*90+7.5:26mm);
			\draw[e:marker,color=AO,opacity=0.2] (\x*90+7.5:26mm) to (\x*90+15:26mm);
			\draw[e:marker,color=AO,opacity=0.2] (\x*90+15:26mm) to (\x*90+7.5:31mm);
		}

		\foreach \x in {1,...,24}
		{
			\draw[e:coloredthin,color=BostonUniversityRed,bend left=6] (\x*15:11mm) to (\x*15-7.5:11mm);
			\draw[e:coloredthin,color=BostonUniversityRed,bend left=6] (\x*15:16mm) to (\x*15-7.5:16mm);
			\draw[e:coloredthin,color=BostonUniversityRed,bend left=6] (\x*15:21mm) to (\x*15-7.5:21mm);
			\draw[e:coloredthin,color=BostonUniversityRed,bend left=6] (\x*15:26mm) to (\x*15-7.5:26mm);
			\draw[e:coloredthin,color=BostonUniversityRed,bend left=6] (\x*15:31mm) to (\x*15-7.5:31mm);
			\draw[e:coloredthin,color=BostonUniversityRed,bend left=6] (\x*15:36mm) to (\x*15-7.5:36mm);
			\draw[e:coloredthin,color=BostonUniversityRed,bend left=5] (\x*15:41mm) to (\x*15-7.5:41mm);
			\draw[e:coloredthin,color=BostonUniversityRed,bend left=5] (\x*15:46mm) to (\x*15-7.5:46mm);
			\draw[e:coloredthin,color=BostonUniversityRed,bend left=5] (\x*15:51mm) to (\x*15-7.5:51mm);
			\draw[e:coloredthin,color=BostonUniversityRed,bend left=5] (\x*15:56mm) to (\x*15-7.5:56mm);
			\draw[e:coloredthin,color=BostonUniversityRed,bend left=5] (\x*15:61mm) to (\x*15-7.5:61mm);
			\draw[e:coloredthin,color=BostonUniversityRed,bend left=5] (\x*15:66mm) to (\x*15-7.5:66mm);
		}
		
		\foreach \x in {1,...,24}
		{
			\node[v:mainempty] () at (\x*15:11mm){};
			\node[v:mainempty] () at (\x*15:16mm){};
			\node[v:mainempty] () at (\x*15:21mm){};
			\node[v:mainempty] () at (\x*15:26mm){};
			\node[v:mainempty] () at (\x*15:31mm){};
			\node[v:mainempty] () at (\x*15:36mm){};
			\node[v:mainempty] () at (\x*15:41mm){};
			\node[v:mainempty] () at (\x*15:46mm){};
			\node[v:mainempty] () at (\x*15:51mm){};
			\node[v:mainempty] () at (\x*15:56mm){};
			\node[v:mainempty] () at (\x*15:61mm){};
			\node[v:mainempty] () at (\x*15:66mm){};
			
			\node[v:main] () at (\x*15+7.5:11mm){};
			\node[v:main] () at (\x*15+7.5:16mm){};
			\node[v:main] () at (\x*15+7.5:21mm){};
			\node[v:main] () at (\x*15+7.5:26mm){};
			\node[v:main] () at (\x*15+7.5:31mm){};
			\node[v:main] () at (\x*15+7.5:36mm){};
			\node[v:main] () at (\x*15+7.5:41mm){};
			\node[v:main] () at (\x*15+7.5:46mm){};
			\node[v:main] () at (\x*15+7.5:51mm){};
			\node[v:main] () at (\x*15+7.5:56mm){};
			\node[v:main] () at (\x*15+7.5:61mm){};
			\node[v:main] () at (\x*15+7.5:66mm){};
		}

		\begin{pgfonlayer}{background}
			\foreach \x in {1,...,8}
			{
				\foreach \x in {1,...,24}
				{
					\draw[e:coloredborder,bend left=6] (\x*15:11mm) to (\x*15-7.5:11mm);
					\draw[e:coloredborder,bend left=6] (\x*15:16mm) to (\x*15-7.5:16mm);
					\draw[e:coloredborder,bend left=6] (\x*15:21mm) to (\x*15-7.5:21mm);
					\draw[e:coloredborder,bend left=6] (\x*15:26mm) to (\x*15-7.5:26mm);
					\draw[e:coloredborder,bend left=6] (\x*15:31mm) to (\x*15-7.5:31mm);
					\draw[e:coloredborder,bend left=6] (\x*15:36mm) to (\x*15-7.5:36mm);
					\draw[e:coloredborder,bend left=5] (\x*15:41mm) to (\x*15-7.5:41mm);
					\draw[e:coloredborder,bend left=5] (\x*15:46mm) to (\x*15-7.5:46mm);
					\draw[e:coloredborder,bend left=5] (\x*15:51mm) to (\x*15-7.5:51mm);
					\draw[e:coloredborder,bend left=5] (\x*15:56mm) to (\x*15-7.5:56mm);
					\draw[e:coloredborder,bend left=5] (\x*15:61mm) to (\x*15-7.5:61mm);
					\draw[e:coloredborder,bend left=5] (\x*15:66mm) to (\x*15-7.5:66mm);
				}
			}
		\end{pgfonlayer}
	\end{tikzpicture}
	\caption{The canonical internal quadrangulation of $CG_4$ as a matching minor of $CG_{12}$.}
	\label{fig:quadrangulation}
\end{figure}
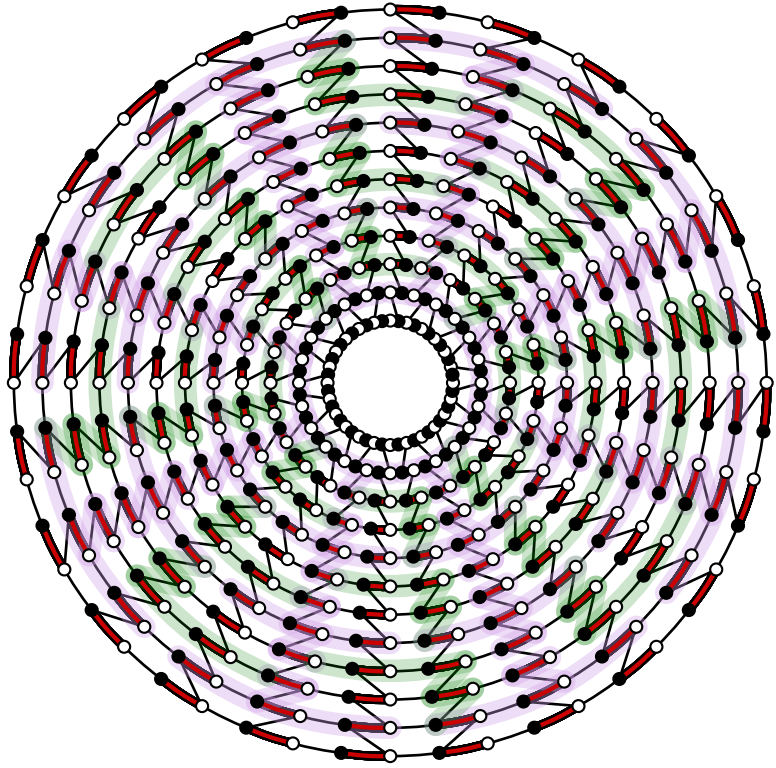

\begin{lemma}\label{lemma:quadrangulate}
	Let $k\in\N$ be a positive integer.
	The cylindrical matching grid $CG_{3k}$ contains $CG_k^{\square}$ as a matching minor.
\end{lemma}

\begin{proof}
	We describe how to create a model $\mu\colon CG_k^{\square}\rightarrow CG_{3k}$.
	
	First let $i\in\CondSet{3\ell-1}{1\leq\ell\leq k}$ and $j\in[1,k]$, we define models for four vertices $a^i_{j,\text{down}}$, $b^i_{j,\text{up}}$, $a^i_{j,\text{up}}$, and $b^i_{j,\text{down}}$.
	See \cref{fig:quadrangulation} for an illustration on how the models of these vertices will be arranged in $CG_{3k}$.
	
	\begin{align*}
		\Fkt{\mu}{a^i_{j,\text{down}}} \coloneqq & \Brace{v^i_{1+12\Brace{j-1}},v^i_{2+12\Brace{j-1}},v^i_{3+12\Brace{j-1}}}\cup\Brace{v^{i-1}_{3+12\Brace{j-1}},v^{i-1}_{4+12\Brace{j-1}},v^{i}_{3+12\Brace{j-1}}}\\
		\Fkt{\mu}{b^i_{j,\text{up}}} \coloneqq & \Brace{v^i_{4+12\Brace{j-1}},v^i_{5+12\Brace{j-1}},v^i_{6+12\Brace{j-1}}}\cup\Brace{v^{i}_{4+12\Brace{j-1}},v^{i+1}_{3+12\Brace{j-1}},v^{i+1}_{4+12\Brace{j-1}}}\\
		\Fkt{\mu}{a^i_{j,\text{up}}} \coloneqq & \Brace{v^i_{7+12\Brace{j-1}},v^i_{8+12\Brace{j-1}},v^i_{9+12\Brace{j-1}}}\cup \Brace{v^{i}_{9+12\Brace{j-1}},v^{i+1}_{10+12\Brace{j-1}},v^{i+1}_{9+12\Brace{j-1}}}\\
		\Fkt{\mu}{b^i_{j,\text{down}}} \coloneqq & \Brace{v^i_{10+12\Brace{j-1}},v^i_{11+12\Brace{j-1}},v^i_{12+12\Brace{j-1}}}\cup  \Brace{v^{i-1}_{10+12\Brace{j-1}},v^{i-1}_{9+12\Brace{j-1}},v^{i-1}_{10+12\Brace{j-1}}}
	\end{align*}
	As a next step we add models for the edges of the concentric cycles of $CG_k^{\square}$, here we identify $b^{i}_{0,\text{down}}$ with $b^{i}_{k,\text{down}}$ and $v^i_0$ with $v^i_{12k}$.
	\begin{align*}
		\Fkt{\mu}{b^i_{j-1,\text{down}}a^i_{j,\text{down}}}\coloneqq & v^i_{12+12\Brace{j-2}}v^i_{1+12\Brace{j-1}}\\
		\Fkt{\mu}{a^i_{j,\text{down}}b^i_{j,\text{up}}}\coloneqq & v^i_{3+12\Brace{j-1}}v^i_{4+12\Brace{j-1}}\\
		\Fkt{\mu}{b^i_{j,\text{up}}a^i_{j,\text{up}}}\coloneqq & v^i_{6+12\Brace{j-1}}v^i_{7+12\Brace{j-1}}\\
		\Fkt{\mu}{a^i_{j,\text{up}}b^i_{j,\text{down}}}\coloneqq & v^i_{9+12\Brace{j-1}}v^i_{10+12\Brace{j-1}}
	\end{align*}
	This in particular means that $C_i\subseteq\Fkt{\mu}{CG_k^{\square}}$ for all $i\in\CondSet{3\ell-1}{\ell\in[1,k]}$.
	Next we connect the cycles, so let $i\in\CondSet{3\ell-1}{\ell\in[1,k-1]}$.
	\begin{align*}
		\Fkt{\mu}{b^i_{j,\text{up}}a^{i+3}_{j,\text{down}}}\coloneqq &v^{i+1}_{4+12\Brace{j-1}}v^{i+2}_{3+12\Brace{j-1}}\\
		\Fkt{\mu}{a^i_{j,\text{up}}b^{i+3}_{j,\text{down}}}\coloneqq &v^{i+1}_{9+12\Brace{j-1}}v^{i+2}_{10+12\Brace{j-1}}
	\end{align*}
	
	Compare \cref{fig:quadrangulation,fig:updown} to see how our model $\mu$ so far describes $CG_k$ as a matching minor of $CG_{3k}$.
	
	\begin{figure}[!h]
		\centering
		\begin{tikzpicture}[scale=0.9]
			\pgfdeclarelayer{background}
			\pgfdeclarelayer{foreground}
			\pgfsetlayers{background,main,foreground}
			
			\foreach \x in {1,...,5}
			{
				\draw[e:main] (4.5,\x) -- (2,\x+1);
				\draw[e:main] (7,\x) -- (9.5,\x+1);
			}
			
			\foreach \x in {0.5,1.5,5,4.5,5.5,6.5,9.5,10.5}
			{
				\draw[e:main] (\x,2) -- (\x+0.5,2);
				\draw[e:main] (\x,5) -- (\x+0.5,5);
			}
			\foreach \x in {2,7}
			{
				\foreach \y in {1,...,6}
				{
					\draw[e:coloredthin,color=BostonUniversityRed] (\x,\y) -- (\x+2.5,\y);
				}
			}
			\foreach \x in {1,5,6,10}
			{
				\draw[e:coloredthin,color=BostonUniversityRed] (\x,2) -- (\x+0.5,2);
				\draw[e:coloredthin,color=BostonUniversityRed] (\x,5) -- (\x+0.5,5);
			}
			\begin{pgfonlayer}{background}
				\foreach \x in {2,7}
				{
					\foreach \y in {1,...,6}
					{
						\draw[e:coloredborder] (\x,\y) -- (\x+2.5,\y);
					}
				}
				\foreach \x in {1,5,6,10}
				{
					\draw[e:coloredborder] (\x,2) -- (\x+0.5,2);
					\draw[e:coloredborder] (\x,5) -- (\x+0.5,5);
				}
			\end{pgfonlayer}
			
			\draw[e:marker,color=BrightUbe] (4.5,6) -- (2,6);
			\draw[e:marker,color=BrightUbe] (2,6) -- (4.5,5) -- (5.5,5);
			\draw[e:marker,color=BrightUbe] (10.5,5) -- (9.5,5) -- (7,4);
			\draw[e:marker,color=BrightUbe] (7,4) -- (9.5,4);
			\draw[e:marker,color=BrightUbe] (4.5,3) -- (2,3);
			\draw[e:marker,color=BrightUbe] (2,3)-- (4.5,2) -- (5.5,2);
			\draw[e:marker,color=BrightUbe] (10.5,2) -- (9.5,2) -- (7,1);
			\draw[e:marker,color=BrightUbe] (7,1) -- (9.5,1);
			
			\draw[e:marker,color=BrightUbe] (7,6) -- (9.5,6);
			\draw[e:marker,color=BrightUbe] (9.5,6)-- (7,5) -- (6,5);
			\draw[e:marker,color=BrightUbe] (1,5) -- (2,5) -- (4.5,4);
			\draw[e:marker,color=BrightUbe] (4.5,4) -- (2,4); 
			\draw[e:marker,color=BrightUbe] (7,3) -- (9.5,3);
			\draw[e:marker,color=BrightUbe] (9.5,3) -- (7,2) -- (6,2);
			\draw[e:marker,color=BrightUbe] (1,2) -- (2,2) -- (4.5,1);
			\draw[e:marker,color=BrightUbe] (4.5,1) -- (2,1);

			\node[] at (2.5,6.4){$v^{i+4}_{10+12(j-1)}$};
			\node[] at (2.5,4.4){$v^{i+2}_{10+12(j-1)}$};
			\node[] at (2.5,3.4){$v^{i+1}_{10+12(j-1)}$};
			\node[] at (2.5,0.6){$v^{i-1}_{10+12(j-1)}$};
			\node[] at (5.5,6){$v^{i+4}_{9+12(j-1)}$};
			\node[] at (5.5,4){$v^{i+2}_{9+12(j-1)}$};
			\node[] at (5.5,3){$v^{i+1}_{9+12(j-1)}$};
			\node[] at (5.5,1){$v^{i-1}_{9+12(j-1)}$};
			\node[] at (1.5,5.4){$b^{i+3}_{j\text{,down}}$}; 
			\node[] at (1.5,2.4){$b^{i}_{j\text{,down}}$}; 
			\node[] at (5.5,5.4){$a^{i+3}_{j\text{,up}}$};
			\node[] at (5.5,2.4){$a^{i}_{j\text{,up}}$};
			
			\node[] at (7.4,6.4){$v^{i+4}_{4+12(j-1)}$};
			\node[] at (7.4,3.6){$v^{i+2}_{4+12(j-1)}$};
			\node[] at (7.4,2.6){$v^{i+1}_{4+12(j-1)}$};
			\node[] at (7.4,0.6){$v^{i-1}_{4+12(j-1)}$};
			\node[] at (10,6.4){$v^{i+4}_{3+12(j-1)}$};
			\node[] at (10,3.6){$v^{i+2}_{3+12(j-1)}$};
			\node[] at (10,2.6){$v^{i+1}_{3+12(j-1)}$};
			\node[] at (10,0.6){$v^{i-1}_{3+12(j-1)}$};
			\node[] at (6.5,4.6){$b^{i+3}_{j\text{,up}}$}; 
			\node[] at (6.5,1.6){$b^{i}_{j\text{,up}}$}; 
			\node[] at (10,5.4){$a^{i+3}_{j\text{,down}}$};
			\node[] at (10.5,1.6){$a^{i}_{j\text{,down}}$}; 
			
			\foreach \x in {2,7}
			{
				\node[v:mainempty] at (\x,1){};
				\node[v:mainempty] at (\x,2){};
				\node[v:mainempty] at (\x,3){};
				\node[v:mainempty] at (\x,4){};
				\node[v:mainempty] at (\x,5){};
				\node[v:mainempty] at (\x,6){};
			}
			\foreach \x in {4.5,9.5}
			{
				\node[v:main] at (\x,1){};
				\node[v:main] at (\x,2){};
				\node[v:main] at (\x,3){};
				\node[v:main] at (\x,4){};
				\node[v:main] at (\x,5){};
				\node[v:main] at (\x,6){};
			}
			\foreach \x in {1,5,6,10}
			{
				\node[v:mainempty] at (\x,2){};
				\node[v:mainempty] at (\x,5){};
				\node[v:main] at (\x+0.5,2){};
				\node[v:main] at (\x+0.5,5){};
			}
		\end{tikzpicture}
		\caption{The situation of the up- and down- vertices in the model of $CG_k^{\square}$ and the models of the edges of its $CG_k$-subgraph.}
		\label{fig:updown}
	\end{figure}
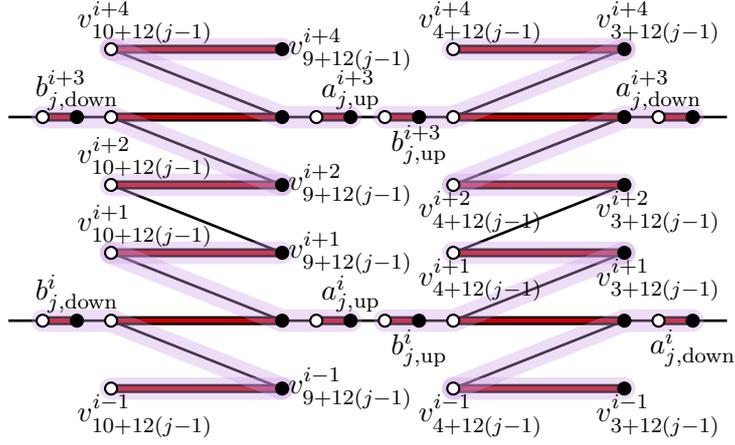
	
	In the next step we describe how to build the models for the new edges that we need to add to our $CG_k$ to form the canonical inner quadrangulation.
	We identify $b^i_{0,\text{down}}$ and $b^i_{k,\text{down}}$.
	\begin{align*}
		\Fkt{\mu}{b^i_{j,\text{up}}a^{i+3}_{j,\text{up}}}\coloneqq & \Brace{v^{i+1}_{4+12\Brace{j-1}},\dots,v^{i+1}_{8+12\Brace{j-1}},v^{i+2}_{7+12\Brace{j-1}},v^{i+2}_{8+12\Brace{j-1}},v^{i+3}_{7+12\Brace{j-1}}}\\
		\Fkt{\mu}{b^i_{j-1,\text{down}}a^{i+3}_{j,\text{down}}}\coloneqq & \Brace{v^{i}_{12+12\Brace{j-2}},v^{i+1}_{11+12\Brace{j-2}},v^{i+1}_{12+12\Brace{j-2}},v^{i+2}_{11+12\Brace{j-2}},\dots,v^{i+2}_{3+12\Brace{j-1}}}
	\end{align*}
	
	For an illustration compare \cref{fig:crossconnection}.
	Moreover, from the construction, it is clear that the model of every edge of $CG_k^{\square}$ is an internally $M$-conformal path, where $M$ is the canonical matching of $CG_{3k}$.
	Also, the model of every vertex is a barycentric tree with exactly one exposed vertex.
	So in total $\mu$ is a matching minor model of $CG_k^{\square}$ in $CG_{3k}$.
\end{proof}

\begin{figure}[!h]
	\centering
	\begin{tikzpicture}[scale=0.9]
		
		\pgfdeclarelayer{background}
		\pgfdeclarelayer{foreground}
		\pgfsetlayers{background,main,foreground}
		
		\foreach \x in {1,3,...,13}
		{
			\draw[e:coloredthin,color=BostonUniversityRed] (\x,1) -- (\x+1,1); 
			\draw[e:coloredthin,color=BostonUniversityRed] (\x,2.2) -- (\x+1,2.2); 
			\draw[e:coloredthin,color=BostonUniversityRed] (\x,3.4) -- (\x+1,3.4); 
			\draw[e:coloredthin,color=BostonUniversityRed] (\x,4.6) -- (\x+1,4.6); 
		}
		
		\begin{pgfonlayer}{background}
			\foreach \x in {1,3,...,13}
			{
				\draw[e:coloredborder] (\x,1) -- (\x+1,1);
				\draw[e:coloredborder] (\x,2.2) -- (\x+1,2.2); 
				\draw[e:coloredborder] (\x,3.4) -- (\x+1,3.4); 
				\draw[e:coloredborder] (\x,4.6) -- (\x+1,4.6); 
			}
			
			\draw[e:marker,color=BrightUbe] (1.7,4.9) -- (2,4.6);
			\draw[e:marker,color=BrightUbe] (2,4.6) -- (4,4.6);
			\draw[e:marker,color=BrightUbe] (10,4.6) -- (8,4.6);
			\draw[e:marker,color=BrightUbe] (8,4.6) -- (7,3.4);
			\draw[e:marker,color=BrightUbe] (7,3.4)-- (8,3.4);
			\draw[e:marker,color=BrightUbe] (5,1) -- (7,1);
			\draw[e:marker,color=BrightUbe] (7,1) -- (8,2.2);
			\draw[e:marker,color=BrightUbe] (8,2.2) -- (7,2.2);
			\draw[e:marker,color=BrightUbe] (11,1) -- (13,1);
			\draw[e:marker,color=BrightUbe] (13,1) -- (13.3,0.7);
			
			\draw[e:marker,color=AO] (4,4.6) -- (3,3.4);
			\draw[e:marker,color=AO] (3,3.4) -- (4,3.4);
			\draw[e:marker,color=AO] (4,3.4) -- (3,2.2);
			\draw[e:marker,color=AO] (3,2.2) -- (7,2.2);
			\draw[e:marker,color=AO] (8,3.4) -- (12,3.4);
			\draw[e:marker,color=AO] (12,3.4) -- (11,2.2);
			\draw[e:marker,color=AO] (11,2.2) -- (12,2.2);
			\draw[e:marker,color=AO] (12,2.2) -- (11,1);
		\end{pgfonlayer}
		
		\foreach \x in {2,4,...,12}
		{
			\draw[e:main] (\x,1) -- (\x+1,1);
			\draw[e:main] (\x,4.6) -- (\x+1,4.6);
		}
		\draw[e:main] (4,2.2) -- (5,2.2);
		\draw[e:main] (6,2.2) -- (7,2.2);
		\draw[e:main] (8,3.4) -- (9,3.4);
		\draw[e:main] (10,3.4) -- (11,3.4);
		\draw[e:main] (0.5,1) -- (1,1);
		\draw[e:main] (0.5,4.6) -- (1,4.6);
		\draw[e:main] (14,1) -- (14.5,1);
		\draw[e:main] (14,4.6) -- (14.5,4.6);
		\draw[e:main] (2,1) -- (1,2.2);
		\draw[e:main] (7,1) -- (8,2.2);
		\draw[e:main] (11,1) -- (12,2.2);
		\draw[e:main] (14,1) -- (13,2.2);
		\draw[e:main] (2,2.2) -- (1,3.4);
		\draw[e:main] (3,2.2) -- (4,3.4);
		\draw[e:main] (7,2.2) -- (8,3.4);
		\draw[e:main] (11,2.2) -- (12,3.4);
		\draw[e:main] (14,2.2) -- (13,3.4);
		\draw[e:main] (2,3.4) -- (1,4.6);
		\draw[e:main] (3,3.4) -- (4,4.6);
		\draw[e:main] (7,3.4) -- (8,4.6);
		\draw[e:main] (14,3.4) -- (13,4.6);
		
		\node[scale=0.8] at (4.65,4.85){$v^{i+3}_{7+12(j-1)}$};
		\node[scale=0.8] at (2.5,5){$\mu(a^{i+3}_{j,\text{up}})$};
		\node[scale=0.8] at (9,5){$\mu(a^{i+3}_{j,\text{down}})$};
		\node[scale=0.8] at (2.9,3.25){$v^{i+2}_{8+12(j-1)}$};
		\node[scale=0.8] at (4.5,3.7){$v^{i+2}_{7+12(j-1)}$};
		\node[scale=0.8] at (8.5,3.7){$v^{i+2}_{3+12(j-1)}$};
		\node[scale=0.8] at (10.5,3.7){$v^{i+2}_{1+12(j-1)}$};
		\node[scale=0.8] at (12.5,3.7){$v^{i+2}_{11+12(j-2)}$};
		\node[scale=0.8] at (9.5,3.1){$v^{i+2}_{2+12(j-1)}$};
		\node[scale=0.8,circle] at (11.5,3){$v^{i+2}_{12+12(j-2)}$};
		\node[scale=0.8] at (11,2.05){$v^{i+1}_{12+12(j-2)}$};
		\node[scale=0.8] at (2.9,2.05){$v^{i+1}_{8+12(j-1)}$};
		\node[scale=0.8] at (5.5,2.5){$v^{i+1}_{6+12(j-1)}$};
		\node[scale=0.8] at (4.5,1.8){$v^{i+1}_{7+12(j-1)}$};
		\node[scale=0.8] at (6.5,1.8){$v^{i+1}_{5+12(j-1)}$};
		\node[scale=0.8] at (12.5,2.5){$v^{i+1}_{11+12(j-2)}$};
		\node[scale=0.8] at (11.5,0.6){$v^{i}_{12+12(j-2)}$};
		\node[scale=0.8] at (6,0.6){$\mu(b^{i}_{j,\text{up}})$};
		\node[scale=0.8] at (12.5,1.3){$\mu(b^{i}_{(j-1),\text{down}})$};
		
		\foreach \x in {1,3,...,13}
		{
			\node[v:mainempty] at (\x,1){};
			\node[v:mainempty] at (\x,2.2){};
			\node[v:mainempty] at (\x,3.4){};
			\node[v:mainempty] at (\x,4.6){};
		}
		\foreach \x in {2,4,...,14}
		{
			\node[v:main] at (\x,1){};
			\node[v:main] at (\x,2.2){};
			\node[v:main] at (\x,3.4){};
			\node[v:main] at (\x,4.6){};
		}
		
	\end{tikzpicture}
	\caption{The models for the new edges forming the inner quadrangulation of $CG_k$.}
	\label{fig:crossconnection}
\end{figure}
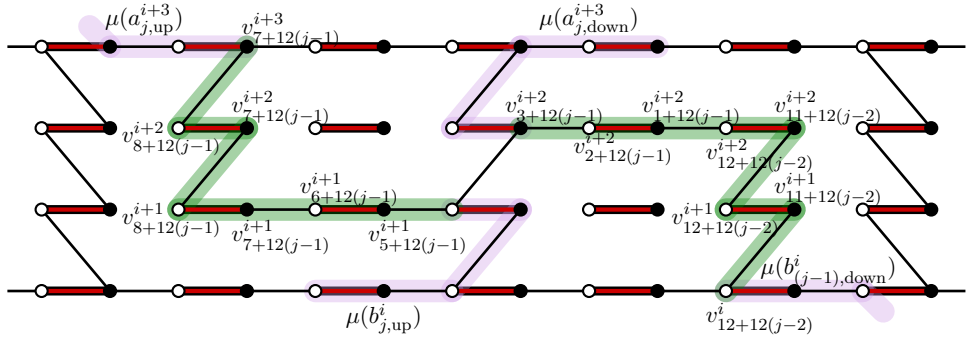

At last we find the $2k\times 2k$-grid as a matching minor in an inner quadragulation of a cylindrical matching grid of appropriate size.

\begin{lemma}\label{lemma:obtainsquaregrid}
	If $k\in\N$ is even, $CG_k^{\square}$ contains the $k\times k$-grid as a matching minor.
\end{lemma}

\begin{proof}
	Let $M$ be the canonical perfect matching of $CG_k^{\square}$ and let $V_1$, $V_2$ be the two colour classes of $CK_k^{\square}$ such that $v^1_1\in V_1$.
	We create a new perfect matching $M'$ for $CG_k^{\square}$ by ``switching'' $M$ along every second of the concentric cycles.
	Formally let
	\begin{align*}
		M'\coloneqq \Brace{M\setminus\bigcup_{i=2\text{, even}}^k\Fkt{E}{C_i}}\cup\bigcup_{i=2\text{, even}}^k\Brace{\Fkt{E}{C_i}\setminus M}.
	\end{align*}
	In the following, we describe how to construct a matching minor model of the $k\times k$-grid iteratively from a $C_4$ in $CG_k^{\square}$ by extending the model by small building blocks.
	A \emph{piece} is one of the following three configurations:
	\begin{itemize}
		\item A \emph{base piece} $B_{i,j}$ starting on the vertex $v_j^i\in V_2$.
		It consists of the two paths 
		\begin{align*}
			&\Brace{v^i_j,v^i_{j+1},v^i_{j+2},v^i_{j+3},v^i_{j+4}} \text{ and}\\ &\Brace{v^{i+1}_{j+1},v^{i+1}_{j+2},v^{i+1}_{j+3},v^{i+1}_{j+4},v^{i+1}_{j+5}}
		\end{align*}
		together with the edges $v^{i}_{j}v^{i+1}_{j+1}$, $v^{i}_{j+3}v^{i+1}_{j+4}$, and $v^{i}_{j+4}v^{i+1}_{j+5}$.
		
		\item A \emph{width piece} $W_{i,j}$ starting on the vertex $v_j^i\in V_1$.
		It consists of the three paths
		\begin{align*}
			&\Brace{v^i_j,v^i_{j+1},v^i_{j+2},v^i_{j+3},v^i_{j+4}},\\
			&\Brace{v^{i+1}_{j+1},v^{i+1}_{j+2},v^{i+1}_{j+3},v^{i+1}_{j+4},v^{i+1}_{j+5}}\text{, and}\\
			&\Brace{v^{i+2}_{j+2},v^{i+2}_{j+3},v^{i+2}_{j+4},v^{i+2}_{j+5},v^{i+2}_{j+6}}
		\end{align*}
		together with the edges $v^{i}_{j}v^{i+1}_{j+1}$, $v^{i}_{j+1}v^{i+1}_{j+2}$, $v^{i}_{j+4}v^{i+1}_{j+5}$, $v^{i+1}_{j+1}v^{i+2}_{j+2}$, $v^{i+1}_{j+4}v^{i+2}_{j+5}$, and $v^{i+1}_{j+5}v^{i+2}_{j+6}$.
		\item And a \emph{height piece} $H_{i,j}$ starting on the vertex $v_j^i\in V_2$.
		It consists of the three paths
		\begin{align*}
			&\Brace{v^i_j,v^i_{j+1},v^i_{j+2},v^i_{j+3},v^i_{j+4},v^i_{j+5},v^i_{j+6},v^i_{j+7}},\\
			&\Brace{v^{i+1}_{j+1},v^{i+1}_{j+2},v^{i+1}_{j+3},v^{i+1}_{j+4},v^{i+1}_{j+5},v^{i+1}_{j+6},v^{i+1}_{j+7},v^{i+1}_{j+8}}\text{, and}\\
			&\Brace{v^{i+2}_{j+4},v^{i+2}_{j+5},v^{i+2}_{j+6},v^{i+2}_{j+7},v^{i+2}_{j+8},v^{i+2}_{j+9}}
		\end{align*}
		together with the edges $v^{i}_{j}v^{i+1}_{j+1}$, $v^{i}_{j+3}v^{i+1}_{j+4}$, $v^{i}_{j+4}v^{i+1}_{j+5}$, $v^{i}_{j+7}v^{i+1}_{j+8}$, $v^{i+1}_{j+3}v^{i+2}_{j+4}$, $v^{i+1}_{j+4}v^{i+2}_{j+5}$, $v^{i+1}_{j+7}v^{i+2}_{j+8}$, and $v^{i+1}_{j+8}v^{i+2}_{j+9}$.
	\end{itemize}
	As a first step, we show how to create a model of the $4\times 4$-grid from a specific $C_4$ in $CG_4^{\square}$, see \cref{fig:squaregrid} for an illustration.
	
	\begin{figure}[!h]
		\begin{subfigure}{0.5\textwidth}
			\centering
			\begin{tikzpicture}[scale=0.9]

				\pgfdeclarelayer{background}
				\pgfdeclarelayer{foreground}
				\pgfsetlayers{background,main,foreground}

				\draw[e:main] (0,0) circle (11mm);
				\draw[e:main] (0,0) circle (16mm);
				\draw[e:main] (0,0) circle (21mm);
				\draw[e:main] (0,0) circle (26mm);
				
				\foreach \x in {1,...,4}
				{
					\draw[e:main] (\x*90:16mm) -- (\x*90+22.5:11mm);
					\draw[e:main] (\x*90:21mm) -- (\x*90+22.5:16mm);
					\draw[e:main] (\x*90:26mm) -- (\x*90+22.5:21mm);
				}
				
				\foreach \x in {1,...,4}
				{
					\draw[e:main] (\x*90-22.5:16mm) -- (\x*90-45:11mm);
					\draw[e:main] (\x*90-22.5:21mm) -- (\x*90-45:16mm);
					\draw[e:main] (\x*90-22.5:26mm) -- (\x*90-45:21mm);
				}
				
				\foreach \x in {1,...,8}
				{
					\draw[magenta,e:main] (\x*45:16mm) -- (\x*45-22.5:11mm);
					\draw[magenta,e:main] (\x*45:21mm) -- (\x*45-22.5:16mm);
					\draw[magenta,e:main] (\x*45:26mm) -- (\x*45-22.5:21mm);
					
				}
				
				\foreach \x in {1,...,8}
				{
					\draw[e:coloredthin,color=BostonUniversityRed,bend right=13] (\x*45:11mm) to (\x*45+22.5:11mm);
					\draw[e:coloredthin,color=AO,bend right=13] (\x*45-22.5:16mm) to (\x*45:16mm);
					\draw[e:coloredthin,color=BostonUniversityRed,bend right=13] (\x*45:21mm) to (\x*45+22.5:21mm);
					\draw[e:coloredthin,color=AO,bend right=13] (\x*45-22.5:26mm) to (\x*45:26mm);
				}
				
				\foreach \x in {1,...,8}
				{
					\node[v:main] () at (\x*45:11mm){};
					\node[v:main] () at (\x*45:16mm){};
					\node[v:main] () at (\x*45:21mm){};
					\node[v:main] () at (\x*45:26mm){};
					\node[v:mainempty] () at (\x*45+22.5:11mm){};
					\node[v:mainempty] () at (\x*45+22.5:16mm){};
					\node[v:mainempty] () at (\x*45+22.5:21mm){};
					\node[v:mainempty] () at (\x*45+22.5:26mm){};
				}
				
				\begin{pgfonlayer}{background}
					\foreach \x in {1,...,8}
					{
						\draw[e:coloredborder,bend right=13] (\x*45:11mm) to (\x*45+22.5:11mm);
						\draw[e:coloredborder,bend right=13] (\x*45-22.5:16mm) to (\x*45:16mm);
						\draw[e:coloredborder,bend right=13] (\x*45:21mm) to (\x*45+22.5:21mm);
						\draw[e:coloredborder,bend right=13] (\x*45-22.5:26mm) to (\x*45:26mm);
					}
					
					\draw[e:marker,color=DarkGoldenrod] (45:11mm) arc (45:157.5:11mm);
					\draw[e:marker,color=DarkGoldenrod] (67.5:16mm) arc (67.5:225:16mm);
					\draw[e:marker,color=DarkGoldenrod] (90:21mm) arc (90:247.5:21mm);
					\draw[e:marker,color=DarkGoldenrod] (157.5:26mm) arc (157.5:270:26mm);
					
					\draw[e:marker,color=DarkGoldenrod] (45:11mm) to (67.5:16mm);
					\draw[e:marker,color=DarkGoldenrod] (67.5:11mm) to (90:16mm);
					\draw[e:marker,color=DarkGoldenrod] (67.5:16mm) to (90:21mm);
					
					\draw[e:marker,color=DarkGoldenrod] (135:11mm) to (157.5:16mm);
					\draw[e:marker,color=DarkGoldenrod] (157.5:11mm) to (180:16mm);
					\draw[e:marker,color=DarkGoldenrod] (157.5:16mm) to (180:21mm);
					\draw[e:marker,color=DarkGoldenrod] (135:16mm) to (157.5:21mm);
					\draw[e:marker,color=DarkGoldenrod] (157.5:21mm) to (180:26mm);
					\draw[e:marker,color=DarkGoldenrod] (135:21mm) to (157.5:26mm);
					
					\draw[e:marker,color=DarkGoldenrod] (225:16mm) to (247.5:21mm);
					\draw[e:marker,color=DarkGoldenrod] (225:21mm) to (247.5:26mm);
					\draw[e:marker,color=DarkGoldenrod] (247.5:21mm) to (270:26mm);
					
				\end{pgfonlayer}
				
			\end{tikzpicture}
		\end{subfigure}
		\begin{subfigure}{0.5\textwidth}
			\centering
			\begin{tikzpicture}
				
				\pgfdeclarelayer{background}
				\pgfdeclarelayer{foreground}
				\pgfsetlayers{background,main,foreground}

				\draw[e:main] (-2.45,1.05) -- (1.05,1.05);
				\draw[e:main] (-2.45,0.35) -- (2.45,0.35);
				\draw[e:main] (-2.45,-0.35) -- (2.45,-0.35);
				\draw[e:main] (-1.05,-1.05) -- (2.45,-1.05);
				\draw[e:main] (-2.45,0.35) -- (-2.45,-0.35);
				\draw[e:main,color=BrightUbe] (-2.45,1.05) -- (-2.45,0.35);
				\draw[e:main] (-1.75,1.05) -- (-1.75,0.35);
				\draw[e:main,color=BrightUbe] (-1.05,-0.35) -- (-1.05,-1.05);
				\draw[e:main,color=BrightUbe] (-0.35,0.35) -- (-0.35,-0.35);
				\draw[e:main] (-0.35,-0.35) -- (-0.35,-1.05);
				
				\draw[e:main,color=BrightUbe] (0.35,1.05) -- (0.35,0.35);
				\draw[e:main] (0.35,0.35) -- (0.35,-0.35);
				
				\draw[e:main] (1.05,1.05) -- (1.05,0.35);
				\draw[e:main,color=BrightUbe] (1.75,-0.35) -- (1.75,-1.05);
				\draw[e:main,color=BrightUbe] (2.45,0.35) -- (2.45,-0.35);
				\draw[e:main] (2.45,-0.35) -- (2.45,-1.05);

				\draw[e:coloredthin,AO] (-2.45,1.05) -- (-1.75,1.05);
				\draw[e:coloredthin,AO] (-1.05,1.05) -- (-0.35,1.05);
				\draw[e:coloredthin,AO] (0.35,1.05) -- (1.05,1.05);
				\draw[e:coloredthin,AO] (-2.45,-0.35) -- (-1.75,-0.35);
				\draw[e:coloredthin,AO] (-1.05,-0.35) -- (-0.35,-0.35);
				\draw[e:coloredthin,AO] (0.35,-0.35) -- (1.05,-0.35);
				\draw[e:coloredthin,AO] (1.75,-0.35) -- (2.45,-0.35);
				\draw[e:coloredthin,color=BostonUniversityRed] (-2.45,0.35) -- (-1.75,0.35);
				\draw[e:coloredthin,color=BostonUniversityRed] (-1.05,0.35) -- (-0.35,0.35);
				\draw[e:coloredthin,color=BostonUniversityRed] (0.35,0.35) -- (1.05,0.35);
				\draw[e:coloredthin,color=BostonUniversityRed] (1.75,0.35) -- (2.45,0.35);
				\draw[e:coloredthin,color=BostonUniversityRed] (-1.05,-1.05) -- (-0.35,-1.05);
				\draw[e:coloredthin,color=BostonUniversityRed] (0.35,-1.05) -- (1.05,-1.05);
				\draw[e:coloredthin,color=BostonUniversityRed] (1.75,-1.05) -- (2.45,-1.05);
				
				\begin{pgfonlayer}{background}
					\draw[e:coloredborder] (-2.45,1.05) -- (-1.75,1.05);
					\draw[e:coloredborder] (-1.05,1.05) -- (-0.35,1.05);
					\draw[e:coloredborder] (0.35,1.05) -- (1.05,1.05);
					\draw[e:coloredborder] (-2.45,-0.35) -- (-1.75,-0.35);
					\draw[e:coloredborder] (-1.05,-0.35) -- (-0.35,-0.35);
					\draw[e:coloredborder] (0.35,-0.35) -- (1.05,-0.35);
					\draw[e:coloredborder] (1.75,-0.35) -- (2.45,-0.35);
					\draw[e:coloredborder] (-2.45,0.35) -- (-1.75,0.35);
					\draw[e:coloredborder] (-1.05,0.35) -- (-0.35,0.35);
					\draw[e:coloredborder] (0.35,0.35) -- (1.05,0.35);
					\draw[e:coloredborder] (1.75,0.35) -- (2.45,0.35);
					\draw[e:coloredborder] (-1.05,-1.05) -- (-0.35,-1.05);
					\draw[e:coloredborder] (0.35,-1.05) -- (1.05,-1.05);
					\draw[e:coloredborder] (1.75,-1.05) -- (2.45,-1.05);
				\end{pgfonlayer}
				
				\node[v:main] at (-2.45,1.05){};
				\node[v:main] at (-2.45,-0.35){};
				\node[v:main] at (-1.05,1.05){};
				\node[v:main] at (-1.05,-0.35){};
				\node[v:main] at (0.35,1.05){};
				\node[v:main] at (0.35,-0.35){};
				\node[v:main] at (1.75,-0.35){};
				\node[v:main] at (2.45,-1.05){};
				\node[v:main] at (2.45,0.35){};
				\node[v:main] at (1.05,-1.05){};
				\node[v:main] at (1.05,0.35){};
				\node[v:main] at (-0.35,-1.05){};
				\node[v:main] at (-0.35,0.35){};
				\node[v:main] at (-1.75,0.35){};
				
				\node[v:mainempty] at (-1.75,1.05){};
				\node[v:mainempty] at (-1.75,-0.35){};
				\node[v:mainempty] at (-0.35,1.05){};
				\node[v:mainempty] at (-0.35,-0.35){};
				\node[v:mainempty] at (1.05,1.05){};
				\node[v:mainempty] at (1.05,-0.35){};
				\node[v:mainempty] at (2.45,-0.35){};
				\node[v:mainempty] at (1.75,-1.05){};
				\node[v:mainempty] at (1.75,0.35){};
				\node[v:mainempty] at (0.35,-1.05){};
				\node[v:mainempty] at (0.35,0.35){};
				\node[v:mainempty] at (-1.05,-1.05){};
				\node[v:mainempty] at (-1.05,0.35){};
				\node[v:mainempty] at (-2.45,0.35){};
				
				\draw[very thick,brown] (-0.35,1.05) ellipse (28pt and 6pt);
				\draw[very thick,brown] (-1.05,0.35) ellipse (28pt and 6pt);
				\draw[very thick,brown] (1.75,0.35) ellipse (28pt and 6pt);
				\draw[very thick,brown] (-1.75,-0.35) ellipse (28pt and 6pt);
				\draw[very thick,brown] (1.05,-0.35) ellipse (28pt and 6pt);
				\draw[very thick,brown] (0.35,-1.05) ellipse (28pt and 6pt);
				
			\end{tikzpicture}
		\end{subfigure}
		\caption{A model of the $4\times4$-grid in $CG_4^{\square}$ (on the left) and a close-up of the model (on the right), in which the white vertices of degree two that should be bicontracted are encircled.}
		\label{fig:squaregrid}
	\end{figure}
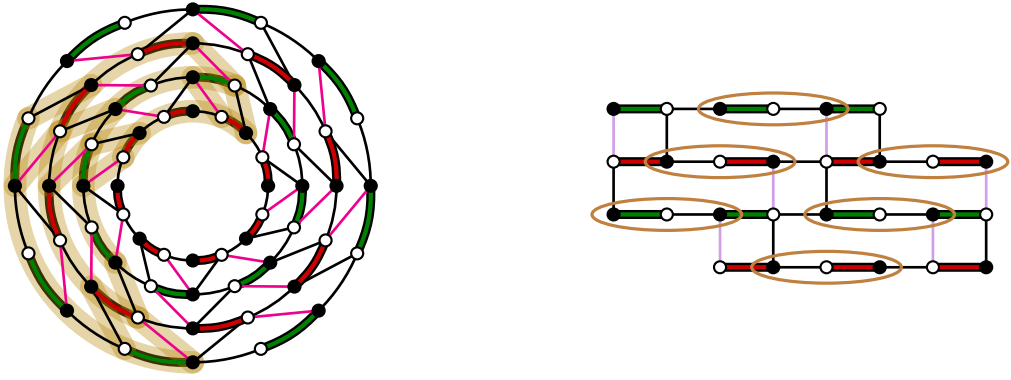
	
	Let us choose as the $C_4$ the one induced by $\Set{v^1_1,v^1_2,v^2_2,v^2_3}$.
	Then take the height piece $H_{2,2}$ and the base piece $B_{1,2}$ and let $G_4$ be the graph obtained by the union of $H_{2,2}$, $B_{1,2}$, and the $C_4$ chosen above.
	The vertex $v^i_j$ with the largest $j$ in $G_4$ is $v^4_{11}$ and thus, since each $C_i$ has $16$ vertices, none of the horizontal paths in $G_4$ closes a cycle.
	More over, if we now bicontract the vertices $v^1_4,~v^2_4,~v^3_4,~v^2_8,~v^3_8\text{, and }v^4_8$, we obtain exactly the $4\times 4$-grid.
	Note that, by construction, $G_4$ is in fact $M'$-conformal in $CG_4^{\square}$ and thus we have found our desired matching minor.
	
	So now assume that for some even $k$ we have already constructed an $M'$-conformal graph $G_k$ in $CG_{k+2}^{\square}$ by using our pieces and starting with the $C_4$ on the vertices $\Set{v^1_1,v^1_2,v^2_2,v^2_3}$.
	In the last step of this proof, we show how to extend $G_k$ to $G_{k+2}$, a bisubdivision of the $\Brace{k+2}\times\Brace{k+2}$-grid.
	Let $k=2z$, we add the following pieces:
	\begin{itemize}
		\item the base piece $B_{1,4z-6}$ and the height piece $H_{2\Brace{z-1},4z-6}$,
		\item for every $i\in[1,z-2]$ the width piece $W_{2i,4\Brace{z+i}-7}$, and
		\item for every $j\in[1,z-2]$ the width piece $W_{2z-2,4\Brace{z+j}-3}$.
	\end{itemize}
	Since $G_k$ is $M'$-conformal, the graph $G_{k+2}$ obtained by adding the above pieces still is $M'$-conformal by construction.
	Consider the set
	\begin{align*}
		S_{k+2}\coloneqq\CondSet{v^1_{i}}{j\in[3,k+2],~i\text{ odd}}\cup\CondSet{v^j_{12}}{j\in[2,k+2],~j\text{ even}},
	\end{align*}
	Every vertex in $S_{k+2}$ has degree two in $G_{k+2}$ and thus is bicontractible.
	Bicontracting all vertices in $S_{k+2}$ yields the desired $\Brace{k+2}$-$\Brace{k+2}$-grid.
\end{proof}

We are finally ready to prove our refined grid theorem for bipartite graphs with perfect matchings.

\begin{proposition}\label{thm:matchinggrid2}
	There exists a function $\MatchingGrid\colon\N\rightarrow\N$ such that for every $k\in\N$ and every bipartite graph $B$ with a perfect matching either $\pmw{B}\leq\Fkt{\MatchingCylinder}{k}$ or $B$ contains the $2k\times 2k$-grid as a matching minor.
\end{proposition}

\begin{proof}
	We set $\Fkt{\MatchingGrid}{k}\coloneqq\Fkt{\MatchingCylinder}{6k}$ for every $k\in\N$.
	Suppose $\pmw{B}>\Fkt{\MatchingGrid}{k}$.
	Then by \cref{thm:matchinggrid}, $B$ contains $CG_{6k}$ as a matching minor.
	Using \cref{lemma:quadrangulate} yields that $B$ contains $CG_{2k}^{\square}$ as a matching minor, and finally \cref{lemma:obtainsquaregrid} lets us find the $2k\times 2k$-grid as a matching minor within $CG_{2k}^{\square}$.
	This completes the proof.
\end{proof}

\begin{proof}[Proof of \Cref{thm:matchinggridminors}]
	Let $B$ be a bipartite, matching covered and planar graph.
	Moreover, for any even $k$, let $M$ be the residual perfect matching obtained by the strategy for finding the $2k\times 2k$-grid as a matching minor of $CG_{2k}^{\square}$ as described in the proof of \cref{lemma:obtainsquaregrid}.
	Instead of proving the claim directly, we show that $k$ can be chosen large enough such that $B$ is a matching minor of the $2k\times 2k$-grid.
	The claim then follows from \cref{lemma:quadrangulate,lemma:obtainsquaregrid}, and the transitivity of the matching minor relation.
	
	We prove the claim by induction on the number of ears in an ear decomposition of $B$ and strengthen it in the sense that we claim that there always exists an $M$-conformal matching minor model.
	As a base consider a single cycle of even length $\ell$.
	Clearly each such cycle is actually contained as an $M$-conformal bisubdivision in the $\ell'\times\ell'$-grid, where $\ell'$ is the smallest natural number satisfying $\frac{\ell}{2}\leq \ell'$.
	So let $K_2=B_1\subset B_2\subset \dots \subset B_t$ be an ear decomposition of $B$.
	By the induction hypothesis, there exists an even number $\omega_{B_{t-1}}$ such that $B_{t-1}$ is a matching minor of the $\omega_{B_{t-1}}\times\omega_{B_{t-1}}$-grid.
	Let $\mu'$ be an $M$-conformal matching minor model of $B_{t-1}$ in said grid.
	Let $P$ be the ear that, added to $B_{t-1}$, creates $B_t=B$.
	Then the canonical embedding of the grid in the plane induces an embedding of $\Fkt{\mu'}{B_{t-1}}$ in the plane and there exists a face $f$ of said drawing that corresponds to the face of $B_{t-1}$ in which $P$ must be placed.
	Since $P$ is non-empty, $f$ must have more than four vertices, and thus there must exist a $C_4$ in the interior of $f$ in the grid.
	Moreover, since $\Fkt{\mu'}{B_{t-1}}$ is $M$-conformal, there must exist such a $C_4$, say $C$, that does not contain a single edge of $M$.
	
	We now draw two orthogonal lines through the centre of $C$, $\ell_1$ in parallel to the columns of our grid and $\ell_2$ in parallel to the rows of the grid.
	Each of the two $\ell_i$ can now be associated with an edge cut of the grid, containing only edges not in $M$.
	Together $\ell_1$ and $\ell_2$ partition the grid into four quadrants, see \cref{fig:gridexpansion} for an illustration.
	Let us say that the shores of $\ell_1$ are $X_1\cup X_2\subseteq\V{B}$ and $Y_1\cup Y_2\subseteq\V{B}$, while the shores of $\ell_2$ are $X_1\cup Y_1$ and $X_2\cup Y_2$.
	Please note that each of the $X_i$ and $Y_i$ is $M$-conformal.
	Moreover, let us fix $X_1$ to be the top left quadrant and $Y_2$ to be the bottom right one.
	
	Now let $H'$ be the $\Brace{\omega_{B_{t-1}}+p}\times\Brace{\omega_{B_{t-1}}+p}$-grid where $p=\Abs{\V{P}}$, note that $p$ is even, and let us map the vertices of the $X_i$ and $Y_i$ to the four corners of $H'$, let $X'_i$ and $Y'_i$ be the corresponding vertex sets in $H'$, let $h$ be said mapping.
	Let us furthermore extend $M$ to the corresponding perfect matching of $H'$.
	In order to extend $\mu'$ to a model of $B_{t-1}$ in $H'$ we need to replace the edges in the two cuts $\ell_1$ and $\ell_2$ by internally $M$-conformal paths connecting $X_1$ with $X_2$ and $Y_1$ with $Y_2$.
	In case $\mu'$ uses two vertical edges incident with the two endpoints of an edge of $M$, this might not be possible.
	To deal with this problem we apply a further blow-up to $H'$, namely we double its width.
	Let $H$ be the $\Brace{\omega_{B_{t-1}}+p}\times\Brace{3\omega_{B_{t-1}}+p-4}$-grid obtained from $H'$ as follows.
	First, let $x\in\mathbb{N}$ be the number of columns in $\InducedSubgraph{H'}{X_1}$.
	Then let $Z_X$ be the $p\times x$-grid made up of the vertices in the columns of $H'$ that connect $\Fkt{h}{X_1}$ and $\Fkt{h}{X_2}$. 
	Let $Z_Y$ be defined analogously, see \cref{fig:gridexpansion} for an illustration.
	
	\begin{figure}[!h]
		\begin{subfigure}{0.4\textwidth}
			\centering
			\begin{tikzpicture}
				
				\pgfdeclarelayer{background}
				\pgfdeclarelayer{foreground}
				\pgfsetlayers{background,main,foreground}

				\foreach \x in {0,...,3}
				{
					\draw[e:coloredthin,color=BostonUniversityRed] (-2.45+\x*1.4,2.45) -- (-2.45 +\x*1.4+0.7,2.45);
					\draw[e:coloredthin,color=BostonUniversityRed] (-2.45+\x*1.4,1.75) -- (-2.45 +\x*1.4+0.7,1.75);
					\draw[e:coloredthin,color=BostonUniversityRed] (-2.45+\x*1.4,1.05) -- (-2.45 +\x*1.4+0.7,1.05);
					\draw[e:coloredthin,color=BostonUniversityRed] (-2.45+\x*1.4,0.35) -- (-2.45 +\x*1.4+0.7,0.35);
					\draw[e:coloredthin,color=BostonUniversityRed] (-2.45+\x*1.4,-2.45) -- (-2.45 +\x*1.4+0.7,-2.45);
					\draw[e:coloredthin,color=BostonUniversityRed] (-2.45+\x*1.4,-1.75) -- (-2.45 +\x*1.4+0.7,-1.75);
					\draw[e:coloredthin,color=BostonUniversityRed] (-2.45+\x*1.4,-1.05) -- (-2.45 +\x*1.4+0.7,-1.05);
					\draw[e:coloredthin,color=BostonUniversityRed] (-2.45+\x*1.4,-0.35) -- (-2.45 +\x*1.4+0.7,-0.35);
				}
				\foreach \x in {0,...,2}
				{
					\draw[e:main] (-1.75+\x*1.4,2.45) -- (-1.75 +\x*1.4+0.7,2.45);
					\draw[e:main] (-1.75+\x*1.4,1.75) -- (-1.75 +\x*1.4+0.7,1.75);
					\draw[e:main] (-1.75+\x*1.4,1.05) -- (-1.75 +\x*1.4+0.7,1.05);
					\draw[e:main] (-1.75+\x*1.4,0.35) -- (-1.75 +\x*1.4+0.7,0.35);
					\draw[e:main] (-1.75+\x*1.4,-2.45) -- (-1.75 +\x*1.4+0.7,-2.45);
					\draw[e:main] (-1.75+\x*1.4,-1.75) -- (-1.75 +\x*1.4+0.7,-1.75);
					\draw[e:main] (-1.75+\x*1.4,-1.05) -- (-1.75 +\x*1.4+0.7,-1.05);
					\draw[e:main] (-1.75+\x*1.4,-0.35) -- (-1.75 +\x*1.4+0.7,-0.35);
				}

				\foreach \x in {0,...,7}
				{
					\draw[e:main](-2.45+\x*0.7,2.45) -- (-2.45+\x*0.7,-2.45);
				}
				
				\foreach \x in {0,...,3}
				{
					\node[v:main] at (-2.45+\x*1.4,2.45){};
					\node[v:main] at (-2.45+\x*1.4,1.05){};
					\node[v:main] at (-2.45+\x*1.4,-0.35){};
					\node[v:main] at (-2.45+\x*1.4,-1.75){};
					\node[v:main] at (-1.75+\x*1.4,-2.45){};
					\node[v:main] at (-1.75+\x*1.4,-1.05){};
					\node[v:main] at (-1.75+\x*1.4,0.35){};
					\node[v:main] at (-1.75+\x*1.4,1.75){};
					\node[v:mainempty] at (-2.45+\x*1.4,-2.45){};
					\node[v:mainempty] at (-2.45+\x*1.4,-1.05){};
					\node[v:mainempty] at (-2.45+\x*1.4,0.35){};
					\node[v:mainempty] at (-2.45+\x*1.4,1.75){};
					\node[v:mainempty] at (-1.75+\x*1.4,2.45){};
					\node[v:mainempty] at (-1.75+\x*1.4,1.05){};
					\node[v:mainempty] at (-1.75+\x*1.4,-0.35){};
					\node[v:mainempty] at (-1.75+\x*1.4,-1.75){};
				}
				
				\begin{pgfonlayer}{background}
					\draw[e:marker,color=CornflowerBlue] (-2.45,1.75) -- (2.45,1.75) -- (2.45,1.05) -- (1.75,1.05) -- (1.75,0.35) -- (2.45,0.35) -- (2.45,-0.35) -- (1.75,-0.35) -- (1.75,-1.05) -- (2.45,-1.05) -- (2.45,-1.75) -- (-2.45,-1.75) -- (-2.45,-1.05) -- (-1.75,-1.05) -- (-1.75,-0.35) -- (-2.45,-0.35) -- (-2.45,0.35) -- (-1.75,0.35) -- (-1.75,1.05) -- (-2.45,1.05) -- (-2.45,1.75);
					\draw[e:marker,color=AO] (-0.35,-0.35) -- (0.35,-0.35) -- (0.35,0.35) -- (-0.35,0.35) -- (-0.35,-0.35);
					
					\draw[e:marker,color=BrightUbe,line width=5pt] (-2.7,0) -- (2.7,0);
					\draw[e:marker,color=BrightUbe,line width=5pt] (0,-2.7) -- (0,2.7);
					
					\foreach \x in {0,...,3}
					{
						\draw[e:coloredborder] (-2.45+\x*1.4,2.45) -- (-2.45 +\x*1.4+0.7,2.45);
						\draw[e:coloredborder] (-2.45+\x*1.4,1.75) -- (-2.45 +\x*1.4+0.7,1.75);
						\draw[e:coloredborder] (-2.45+\x*1.4,1.05) -- (-2.45 +\x*1.4+0.7,1.05);
						\draw[e:coloredborder] (-2.45+\x*1.4,0.35) -- (-2.45 +\x*1.4+0.7,0.35);
						\draw[e:coloredborder] (-2.45+\x*1.4,-2.45) -- (-2.45 +\x*1.4+0.7,-2.45);
						\draw[e:coloredborder] (-2.45+\x*1.4,-1.75) -- (-2.45 +\x*1.4+0.7,-1.75);
						\draw[e:coloredborder] (-2.45+\x*1.4,-1.05) -- (-2.45 +\x*1.4+0.7,-1.05);
						\draw[e:coloredborder] (-2.45+\x*1.4,-0.35) -- (-2.45 +\x*1.4+0.7,-0.35);
					}
					
				\end{pgfonlayer}
			\end{tikzpicture}
		\end{subfigure} 
		\begin{subfigure}{0.6\textwidth}
			\centering
			\begin{tikzpicture}
				
				\pgfdeclarelayer{background}
				\pgfdeclarelayer{foreground}
				\pgfsetlayers{background,main,foreground}

				\foreach \x in {0,...,5}
				{
					\draw[e:coloredthin,color=BostonUniversityRed] (-3.85+\x*1.4,3.85) -- (-3.85 +\x*1.4+0.7,3.85);
					\draw[e:coloredthin,color=BostonUniversityRed] (-3.85+\x*1.4,3.15) -- (-3.85 +\x*1.4+0.7,3.15);
					\draw[e:coloredthin,color=BostonUniversityRed] (-3.85+\x*1.4,2.45) -- (-3.85 +\x*1.4+0.7,2.45);
					\draw[e:coloredthin,color=BostonUniversityRed] (-3.85+\x*1.4,1.75) -- (-3.85 +\x*1.4+0.7,1.75);
					\draw[e:coloredthin,color=BostonUniversityRed] (-3.85+\x*1.4,1.05) -- (-3.85 +\x*1.4+0.7,1.05);
					\draw[e:coloredthin,color=BostonUniversityRed] (-3.85+\x*1.4,0.35) -- (-3.85 +\x*1.4+0.7,0.35);
					\draw[e:coloredthin,color=BostonUniversityRed] (-3.85+\x*1.4,-3.85) -- (-3.85 +\x*1.4+0.7,-3.85);
					\draw[e:coloredthin,color=BostonUniversityRed] (-3.85+\x*1.4,-3.15) -- (-3.85 +\x*1.4+0.7,-3.15);
					\draw[e:coloredthin,color=BostonUniversityRed] (-3.85+\x*1.4,-2.45) -- (-3.85 +\x*1.4+0.7,-2.45);
					\draw[e:coloredthin,color=BostonUniversityRed] (-3.85+\x*1.4,-1.75) -- (-3.85 +\x*1.4+0.7,-1.75);
					\draw[e:coloredthin,color=BostonUniversityRed] (-3.85+\x*1.4,-1.05) -- (-3.85 +\x*1.4+0.7,-1.05);
					\draw[e:coloredthin,color=BostonUniversityRed] (-3.85+\x*1.4,-0.35) -- (-3.85 +\x*1.4+0.7,-0.35);
				}

				\foreach \x in {0,...,4}
				{
					
					\draw[e:main,gray] (-3.15+\x*1.4,1.05) -- (-3.15 +\x*1.4+0.7,1.05);
					\draw[e:main,gray] (-3.15+\x*1.4,0.35) -- (-3.15 +\x*1.4+0.7,0.35);
					
					\draw[e:main,gray] (-3.15+\x*1.4,-1.05) -- (-3.15 +\x*1.4+0.7,-1.05);
					\draw[e:main,gray] (-3.15+\x*1.4,-0.35) -- (-3.15 +\x*1.4+0.7,-0.35);
				}
				
				\foreach \x in {1,...,3}
				{
					\draw[e:main,gray] (-3.15+\x*1.4,-3.85) -- (-3.15 +\x*1.4+0.7,-3.85);
					\draw[e:main,gray] (-3.15+\x*1.4,-3.15) -- (-3.15 +\x*1.4+0.7,-3.15);
					\draw[e:main,gray] (-3.15+\x*1.4,-2.45) -- (-3.15 +\x*1.4+0.7,-2.45);
					\draw[e:main,gray] (-3.15+\x*1.4,-1.75) -- (-3.15 +\x*1.4+0.7,-1.75);
					\draw[e:main,gray] (-3.15+\x*1.4,3.85) -- (-3.15 +\x*1.4+0.7,3.85);
					\draw[e:main,gray] (-3.15+\x*1.4,3.15) -- (-3.15 +\x*1.4+0.7,3.15);
					\draw[e:main,gray] (-3.15+\x*1.4,2.45) -- (-3.15 +\x*1.4+0.7,2.45);
					\draw[e:main,gray] (-3.15+\x*1.4,1.75) -- (-3.15 +\x*1.4+0.7,1.75);
				}
				
				\foreach \x in {0,4}
				{
					\draw[e:main] (-3.15+\x*1.4,-3.85) -- (-3.15 +\x*1.4+0.7,-3.85);
					\draw[e:main] (-3.15+\x*1.4,-3.15) -- (-3.15 +\x*1.4+0.7,-3.15);
					\draw[e:main] (-3.15+\x*1.4,-2.45) -- (-3.15 +\x*1.4+0.7,-2.45);
					\draw[e:main] (-3.15+\x*1.4,-1.75) -- (-3.15 +\x*1.4+0.7,-1.75);
					\draw[e:main] (-3.15+\x*1.4,3.85) -- (-3.15 +\x*1.4+0.7,3.85);
					\draw[e:main] (-3.15+\x*1.4,3.15) -- (-3.15 +\x*1.4+0.7,3.15);
					\draw[e:main] (-3.15+\x*1.4,2.45) -- (-3.15 +\x*1.4+0.7,2.45);
					\draw[e:main] (-3.15+\x*1.4,1.75) -- (-3.15 +\x*1.4+0.7,1.75);
				}

				\foreach \x in {4,...,7}
				{
					\draw[e:main,gray](-3.85+\x*0.7,3.85) -- (-3.85+\x*0.7,-3.85);
				}
				
				\foreach \x in {0,...,3}
				{
					\draw[e:main](-3.85+\x*0.7,3.85) -- (-3.85+\x*0.7,1.75);
					\draw[e:main](-3.85+\x*0.7,-3.85) -- (-3.85+\x*0.7,-1.75);
					\draw[e:main,gray] (-3.85+\x*0.7,1.75) -- (-3.85+\x*0.7,-1.75);
				}
				\foreach \x in {8,...,11}
				{
					\draw[e:main](-3.85+\x*0.7,3.85) -- (-3.85+\x*0.7,1.75);
					\draw[e:main](-3.85+\x*0.7,-3.85) -- (-3.85+\x*0.7,-1.75);
					\draw[e:main,gray] (-3.85+\x*0.7,1.75) -- (-3.85+\x*0.7,-1.75);
				}
				
				\foreach \x in {0,...,5}
				{
					\node[v:main,gray] at (-3.85+\x*1.4,1.05){};
					\node[v:main,gray] at (-3.85+\x*1.4,-0.35){};
					\node[v:main,gray] at (-3.15+\x*1.4,-1.05){};
					\node[v:main,gray] at (-3.15+\x*1.4,0.35){};
					\node[v:mainemptygray] at (-3.85+\x*1.4,-1.05){};
					\node[v:mainemptygray] at (-3.85+\x*1.4,0.35){};
					\node[v:mainemptygray] at (-3.15+\x*1.4,1.05){};
					\node[v:mainemptygray] at (-3.15+\x*1.4,-0.35){};
				}

				\foreach \x in {0,1,4,5}
				{
					\node[v:main] at (-3.85+\x*1.4,3.85){};
					\node[v:main] at (-3.85+\x*1.4,2.45){};
					\node[v:main] at (-3.85+\x*1.4,-1.75){};
					\node[v:main] at (-3.85+\x*1.4,-3.15){};
					\node[v:main] at (-3.15+\x*1.4,-3.85){};
					\node[v:main] at (-3.15+\x*1.4,-2.45){};
					\node[v:main] at (-3.15+\x*1.4,1.75){};
					\node[v:main] at (-3.15+\x*1.4,3.15){};
					\node[v:mainempty] at (-3.85+\x*1.4,-3.85){};
					\node[v:mainempty] at (-3.85+\x*1.4,-2.45){};
					\node[v:mainempty] at (-3.85+\x*1.4,1.75){};
					\node[v:mainempty] at (-3.85+\x*1.4,3.15){};
					\node[v:mainempty] at (-3.15+\x*1.4,3.85){};
					\node[v:mainempty] at (-3.15+\x*1.4,2.45){};
					\node[v:mainempty] at (-3.15+\x*1.4,-1.75){};
					\node[v:mainempty] at (-3.15+\x*1.4,-3.15){};
				}
				
				\foreach \x in {2,3}
				{
					\node[v:main,gray] at (-3.85+\x*1.4,3.85){};
					\node[v:main,gray] at (-3.85+\x*1.4,2.45){};
					\node[v:main,gray] at (-3.85+\x*1.4,-1.75){};
					\node[v:main,gray] at (-3.85+\x*1.4,-3.15){};
					\node[v:main,gray] at (-3.15+\x*1.4,-3.85){};
					\node[v:main,gray] at (-3.15+\x*1.4,-2.45){};
					\node[v:main,gray] at (-3.15+\x*1.4,1.75){};
					\node[v:main,gray] at (-3.15+\x*1.4,3.15){};
					\node[v:mainemptygray] at (-3.85+\x*1.4,-3.85){};
					\node[v:mainemptygray] at (-3.85+\x*1.4,-2.45){};
					\node[v:mainemptygray] at (-3.85+\x*1.4,1.75){};
					\node[v:mainemptygray] at (-3.85+\x*1.4,3.15){};
					\node[v:mainemptygray] at (-3.15+\x*1.4,3.85){};
					\node[v:mainemptygray] at (-3.15+\x*1.4,2.45){};
					\node[v:mainemptygray] at (-3.15+\x*1.4,-1.75){};
					\node[v:mainemptygray] at (-3.15+\x*1.4,-3.15){};
				}
				
				\begin{pgfonlayer}{background}
					\draw[e:marker,color=AO] (-1.75,-1.75) -- (-1.75,1.75) -- (1.75,1.75) -- (1.75,-1.75) -- (-1.75,-1.75);
					\draw[e:marker,color=CornflowerBlue] (-3.85,1.4) -- (-3.85,1.75) -- (-3.15,1.75) -- (-3.15,2.45) -- (-3.85,2.45) -- (-3.85,3.15) -- (-1.4,3.15);
					\draw[e:marker,color=BrightUbe] (-1.4,3.15) -- (1.4,3.15);
					\draw[e:marker,color=CornflowerBlue] (1.4,3.15) -- (3.85,3.15) -- (3.85,2.45) -- (3.15,2.45) -- (3.15,1.75) -- (3.85,1.75) -- (3.85,1.4) ;
					\draw[e:marker,color=BrightUbe] (3.85,1.4) -- (3.85,1.05) -- (3.15,1.05) -- (3.15,0.35) -- (3.85,0.35) -- (3.85,-0.35) -- (3.15,-0.35) -- (3.15,-1.05) -- (3.85,-1.04) -- (3.85,-1.4);
					\draw[e:marker,color=CornflowerBlue] (1.4,-3.15) -- (3.85,-3.15) -- (3.85,-2.45) -- (3.15,-2.45) -- (3.15,-1.75) -- (3.85,-1.75) -- (3.85,-1.4) ;
					\draw[e:marker,color=BrightUbe] (-1.4,-3.15) -- (1.4,-3.15);
					\draw[e:marker,color=CornflowerBlue] (-3.85,-1.4) -- (-3.85,-1.75) -- (-3.15,-1.75) -- (-3.15,-2.45) -- (-3.85,-2.45) -- (-3.85,-3.15) -- (-1.4,-3.15);
					\draw[e:marker,color=BrightUbe] (-3.85,1.4) -- (-3.85,1.05) -- (-3.15,1.05) -- (-3.15,0.35) -- (-3.85,0.35) -- (-3.85,-0.35) -- (-3.15,-0.35) -- (-3.15,-1.05) -- (-3.85,-1.04) -- (-3.85,-1.4);
					\draw[e:marker,color=brown] (-2.45,3.15) -- (-2.45,2.45) -- (-0.35,2.45) -- (-0.35,1.75) -- (-1.05,1.75) -- (-1.05,1.05) -- (-0.35,1.05) -- (-0.35,0.35) -- (-1.05,0.35) -- (-1.05,-0.35) -- (1.05,-0.35)-- (1.05,-1.05) -- (0.35,-1.05) -- (0.35,-1.75) -- (2.45,-1.75) -- (2.45,-2.45) -- (1.75,-2.45) -- (1.75,-3.15);
					
					\foreach \x in {0,...,5}
					{
						\draw[e:coloredborder] (-3.85+\x*1.4,3.85) -- (-3.85 +\x*1.4+0.7,3.85);
						\draw[e:coloredborder] (-3.85+\x*1.4,3.15) -- (-3.85 +\x*1.4+0.7,3.15);
						\draw[e:coloredborder] (-3.85+\x*1.4,2.45) -- (-3.85 +\x*1.4+0.7,2.45);
						\draw[e:coloredborder] (-3.85+\x*1.4,1.75) -- (-3.85 +\x*1.4+0.7,1.75);
						\draw[e:coloredborder] (-3.85+\x*1.4,1.05) -- (-3.85 +\x*1.4+0.7,1.05);
						\draw[e:coloredborder] (-3.85+\x*1.4,0.35) -- (-3.85 +\x*1.4+0.7,0.35);
						\draw[e:coloredborder] (-3.85+\x*1.4,-3.85) -- (-3.85 +\x*1.4+0.7,-3.85);
						\draw[e:coloredborder] (-3.85+\x*1.4,-3.15) -- (-3.85 +\x*1.4+0.7,-3.15);
						\draw[e:coloredborder] (-3.85+\x*1.4,-2.45) -- (-3.85 +\x*1.4+0.7,-2.45);
						\draw[e:coloredborder] (-3.85+\x*1.4,-1.75) -- (-3.85 +\x*1.4+0.7,-1.75);
						\draw[e:coloredborder] (-3.85+\x*1.4,-1.05) -- (-3.85 +\x*1.4+0.7,-1.05);
						\draw[e:coloredborder] (-3.85+\x*1.4,-0.35) -- (-3.85 +\x*1.4+0.7,-0.35);
					}
				\end{pgfonlayer}
			\end{tikzpicture}
		\end{subfigure}
		\caption{Expanding a grid together with a model of an even cycle to add an ear.
			The small marked $C_4$ is replaced by a large grid which then is extended to make a new quadratic grid.
			Then the old model is extended by routing through the new part and lastly the ear is routed through the newly added central grid.}
		\label{fig:gridexpansion}
	\end{figure}
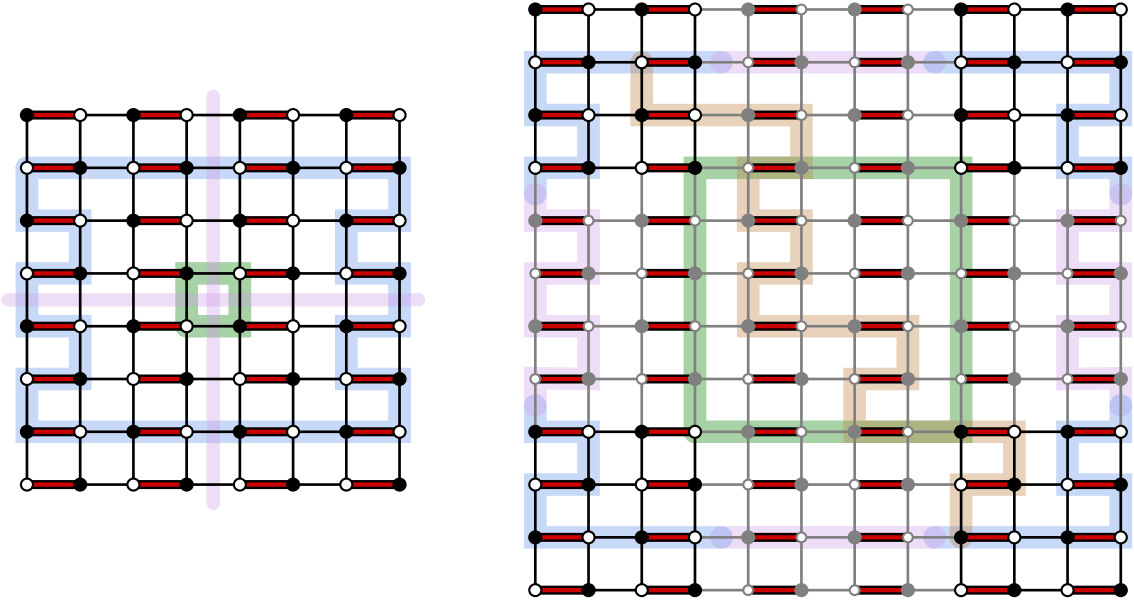
	
	For every $W\in\Set{X,Y}$ take $H'_W\coloneqq\InducedSubgraph{H'}{W_1\cup W_W\cup W_2}$, then subdivide every horizontal edge and complete each thereby newly created column to a path.
	For every $v_{i,j}$ of $H'_W$, $1\leq i\leq \omega_{B_{t-1}}+p$, $1\leq j\leq x-1$, we thereby created two new vertices $v^1_{i,j}$ and $v^2_{i,j}$ subdividing the edge $v_{i,j}v_{i,j+1}$.
	Similarly for every $1\leq i\leq \omega_{B_{t-1}}+p$ and every $x+p+1\leq j\leq \omega_{B_{t-1}}+p-1$.
	Let us again adapt $M$ to be the canonical extension of the perfect matching we used for $H'$.
	
	We now describe how to extend $\mu'$ to $H$.
	Let $v\in\V{B_{t-1}}$ and $u_{i',j'}\in\V{\Fkt{\mu'}{v}}$, then let $\Fkt{h}{u_{i',j'}}=v_{i,j}$.
	Every edge $u_{i'',j'}u_{i',j'}\in\Fkt{E}{\Fkt{\mu'}{v}}$ with $i''\in\Set{i'-1,i'+1}$ is replaced by the edge $\Fkt{h}{u_{i'',j'}}\Fkt{h}{u_{i',j'}}$.
	We extend the model of $v$ by the path $\Brace{v_{i,j},v^1_{i,j},v^2_{i,j}}$ if $j\in\CondSet{n}{1\leq n\leq x-1,\text{ or }x+p+1\leq n \leq \omega_{B_{t-1}}+p-1}$, and every edge $u_{i',j'}u_{i',j'+1}\in\Fkt{E}{\Fkt{\mu'}{v}}\setminus\ell_1$ is replaced by the edge $v^2_{i,j}\Fkt{h}{u_{i',j'+1}}$.
	If there is an edge $ab\in\Fkt{E}{\Fkt{\mu'}{v}}\cap\ell_1$, we replace this edge by the horizontal and internally $M$-conformal $\Fkt{h}{a}$-$\Fkt{h}{b}$-path in $H$.
	At last, an edge $ab\in\Fkt{E}{\Fkt{\mu'}{v}}\cap\ell_2$ is replaced by a vertical $\Fkt{h}{a}$-$\Fkt{h}{b}$-path in $H$.
	This path has to use vertices from at most two columns and may go to the left (in decreasing $j$ direction) if and only if $\Fkt{h}{a}$ and $\Fkt{h}{b}$ are in the column $x$ or $\omega_{B_{t-1}}+p$.
	
	Now let $uw\in\Fkt{E}{B_{t-1}}$.
	If $u_{i'j'}u_{i',j'+1}\in\Fkt{E}{\Fkt{\mu}{uw}}\setminus\ell_1$, then let $v_{i,j}\coloneqq\Fkt{h}{u_{i'j'}}$ and we replace the edge by the path $\Brace{v_{i,j},v^1_{i,j},v^2_{i,j},v_{i,j+1}}$ where $v_{i,j+1}=\Fkt{h}{u_{i',j'+1}}$.
	Edges $u_{i'j'}u_{i',j'+1}\in\Fkt{E}{\Fkt{\mu}{uw}}\cap\ell_1$ are replaced by the unique internally $M$-conformal horizontal $\Fkt{h}{u_{i',j'}}$-$\Fkt{h}{u_{i',j'+1}}$-path in $H$.
	An edge $ab\in\Fkt{E}{\Fkt{\mu'}{uw}}\cap\ell_2$ is replaced by a vertical $\Fkt{h}{a}$-$\Fkt{h}{b}$-path in $H$.
	This path has to use vertices from at most two columns and may go to the left (in decreasing $j$ direction) if and only if $\Fkt{h}{a}$ and $\Fkt{h}{b}$ are in the column $x$ or $\omega_{B_{t-1}}+p$.
	At last, a vertical edge $ab\in\Fkt{E}{\Fkt{\mu}{uw}}\setminus\ell_2$ will simply be replaced by $\Fkt{h}{a}\Fkt{h}{b}$.
	
	In total let $\mu''$ be the matching minor model of $B_{t-1}$ constructed following the rules above.
	It is straight forward to check that $\Fkt{\mu''}{B_{t-1}}$ is $M$-conformal.
	
	By construction there exists a $p\times p$-grid $F$ in the face $f'$ of $\Fkt{\mu''}{B_{t-1}}$ corresponding to the face $f$ we chose in $\Fkt{\mu'}{B_{t-1}}$.
	As a last step, we have to add an internally $M$-conformal path $P$ to our matching minor model in order to form a matching minor model $\mu$ of $B$.
	Let $a,b\in \V{B_{t-1}}$ be the endpoints of $P$, then both $\Fkt{\mu''}{a}$ and $\Fkt{\mu''}{b}$ must have an old vertex on $f'$.
	After possibly stretching the model of $f'$ a bit we can find disjoint internally $M$-conformal paths from $a$ and $b$ to $F$, let $a'$ and $b'$ be their respective endpoints.
	Since $F$ is a $p\times p$-grid we can easily find an internally $M$-conformal $a'$-$b'$-path $P'$ within $F$.
	This path $P'$ together with $\mu''$ forms our desired matching minor model $\mu$ of $B$ in $H$.
	At last note that $H$ is a conformal subgraph of the $\Brace{3\omega_{B_{t-1}}+p-4}\times\Brace{3\omega_{B_{t-1}}+p-4}$-grid and thus we are done.
\end{proof}

\subsection{The Erd\H{o}s-P{\'o}sa Property for Butterfly Minor Anti-Chains}\label{subsec:generalisedEP}

With \cref{thm:directedgridminors} we have an exact description of all strongly connected digraphs $D$ for which $\Antichain{D}$ contains a butterfly minor of the cylindrical grid.
Moreover, since recognising strongly planar digraphs is equivalent to recognising planar bipartite graphs with perfect matchings, we can recognise these digraphs in polynomial time.
Let us define a generalised version of the Erd\H{o}s-P\'osa property for digraphs based on canonical anti-chains.

\begin{definition}[Generalised Erd\H{o}s-P\'osa Property for Butterfly Minors]
	Let $H$ be a strongly connected digraph.
	We say that $H$ has the \emph{generalised Erd\H{o}s-P\'osa property for digraphs} if there exists a function $f\colon\N\rightarrow\N$ such that for every $k\in\N$, every digraph $D$ either contains $k$ pairwise disjoint subgraphs such that each of them has a butterfly minor isomorphic to some member of $\Antichain{H}$, or there exists a set $S\subseteq\V{D}$ with $\Abs{S}\leq\Fkt{f}{k}$ such that $D-S$ does not contain a digraph from $\Antichain{H}$ as a butterfly minor. 
\end{definition}

Our digraphic analogue of \cref{thm:matchingEP} is as follows.
In the forward direction of the proof we use a generalised argument similar to the one used to proof the forward direction of \cite{amiri2016erdos}, while for the reverse we also adapt the strategy from \cite{amiri2016erdos}, this time we stick even closer to the original.

\begin{theorem}\label{thm:generalbutterflyEP}
	A strongly connected digraph $D$ has the generalised Erd\H{o}s-P\'osa property for butterfly minors if and only if $D$ is strongly planar. 	
\end{theorem}

\begin{proof}
	Given a strongly connected strongly planar digraph $D$ let us denote by $\omega_D$ the smallest integer $w$ such that $\Antichain{D}$ contains a butterfly minor of the cylindrical grid of order $w$.
	Note that for any positive integer $k\in\N$ the cylindrical grid of order $k\omega_D$ contains $k$ pairwise vertex disjoint subgraphs, all of which contain a digraph from $\Antichain{D}$ as a matching minor.
	Let us recursively define the function $f_D\colon\N\rightarrow\N$ for the generalised Erd\H{o}s-P\'osa property, where $\Fkt{f_D}{0}\coloneqq 0$, and for $k\geq 1$ let
	\begin{align*}
		\Fkt{f_D}{k}\coloneqq \Fkt{f_D}{k-1}+\Fkt{\DirectedGrid}{k\omega_D}+1.
	\end{align*}
	Now if $\dtw{D}\geq \Fkt{\DirectedGrid}{k\omega_D}+1$, then by the Directed Grid Theorem \cite{kawarabayashi2015directed} $D$ contains the cylindrical grid of order $k\omega_D$ as a butterfly minor and thus, as discussed above, $D$ contains $k$ pairwise disjoint subgraphs, each of which contain a digraph from $\Antichain{D}$ as a butterfly minor.
	So we may assume $D$ to have a directed tree decomposition $\Brace{T,\beta,\gamma}$ of width at most $\Fkt{\DirectedGrid}{k\omega_D}$.
	Let us choose $t\in\V{T}$ such that $\InducedSubgraph{D}{\Fkt{\beta}{T_t}}$ contains a butterfly minor isomorphic to some member of $\Antichain{D}$, but for all $t'\in\V{T_t}$ with $t\neq t'$, $\InducedSubgraph{D}{\Fkt{\beta}{T_{t'}}}$ does not contain any digraph from $\Antichain{D}$ as a butterfly minor.
	If no such $t$ exists, $D$ does not contain a digraph from $\Antichain{D}$ as a butterfly minor and thus we are done immediately.
	Indeed, we may use this case as the base case $k=0$ of our induction.
	Hence we may assume $k\geq 1$ and thus $t$ exists.
	Then $\Abs{\Fkt{\beta}{t}}\leq\dtw{D}+1\leq\Fkt{\DirectedGrid}{k\omega_D}+1$ and every butterfly minor of $\InducedSubgraph{D}{\Fkt{\beta}{t}}$ that belongs to $\Antichain{D}$ must contain a vertex of $\Fkt{\beta}{t}$.
	By induction we either find $k-1$ pairwise vertex disjoint subgraph of $D-\Fkt{\beta}{T_t}$ all of which have a member $\Antichain{D}$ as a butterfly minor, or there is a set $S'$ of vertices with $\Abs{S'}\leq\Fkt{f_D}{k-1}$ such that $D-\Fkt{\beta}{t}-S'$ has no member of $\Antichain{D}$ has a butterfly minor.
	In the first case, all $k-1$ subgraphs are vertex disjoint from $\InducedSubgraph{D}{\Fkt{\beta}{T_t}}$ and thus we are done.
	Otherwise $\Abs{\Fkt{\beta}{t}\cup S}\leq\Fkt{f_D}{k}$, and we are also done.
	Therefore every strongly connected strongly planar digraph $D$ has the generalised Erd\H{o}s-P\'osa property for butterfly minors.
	
	For the reverse direction let $H$ be a strongly connected digraph which is not strongly planar.
	For each $k\in\N$, $k\geq 1$, we construct a digraph $D_{H,k}$ which contains no two disjoint subgraphs that have a digraph from $\Antichain{H}$ as a butterfly minor, but where one must delete at least $k$ vertices to remove all occurrences of members of $\Antichain{H}$ as butterfly minors in $D_{H,k}$.
	Since $k$ is arbitrary, this proves that no non-strongly planar digraph can have the generalised Erd\H{o}s-P\'osa property for butterfly minors.
	Let $G_k$ be the cylindrical grid of order $k$ and let $C_1$ be the outer-most of its concentric cycles.
	Let us select $e_1=\Brace{v_1^1,v_2^1}$, $e_2=\Brace{v_3^1,v_4^1}$, $\dots$, $\Brace{v_{2k-1}^1,v_{2k}^1}\in\E{C_1}$, where we identify $v_{2k}^1$ and $v_0^1$.
	Then let $e=\Brace{u,v}\in\E{H}$ be an arbitrary edge.
	We introduce $k$ pairwise vertex disjoint copies $H_1,\dots, H_k$ of $H$ and denote the copy of $\Brace{u,v}$ in $H_i$ by $\Brace{u_i,v_i}$ for all $i\in[1,k]$.
	Then $D_{H,k}$ is defined as the digraph obtained by deleting the edges $\Brace{u_i,v_i}$ for every $i\in[1,k]$ and introducing the edges $\Brace{u_i,v^1_{2i}}$ and $\Brace{v^1_{2i-1},v_i}$ for each $i\in[1,k]$.
	Again we identify $v_{2k}^1$ and $v_0^1$.
	See \cref{fig:EPconstruction} for an illustration
	
	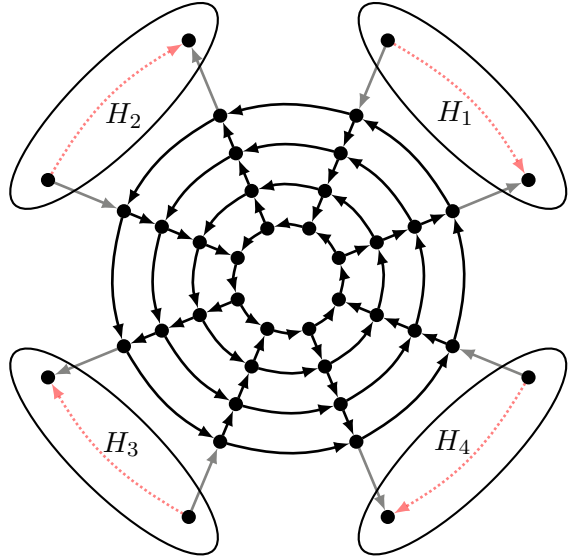
\begin{figure}[!h]
		\centering
		\begin{tikzpicture}[scale=0.9]
			\pgfdeclarelayer{background}
			\pgfdeclarelayer{foreground}
			\pgfsetlayers{background,main,foreground}
			
			%	\tikzset{res/.style={ellipse,draw,minimum height=0.5cm,minimum width=0.8cm}}
			
			\node (mid) [v:ghost] {};
			
			\node (u11) [v:main,position=22.5:8mm from mid] {};
			\node (u21) [v:main,position=67.5:8mm from mid] {};
			\node (u31) [v:main,position=112.5:8mm from mid] {};
			\node (u41) [v:main,position=157.5:8mm from mid] {};
			\node (u51) [v:main,position=202.5:8mm from mid] {};
			\node (u61) [v:main,position=247.5:8mm from mid] {};
			\node (u71) [v:main,position=292.5:8mm from mid] {};
			\node (u81) [v:main,position=337.5:8mm from mid] {};
			
			\node (u12) [v:main,position=22.5:14mm from mid] {};
			\node (u22) [v:main,position=67.5:14mm from mid] {};
			\node (u32) [v:main,position=112.5:14mm from mid] {};
			\node (u42) [v:main,position=157.5:14mm from mid] {};
			\node (u52) [v:main,position=202.5:14mm from mid] {};
			\node (u62) [v:main,position=247.5:14mm from mid] {};
			\node (u72) [v:main,position=292.5:14mm from mid] {};
			\node (u82) [v:main,position=337.5:14mm from mid] {};
			
			\node (u13) [v:main,position=22.5:20mm from mid] {};
			\node (u23) [v:main,position=67.5:20mm from mid] {};
			\node (u33) [v:main,position=112.5:20mm from mid] {};
			\node (u43) [v:main,position=157.5:20mm from mid] {};
			\node (u53) [v:main,position=202.5:20mm from mid] {};
			\node (u63) [v:main,position=247.5:20mm from mid] {};
			\node (u73) [v:main,position=292.5:20mm from mid] {};
			\node (u83) [v:main,position=337.5:20mm from mid] {};
			
			\node (u14) [v:main,position=22.5:26mm from mid] {};
			\node (u24) [v:main,position=67.5:26mm from mid] {};
			\node (u34) [v:main,position=112.5:26mm from mid] {};
			\node (u44) [v:main,position=157.5:26mm from mid] {};
			\node (u54) [v:main,position=202.5:26mm from mid] {};
			\node (u64) [v:main,position=247.5:26mm from mid] {};
			\node (u74) [v:main,position=292.5:26mm from mid] {};
			\node (u84) [v:main,position=337.5:26mm from mid] {};
			
			\node (Hu1) [v:main,position=67.5:38mm from mid] {};
			\node (Hv1) [v:main,position=22.5:38mm from mid] {};
			\node (H1) [draw,thick,ellipse,position=45:36mm from mid,minimum height=11mm,minimum width=37mm,rotate=135] {};
			\node (H1L) [v:ghost,position=45:34.5mm from mid] {$H_1$};
			
			\node (Hu2) [v:main,position=157.5:38mm from mid] {};
			\node (Hv2) [v:main,position=112.5:38mm from mid] {};
			\node (H2) [draw,thick,ellipse,position=135:36mm from mid,minimum height=11mm,minimum width=37mm,rotate=225] {};
			\node (H2L) [v:ghost,position=135:34mm from mid] {$H_2$};
			
			\node (Hu3) [v:main,position=247.5:38mm from mid] {};
			\node (Hv3) [v:main,position=202.5:38mm from mid] {};
			\node (H3) [draw,thick,ellipse,position=225:36mm from mid,minimum height=11mm,minimum width=37mm,rotate=315] {};
			\node (H3L) [v:ghost,position=225:34.5mm from mid] {$H_3$};
			
			\node (Hu4) [v:main,position=337.5:38mm from mid] {};
			\node (Hv4) [v:main,position=292.5:38mm from mid] {};
			\node (H4) [draw,thick,ellipse,position=315:36mm from mid,minimum height=11mm,minimum width=37mm,rotate=45] {};
			\node (H4L) [v:ghost,position=315:34mm from mid] {$H_4$};
			
			\begin{pgfonlayer}{background}
				
				\draw[e:main,densely dotted,red,opacity=0.5,->,bend left=15] (Hu1) to (Hv1);
				\draw[e:main,color=DarkGray,->] (u14) to (Hv1);
				\draw[e:main,color=DarkGray,->] (Hu1) to (u24);
				
				\draw[e:main,densely dotted,red,opacity=0.5,->,bend left=15] (Hu2) to (Hv2);
				\draw[e:main,color=DarkGray,->] (u34) to (Hv2);
				\draw[e:main,color=DarkGray,->] (Hu2) to (u44);
				
				\draw[e:main,densely dotted,red,opacity=0.5,->,bend left=15] (Hu3) to (Hv3);
				\draw[e:main,color=DarkGray,->] (u54) to (Hv3);
				\draw[e:main,color=DarkGray,->] (Hu3) to (u64);
				
				\draw[e:main,densely dotted,red,opacity=0.5,->,bend left=15] (Hu4) to (Hv4);
				\draw[e:main,color=DarkGray,->] (u74) to (Hv4);
				\draw[e:main,color=DarkGray,->] (Hu4) to (u84);
				
				\draw[e:main,bend right=15,->] (u11) to (u21);
				\draw[e:main,bend right=15,->] (u21) to (u31);
				\draw[e:main,bend right=15,->] (u31) to (u41);
				\draw[e:main,bend right=15,->] (u41) to (u51);
				\draw[e:main,bend right=15,->] (u51) to (u61);
				\draw[e:main,bend right=15,->] (u61) to (u71);
				\draw[e:main,bend right=15,->] (u71) to (u81);
				\draw[e:main,bend right=15,->] (u81) to (u11);
				
				\draw[e:main,bend right=15,->] (u12) to (u22);
				\draw[e:main,bend right=15,->] (u22) to (u32);
				\draw[e:main,bend right=15,->] (u32) to (u42);
				\draw[e:main,bend right=15,->] (u42) to (u52);
				\draw[e:main,bend right=15,->] (u52) to (u62);
				\draw[e:main,bend right=15,->] (u62) to (u72);
				\draw[e:main,bend right=15,->] (u72) to (u82);
				\draw[e:main,bend right=15,->] (u82) to (u12);
				
				\draw[e:main,bend right=15,->] (u13) to (u23);
				\draw[e:main,bend right=15,->] (u23) to (u33);
				\draw[e:main,bend right=15,->] (u33) to (u43);
				\draw[e:main,bend right=15,->] (u43) to (u53);
				\draw[e:main,bend right=15,->] (u53) to (u63);
				\draw[e:main,bend right=15,->] (u63) to (u73);
				\draw[e:main,bend right=15,->] (u73) to (u83);
				\draw[e:main,bend right=15,->] (u83) to (u13);
				
				\draw[e:main,bend right=15,->] (u14) to (u24);
				\draw[e:main,bend right=15,->] (u24) to (u34);%
				\draw[e:main,bend right=15,->] (u34) to (u44);
				\draw[e:main,bend right=15,->] (u44) to (u54);%
				\draw[e:main,bend right=15,->] (u54) to (u64);
				\draw[e:main,bend right=15,->] (u64) to (u74);%
				\draw[e:main,bend right=15,->] (u74) to (u84);
				\draw[e:main,bend right=15,->] (u84) to (u14);%
				
				\draw[e:main,->] (u11) to (u12);
				\draw[e:main,->] (u12) to (u13);
				\draw[e:main,->] (u13) to (u14);
				
				\draw[e:main,->] (u24) to (u23);
				\draw[e:main,->] (u23) to (u22);
				\draw[e:main,->] (u22) to (u21);
				
				\draw[e:main,->] (u31) to (u32);
				\draw[e:main,->] (u32) to (u33);
				\draw[e:main,->] (u33) to (u34);
				
				\draw[e:main,->] (u44) to (u43);
				\draw[e:main,->] (u43) to (u42);
				\draw[e:main,->] (u42) to (u41);
				
				\draw[e:main,->] (u51) to (u52);
				\draw[e:main,->] (u52) to (u53);
				\draw[e:main,->] (u53) to (u54);
				
				\draw[e:main,->] (u64) to (u63);
				\draw[e:main,->] (u63) to (u62);
				\draw[e:main,->] (u62) to (u61);
				
				\draw[e:main,->] (u71) to (u72);
				\draw[e:main,->] (u72) to (u73);
				\draw[e:main,->] (u73) to (u74);
				
				\draw[e:main,->] (u84) to (u83);
				\draw[e:main,->] (u83) to (u82);
				\draw[e:main,->] (u82) to (u81);
				
			\end{pgfonlayer}
		\end{tikzpicture}
		\caption{A sketch of the construction of $D_{H,4}$ in the proof of \cref{thm:generalbutterflyEP}.}
		\label{fig:EPconstruction}
	\end{figure}
	
	First notice that any strongly connected subgraph $K$ of $D_{H,k}$ such that $K$ has a butterfly minor among $\Antichain{H}$ would need to contain a path from $v_{2i}^1$ to $v_{2i-1}^1$.
	To see this observe that any strongly connected subgraph $K'$ of $D_{H,k}$ without such a path would either be a proper subgraph of $H_i$ for some $i\in[1,k]$ and as $\Abs{\V{J}}\geq\Abs{\V{H}}$ and $\Abs{\E{J}}\geq\Abs{\E{H}}$ for all $J\in\Antichain{H}$ $K'$ could not have a butterfly minor among $\Antichain{H}$,
	or $K'$ would be a subgraph of $G_k$.
	But since $H$ is strongly planar, $\Antichain{H}$ cannot contain a butterfly minor of the cylindrical grid by \cref{thm:directedgridminors}.
	Let $P$ be a path from $v_{2i}^1$ to $v_{2i-1}^1$ as mentioned above.
	Note that for every $j\in[1,k]\setminus\Set{i}$, $D_{H,k}-P$ does not contain a path from $v_{2j}^1$ to $v_{2j-1}^1$.
	Hence $D_{H,k}-P$ does not have a butterfly minor among the graphs in $\Antichain{H}$ and thus $D_{H,k}$ cannot have two vertex disjoint subgraphs which each contain a butterfly minor from $\Antichain{H}$.
	On the other hand, let $S\subseteq\V{D_{H,k}}$ be a set of at most $k-1$ vertices.
	Then there must be some $i\in[1,k]$ such that $S$ does not contain a vertex from $H_i$, and there is a directed path $Q$ from $v_{2i}^1$ to $v_{2i-1}^1$ in $G_k$.
	Hence $H_i+\Brace{u_i,v_{2i}^1}+Q+\Brace{v_{2i-1}^1,v_i}-\Brace{u_i,v_i}$ is a subgraph of $D_{H,k}-S$ and it contains $H$ as a butterfly minor and our proof is complete.
\end{proof}

\subsection{Proving \cref{thm:matchingEP}}

The primary goal of this section is the establishment of \cref{thm:matchingEP}.
As pointed out in the introduction, the Erd\H{o}s-P\'osa property for matching minors does not necessarily ask for a hitting set as sometimes deleting a certain conformal set of vertices might destroy some perfect matchings in $B$ and thereby render existing matching minor models non-conformal any more without actually hitting them.
Due to this it is not obvious whether the approach used to prove the undirected analogue can be applied for the reverse of \cref{thm:matchingEP}.
So we take a different route and link the Erd\H{o}s-P\'osa property for matching minors directly to the generalised Erd\H{o}s-P\'osa property for butterfly minors.
By doing so, \cref{thm:matchingEP} follows immediately from \cref{thm:generalbutterflyEP} and the following.

\begin{proposition}\label{thm:bigEP}
	Let $H$ be a matching covered bipartite graph.
	The following statements are equivalent:
	\begin{enumerate}
		\item $H$ has the Erd\H{o}s-P\'osa property for matching minors
		\item $\DirM{H}{M}$ has the generalised Erd\H{o}s-P\'osa property for butterfly minors for some $M\in\Perf{H}$, and
		\item $\DirM{H}{M}$ has the generalised Erd\H{o}s-P\'osa property for butterfly minors for every $M\in\Perf{H}$.
	\end{enumerate}
\end{proposition}

\begin{proof}
	To prove the assertion we take the following route: First we show that (i) implies (ii), then we deduce (iii) from (ii), and finally we show that (iii) implies (i) which completes the proof.
	
	So let us assume $H$ has the Erd\H{o}s-P\'osa property for matching minors and let $\varepsilon_H\colon\N\rightarrow\N$ be the associated function.
	Let us choose $M_H\in\Perf{H}$ and set $D_H\coloneqq\DirM{H}{M_H}$.
	Then let $D$ be any digraph, $B\coloneqq\Split{D}$ and $M\in\Perf{B}$ such that $D=\DirM{B}{M}$.
	Notice that $D$ has $k$ pairwise disjoint subgraphs, all of which contain some member of $\Antichain{D_H}$ as a butterfly minor, if and only if $B$ has $k$ pairwise disjoint $M$-conformal subgraphs all of which contain $H$ as a matching minor.
	So in case $D$ does not have $k$ pairwise disjoint such subgraphs, there must be an $M$-conformal set $S_H\subseteq\V{B}$ with $\Abs{S_H}\leq\Fkt{\varepsilon_H}{k}$ such that $B-S_H$ does not contain $H$ as a matching minor.
	Let $F\coloneqq M\cap\E{\InducedSubgraph{B}{S_H}}$.
	Then, as $S_H$ is $M$-conformal, we have $\Abs{F}\leq\frac{1}{2}\Fkt{\varepsilon_H}{k}$, and by \cref{lemma:excludingantichains} $D-F$ does not contain any digraph from $\Antichain{D_H}$ as a butterfly minor.
	As our choice of $D$ was arbitrary, we may set $f_{D_H}\coloneqq\frac{1}{2}\varepsilon_H$ and thus $D_H$ has the generalised Erd\H{o}s-P\'osa property for butterfly minors.
	
	Now let us assume there is $M_H\in\Perf{H}$ such that $D_H\coloneqq\DirM{H}{M_H}$ has the generalised Erd\H{o}s-P\'osa property for butterfly minors.
	Let $M_H'\in\Perf{H}\setminus\Set{M_H}$ and $D_H'\coloneqq\DirM{H}{M_H'}$.
	Since the generalised Erd\H{o}s-P\'osa property for butterfly minors only cares about $\Antichain{D_H}$ and not necessarily about $D_H$ itself, it suffices to show $\Antichain{D_H}=\Antichain{D_H'}$.
	Consider a digraph $J\in\Antichain{D_H'}$.
	Then every proper butterfly minor $J'$ of $J$ has the property that $\Split{J'}$ does not contain $H$ as a matching minor, while $\Split{J}$ does contain $H$ as a matching minor.
	Therefore, $J$ must be $D_H$-minimal and thus $J\in\Antichain{D_H}$.
	With the same argument one can also obtain $\Antichain{D_H}\subseteq\Antichain{D_H'}$ and our claim follows.
	
	So at last we may assume $\DirM{H}{M_H}$ has the generalised Erd\H{o}s-P\'osa property for butterfly minors for every $M_H\in\Perf{H}$ and let us fix any $M_H\in\Perf{H}$.
	Let $D_H\coloneqq\DirM{H}{M_H}$, and let $f_{D_H}\colon\N\rightarrow\N$ be the function associated with the generalised Erd\H{o}s-P\'osa property for butterfly minors of $D_H$.
	Let $B$ be any bipartite graph with a perfect matching $M$.
	As before, $B$ contains $k$ pairwise disjoint $M$-conformal subgraph, all of which have $H$ as a matching minor, if and only if $D\coloneqq \DirM{B}{M}$ contains $k$ pairwise disjoint subgraphs, all of which have a butterfly minor from $\Antichain{D_H}$.
	So in case $B$ does not have $k$ such $M$-conformal subgraphs, then $D$ does not have $k$ such subgraphs either and thus there must exist a set $S_H\subseteq\V{D}$ with $\Abs{S_H}\leq\Fkt{f_{D_H}}{k}$ such that $D-S_H$ does not have any butterfly minor isomorphic to a member of $\Antichain{D_H}$.
	Note that $S_H\subseteq M$ and thus $\Abs{\V{S_H}}\leq 2\Fkt{f_{D_H}}{k}$, and $\V{S_H}$ is an $M$-conformal set of vertices in $B$.
	Moreover, by \cref{lemma:excludingantichains} we know that $B-\V{S_H}$ does not have $H$ as a matching minor.
	So by setting $\varepsilon_H\coloneqq 2f_{D_H}$ we have found a function that witnesses the Erd\H{o}s-P\'osa property for matching minors of $H$.
\end{proof}

\section{Algorithmic Applications of Perfect Matching Width}\label{sec:algorithms}

The majority of this section is dedicated to solve a matching theoretic version of the $t$-Disjoint Paths Problem.
Towards this goal we need to solve several small subproblems, including the computation of a bounded width perfect matching decomposition with additional properties, which is done in \cref{subsec:decomposition}.
As a special case we obtain \cref{thm:approximatepmw} from this.
In \cref{subsec:linkages} we then discuss the dynamic programming for the matching theoretic linkage problem.
Finally, in \cref{subsec:counting}, we present the dynamic programming necessary for \cref{thm:countmatchings}.
The proofs to all other algorithmic results announced in the introduction can be found in \cref{subsec:rest}.

Before we discuss the computation of a decomposition of bounded perfect width, we need to make some preliminary observations regarding the matching theoretic Linkage Problem.
We start with a formal definition.

\begin{definition}[The Bipartite $k$-Disjoint Alternating Paths Problem]
	Let $B$ be a bipartite graph with a perfect matching, $k\in\N$ a positive integer, and $s_1,\dots,s_t\in V_1$, $t_1,\dots,t_k\in V_2$, called the \emph{terminals}.
	The question whether there exists a perfect matching $M$ of $B$ and internally $M$-conformal paths $P_1,\dots,P_k$ in $B$ which are pairwise internally disjoint and for all $i\in[1,k]$, $P_i$ has endpoints $s_i$ and $t_i$ is called the \emph{bipartite $k$-disjoint alternating paths problem} ($k$-DAPP).
\end{definition}

From here on, most of our effort is directed towards proving the following statement.

\begin{proposition}\label{thm:disjointpaths}
	Let $B$ be a bipartite graph with a perfect matching, $k\in\N$ a positive integer and $\mathcal{I}$ a family of $k$ terminal pairs.
	There exists an algorithm that decides in time $\Abs{\V{B}}^{\Fkt{\mathcal{O}}{k+\pmw{B}^2}}$ the $k$-DAPP with input $\mathcal{I}$ on $B$.
\end{proposition}

Let $B$ be a bipartite graph with a perfect matching and $\mathcal{I}=\Set{\Brace{s_1,t_1},\dots,\Brace{s_k,t_k}}$ a family of terminal pairs.
Let $M$ be a perfect matching of $B$ and $\mathcal{P}=\Set{P_1,\dots,P_k}$ a family of internally disjoint and internally $M$-conformal paths in $B$ such that $P_i$ has endpoints $s_i$ and $t_i$ for every $i\in[1,k]$.
We call $\Brace{M,\mathcal{P}}$ a \emph{solution} for $\mathcal{I}$.
Let $W\subseteq\E{B}$ be a matching.
A solution $\Brace{M,\mathcal{P}}$ for $\mathcal{I}$ in $B$ \emph{extends} $W$ if $W\subseteq M$ and every terminal is matched by some edge in $W$.

A problem that needs to be addressed before we go any further is that our terminals are not necessarily distinct.
In some cases this might lead to problems for the way our algorithm works.
Before we continue, let us discuss how we get around this issue. Notice the following:
Let $x\in \V{B}$, be a vertex that occurs in at least one pair of $\mathcal{I}$ and let us denote the total number of occurrences of $x$ as a terminal by $\Multiplicity{a}$.
Then in every solution $\Brace{M,\mathcal{P}}$, every path $P\in\mathcal{P}$ that connects $x$ to some other terminal must end in an edge that is not contained in $M$ and connects $x$ to some neighbour $x'$ of $x$.
Moreover, the edge of $M$ covering $x$ cannot be contained in any $P\in\mathcal{P}$.
Hence we may pick a collection of $\Multiplicity{x}$ many neighbours of $x$, select an extendable matching $W'$ that covers the selected vertices, but not $x$, and now for each of these edges pick the endpoint not adjacent to $x$.
Each of these picked vertices belong to the same colour class as $x$.
For our graph $B$ let $V_1'\subseteq V_1\setminus\Set{s_1,\dots,s_k}$ and $V_2'\subseteq V_2\setminus\Set{t_1,\dots,t_k}$ be selections of such vertices together with the extendable set of all matching edges $W'$ covering these new vertices.
Note that $W'$ must be chosen such that $W\cup W'$ is extendable.
For every $\Brace{s_i,t_i}\in\mathcal{I}$ now select a vertex $s'_i\in V_1'$ and $t'_i\in V_2'$ that is a neighbour of $s_i$, $t_i$ respectively.
Then we have formed a distinct family $\mathcal{I}'$ of $k$ terminal pairs and therefore we may now consider an instance of the bipartite $k$-matching linkage problem instead.

Let us now formalise the above discussion.
Given a family of terminal pairs $\mathcal{I}=\Set{\Brace{s_1,t_1},\dots,\Brace{s_k,t_k}}$ and an extendable set $W$ such that all terminals are matched by $W$ and every edge of $W$ matches a terminal, we call a pair $\Brace{\mathcal{I}',W'}$ a \emph{$\Brace{\mathcal{I},W}$-proxy} if
\begin{enumerate}
	
	\item $\mathcal{I}'=\Set{\Brace{s'_1,t'_1},\dots,\Brace{s'_k,t'_k}}$ is a family of $k$ terminal pairs where $s_i\neq s_j$ and $t_i\neq t_j$ for every choice of distinct values for $i,j\in[1,k]$ (we call such a family \emph{distinct}),
	
	\item $W'\cup W$ is extendable, every terminal of $\mathcal{I}'$ is matched by some edge of $W'$, every edge of $W'$ matches a terminal of $\mathcal{I}'$, and $W\cap W'=\emptyset$, and
	
	\item for every $i\in[1,k]$, if $s'_iv\in W'$, then $v$ is a neighbour of $s_i$ and if $vt'_i\in W'$, then $v$ is a neighbour of $t_i$.
	
\end{enumerate}

It might happen, that $s_i$ and $t_i$ of the original instance are already adjacent, in such cases, we might have to consider additional cases of smaller instances, where the edge $s_it_i$ is already one of the paths in a possible solution.
Indeed, without loss of generality, we may always assume $s_it_i$ to be part of our solution and thus the terminal pair $\Brace{s_i,t_i}$ does not need to be considered.
Hence we may assume all terminal pairs to be non-adjacent.

A perfect matching decomposition $\Brace{T,\delta}$ is \emph{safe} for $W$ and $\mathcal{I}$ if every $W$-extending solution $\mathcal{P}$ for $\mathcal{I}$ satisfies the following inequality for every $e\in\E{T}$:
\begin{align*}
	\Abs{\CutG{B}{e}\cap\bigcup_{P\in\mathcal{P}}\E{P}}\leq2\Width{T,\delta}.
\end{align*}

The high level strategy of our algorithm is as follows:
\begin{itemize}
	\item We choose an extendable matching $W\subseteq\E{B}$ of size at most $2k$ such that all terminals of $\mathcal{I}$ are covered.
	
	\item Next choose a $\Brace{\mathcal{I},W}$-proxy $(\mathcal{I}',W')$.
	
	\item Then we compute a perfect matching decomposition $\Brace{T,\delta}$ for $B-\V{W}$ that is safe for $(\mathcal{I}',W')$ and its width is bounded in a function of $\pmw{B}$ and $k$.
	
	\item We apply dynamic programming on $\Brace{T,\delta}$ in order to either find a solution that extends $W'$ or refute the existence of such a solution.
	
	\item Finally, if for some $W$ and some $\Brace{\mathcal{I},W}$-proxy we find a solution we extend it to a solution for $\mathcal{I}$ and $W$ and return ``Yes'', otherwise we return ``No''.
\end{itemize}
Let $\Abs{\V{B}}=n$.
If $\Fkt{f_1}{\pmw{B},k,n}$ describes the time needed to compute $\Brace{T,\delta}$ and $\Fkt{f_2}{\pmw{B},k,n}$ describes the time necessary for the dynamic programming on $\Brace{T,\delta}$, the overall running time of our algorithm can then be expressed by $\Fkt{\mathcal{O}}{n^{2k}\cdot n^{4k}\cdot\Fkt{f_1}{\pmw{B},k,n}\cdot\Fkt{f_2}{\pmw{B},k,n}}$ where the constants only depend on $k$ and $\pmw{B}$.

Once the functions $f_1$ and $f_2$ are established, \cref{thm:disjointpaths} follows as an immediate consequence of the high-level approach described above.
Indeed, please note that, by slightly modifying the proofs below, one can obtain the following more general result:

\begin{corollary}\label{cor:disjointpathswithprescribedset}
	Let $B$ be a bipartite graph with a perfect matching, $k\in\mathbb{N}$ an integer, $\mathcal{I}$ a family of $k$ terminal pairs, and $F\subseteq\E{B}$ an extendable set.
	There exists an algorithm that decides in time $\Abs{\V{B}}^{\mathcal{O}{\Brace{k+\pmw{B}}^2}}$ whether there exists an $F$-extending solution for $\mathcal{I}$ or not.
\end{corollary}

As an immediate consequence, if $D=\DirM{B}{M}$ is some digraph, by choosing $F=M$ \cref{cor:disjointpathswithprescribedset} together with \cref{thm:cycwandpmw,thm:cycwdtw} implies the original result on the directed disjoint path problem for digraphs of bounded directed treewidth in \cite{johnson2001directed}.

\subsection{Computing A Perfect Matching Decomposition}\label{subsec:decomposition}

Some preliminary results are needed.
For one, we need to be able to check for given $W\subseteq\E{B}$ whether there exists a perfect matching extending $W$ if this is true, $W$ is called an \emph{extendable set}.
This boils down to checking if $B-\V{W}$ has a perfect matching.
And second, we must be able to compute a perfect matching decomposition of bounded width.

The first problem can be solved in polynomial time by Edmonds' famous Blossom Algorithm \cite{edmonds1965paths}, or, since we work on bipartite graphs, by the Hungarian Method \cite{kuhn1955hungarian}, so this part will not be much of a concern to us.

For the second part, we make use of the following theorem.

A directed tree decomposition $\Brace{T,\beta,\gamma}$ for a digraph $D$ is \emph{nice} if for every $\Brace{t',t}\in\E{T}$,
\begin{enumerate}
	
	\item $\Fkt{\beta}{T_t}$ induces a strong component of $D-\Fkt{\gamma}{t',t}$, and
	
	\item $\Fkt{\gamma}{t',t}\cap\Fkt{\beta}{T_t}=\emptyset$.
	
\end{enumerate}

\begin{theorem}[\cite{campos2019adapting}]\label{thm:approximatedtw}
	Let $D$ be a digraph, $k\in\N$, and $\dtw{D}\leq k$.
	There exists an algorithm with running time $2^{\Fkt{\mathcal{O}}{k\log k}}n^{\Fkt{\mathcal{O}}{1}}$ that computes a nice directed tree-decomposition of width at most $3k-2$ for $D$.
\end{theorem}

Let $B$ be a bipartite graph with a perfect matching $M$ and $D\coloneqq\DirM{B}{M}$.
In light of \cref{thm:approximatedtw}, it would be enough to compute a perfect matching decomposition of bounded width for $B$ from a directed tree decomposition of bounded width for $D$ which we already know how to do in polynomial time by the results from \cite{hatzel2019cyclewidth}.
We would like to maintain a bit of this niceness in the perfect matching decomposition we produce.

Let $B$ be a bipartite graph with a perfect matching.
A perfect matching decomposition $\Brace{T,\delta}$ of width $w$ is \emph{nice} if $T$ is rooted at some vertex $r\in \V{T}$ and 
\begin{enumerate}
	
	\item $\V{T-r}$ can be partitioned into four sets of vertices:
	\begin{itemize}
		\item The leaves, $\Leaves{T}$, which are the vertices of degree one.
		\item The \emph{basic} vertices, $\Basic{T}$, which are those vertices $t\in\V{T}$ whose successors are leaves of $T$.
		\item The \emph{joins}, $\Joins{T}$, which are the vertices $t\in\V{T}$ with two distinct successors $t_1$ and $t_2$ such that there is no edge from $V_2\cap\Fkt{\delta}{T_{t_1}}$ to $V_1\cap\Fkt{\delta}{T_{t_2}}$, and $\InducedSubgraph{B}{\Fkt{\delta}{T_{t_1}}}$ is elementary.
		\item The \emph{guards}, $\Guards{T}$, which are the vertices $t\in\V{T}$ satisfying one of the following properties:
		\begin{itemize}
			\item $\Abs{\Fkt{\delta}{T_t}}\leq 2k$ and $\Fkt{\delta}{T_t}$ is conformal (Type 1), or
			\item $t$ has two distinct successors $t_1$ and $t_2$ such that
			$t_{1}$ is a guard of Type 1
			and $t_2$ either is a join, or $\InducedSubgraph{B}{\Fkt{\delta}{T_{t_2}}}$ is conformal and elementary. (Type 2)
		\end{itemize}
	\end{itemize}
	\item if $r$ is not a leaf of $T$ for every successor $t$ of $r$ one of the following holds:
	\begin{itemize}
		\item $t$ is a guard of Type 1, or
		\item $t$ either is a join, or $\InducedSubgraph{B}{\Fkt{\delta}{T_{t}}}$ is conformal and elementary, and
	\end{itemize}
	the successors of $r$ of this type can be sorted as $t_1,\dots,t_h$, $h\leq 3$ such that if $1\leq i<j\leq h$, then there is no edge from $V_1\cap\Fkt{\delta}{T_{t_j}}$ to $V_2\cap\Fkt{\delta}{T_{t_i}}$.
\end{enumerate}
Given a distinct set $\mathcal{I}$ of terminal pairs for the $k$-DAPP on $B$ and an extendable set $W\subseteq\E{B}$ matching all terminals such that if $e\in W$, then an endpoint of $e$ is a terminal, we call a perfect matching decomposition $\Brace{T,\delta}$ a \emph{$\Brace{\mathcal{I},W}$-decomposition} for $B$, if it is nice and safe for $\mathcal{I}$ and $W$.

In the following we describe how to obtain a $\Brace{\mathcal{I},W}$-decomposition for a bipartite graph $B$ with a perfect matching.
As a base of our algorithm, we are going to use \cref{thm:approximatedtw} and then manipulate the obtained decomposition in order to create a nice perfect matching decomposition.
A tuple $\Brace{T,\beta,\gamma}$ is called a \emph{proto-directed tree decomposition} for the digraph $D$ if it satisfies all the conditions of a directed tree decomposition except that we allow empty bags and still every vertex of $D$ must be contained in exactly one bag of $\Brace{T,\beta,\gamma}$.

A proto-directed tree decomposition $\Brace{T,\beta,\gamma}$ of width $w$ for a digraph $D$ is \emph{prepared} if
\begin{enumerate}
	\item $T$ is subcubic,
	\item if $t\in\V{T}$ has a unique successor $t'$, then $\Fkt{\beta}{T_{t'}}$ induces a strong component of $D-\Fkt{\gamma}{t,t'}$ or contains at most $w+1$ vertices, and
	\item if $t\in\V{T}$ has two distinct successors $t_1$ and $t_2$, then 
	\begin{itemize}
		\item $\Fkt{\beta}{T_{t_1}}$ either contains at most $w+1$ vertices and $\Fkt{\beta}{T_{t_2}}$ also either has at most $w+1$ vertices or induces a strongly connected subgraph of $D-\Fkt{\gamma}{t,t_2}$, or
		\item $\Fkt{\beta}{T_{t_1}}$ induces a strongly connected subgraph of $D-\Fkt{\gamma}{t,t_1}$ and there is no directed edge with tail in $\Fkt{\beta}{T_{t_2}}$ and head in $\Fkt{\beta}{T_{t_1}}$ in $D$.
	\end{itemize}
\end{enumerate}

\begin{lemma}\label{lemma:preparedtw}
	Let $D$ be a digraph and $\dtw{D}\leq w$.
	There exists an algorithm with running time $2^{\Fkt{\mathcal{O}}{w\log w}}n^{\Fkt{\mathcal{O}}{1}}$ that computes a prepared proto-directed tree-decomposition of width at most $3w-2$ for $D$.
\end{lemma}

\begin{proof}
	Let $\Brace{T_0,\gamma_0,\beta_0}$ be the nice directed tree-decomposition obtained via the algorithm in \cref{thm:approximatedtw}.
	Let us call a proto-directed tree-decomposition where every vertex of degree at most three satisfies the axioms of a prepared proto-directed tree decomposition and every other vertex satisfies the axioms of a nice directed tree decomposition \emph{almost prepared}.
	Clearly $\Brace{T_0,\gamma_0,\beta_0}$ is almost prepared.
	Now let $\Brace{T_j,\gamma_j,\beta_j}$ be an almost prepared proto-directed tree-decomposition.
	
	Pick any vertex $t\in\V{T}$ of degree more than three and let $t_1,\dots,t_{\ell}$ be its successors.
	Then $\Fkt{\beta_j}{T_{j,t_i}}$ induces a strong component of $D-\Fkt{\gamma_j}{t,t_i}$ for all $i\in[1,\ell]$.
	Indeed, $\Fkt{\beta_j}{T_{j,t_i}}$ induces a strong component of $D-\Fkt{\Gamma_j}{t}$ for all $i\in[1,\ell]$, without loss of generality let us assume that the $t_i$ are numbered in such a way that for all $1\leq i<k\leq\ell$ there is no directed edge with tail in $\Fkt{\beta_j}{T_{j,t_k}}$ and head in $\Fkt{\beta_j}{T_{j,t_i}}$.
	We define a proto-directed tree decomposition $\Brace{T_{j+1},\beta_{j+1},\gamma_{j+1}}$ as follows.
	Let $T_{j+1}$ be the arborescence obtained from $T_j$ by introducing a new vertex $t'$, the edge $\Brace{t,t'}$ and replacing $\Brace{t,t_i}$ by $\Brace{t',t_i}$ for all $i\in[2,\ell]$.
	Then $\Fkt{\beta_{j+1}}{t''}\coloneqq\Fkt{\beta_j}{t''}$ for all $t''\in\V{T_j}$ and $\Fkt{\beta_{j+1}}{t'}\coloneqq\emptyset$.
	Moreover, let $\Fkt{\gamma_{j+1}}{e}\coloneqq\Fkt{\gamma_j}{e}$ for all $e\in\E{T_j}\setminus\Set{\Brace{t,t_2},\dots,\Brace{t,t_{\ell}}}$, $\Fkt{\gamma_{j+1}}{t,t'}\coloneqq \Fkt{\gamma_j}{d,t}\cup\Fkt{\beta_j}{t}$, where $\Brace{d,t}$ is the unique ingoing edge at $t$ in $T_j$, and $\Fkt{\gamma_{j+1}}{t',t_i}\coloneqq\Fkt{\gamma_j}{t,t_i}$ for all $i\in[2,\ell]$.
	Clearly $\Width{T_{j+1},\beta_{j+1},\gamma_{j+1}}\leq\Width{T_j,\beta_j,\gamma_j}$, so we just need to show that $\Brace{T_{j+1},\gamma_{j+1},\beta_{j+1}}$ is indeed a proto-directed tree decomposition.
	To be more precise, we only need to show that $\Fkt{\gamma_{j+1}}{t,t'}$ is a valid guard for $\Fkt{\beta_{j+1}}{T_{j+1,t'}}$.
	Let $P$ be any directed walk starting and ending on a vertex of $\Fkt{\beta_{j+1}}{T_{j+1,t'}}$ while containing a vertex of $D-\Fkt{\beta_{j+1}}{T_{j+1,t'}}$.
	If $P$ lies in $\Fkt{\beta_{j+1}}{T_{j+1,t}}$, then $P$ must contain a vertex of $\Fkt{\beta_{j+1}}{t}$ since there is no edge from $\Fkt{\beta_{j+1}}{T_{j+1,t'}}$ to $\Fkt{\beta_{j+1}}{T_{j+1,t_1}}$ by construction.
	So if $P$ avoids $\Fkt{\beta_{j+1}}{t}$, then $P$ must contain a vertex of $D-\Fkt{\beta_{j+1}}{T_{j+1,t}}$ and thus, it must contain a vertex of $\Fkt{\gamma_{j+1}}{d,t}$.
	
	Then $\Brace{T_{i+1},\gamma_{i+1},\beta_{i+1}}$ is almost prepared and has less vertices of degree at least four that $\Brace{T_i,\gamma_i,\beta_i}$.
	In fact, after at most $\Abs{\V{T_0}}$ steps we have obtained a prepared proto-directed tree decomposition.
\end{proof}

Given a bipartite graph $B$ with a perfect matching, $\mathcal{I}=\Set{\Brace{s_1,t_1},\dots,\Brace{s_k,s_k}}$ a distinct family of terminal pairs, and an extendable $W\subseteq\E{B}$ matching all terminals, we call a set $F\subseteq\Choose{\V{B}}{2}$ a \emph{$W$-completion}, if for every $W$-extending solution $\Brace{M,\mathcal{P}}$, the graph induced by the edge set
\begin{align*}
	\Brace{F\cup W\cup\bigcup_{P\in\mathcal{P}}\E{P}}\setminus\CondSet{xy\in W}{x=s_i\text{ and }y=t_i\text{ for some }i\in[1,k]}
\end{align*}
consists exclusively of $M$-alternating cycles.
Please note that, by definition, $\Abs{F}\leq k$ for all $W$-completing $F$.

\begin{lemma}\label{lemma:Wcompletion}
	Let $B$ be a bipartite graph with a perfect matching, $\mathcal{I}=\Set{\Brace{s_1,t_1},\dots,\Brace{s_k,t_k}}$ a family of distinct terminal pairs, and $W\subseteq\E{B}$ an extendable set covering all terminals such that if $e\in W$, then an endpoint of $e$ is a terminal.
	Then a $W$-completion $F$ can be found in linear time.
\end{lemma}

\begin{proof}
	We obtain $F$ as follows:
	Initialise $F$ and $U$ with $\emptyset$.
	Pick some $s_i\in\V{\mathcal{I}}\setminus U$ and add it to $U$, let $e\in W$ be the edge of $W$ covering $s_i$ and let $x$ be its other endpoint.
	Next we have to consider several cases.
	If $x=t_i$ we just add $t_i$ to $U$ and continue with a new $s_{i'}\in\V{\mathcal{I}}\setminus U$, in this case nothing else must be done.
	If $x=t_j$ for some $j\in[1,k]\setminus\Set{i}$ add $s_j$ and $t_j$ to $U$, select the edge $e\in W$ to be the edge covering $s_j$, let $x$ be its other endpoint.
	In case $s_j$ was already in $U$ before, it must be equal to the original $s_{i'}$ this cycle of the process was started with.
	Then every $W$-extending solution clearly closes an alternating cycle and we may proceed with a new $s_i\in\V{\mathcal{I}}\setminus U$.
	Otherwise reiterate the process with $s_j$ in the role of $s_i$, in this case, we are still within the same cycle.
	And at last, if $x\notin\V{\mathcal{I}}$ we may consider $t_i$ and the edge $e'\in W$ covering $t_i$, let $y$ be its other endpoint.
	
	The process here is rather similar to what came before.
	Clearly $y\neq s_i$ and, moreover, $y\notin U$.
	If however $y=s_j$ for some $i\in[1,k]\setminus\Set{i}$, then add $s_j$ and $t_j$ to $U$, set $e'\in W$ to be the edge covering $t_j$ and reiterate the process with $t_j$ in the role of $t_i$.
	If, on the other hand, $y\notin\V{\mathcal{I}}$, let $s_{i'}$ be the vertex this cycle was started with.
	Now any solution, together with the edges of $W$, produces a path $P'$ with endpoints $x$ and $y$ that can be made into an alternating cycle by adding the edge $xy$ to $B$ if it does not already exist.
	Hence we add $xy$ to $F$ and proceed with the next $s_i\in\V{\mathcal{I}}\setminus U$.
	Once $U=\V{\mathcal{I}}$ our set $F$ is $W$-completing by the discussion above.
\end{proof}

Adding a $W$-completion $F$ to our graph $B$ should not change its perfect matching width by too much.

\begin{observation}\label{obs:addadgestodtw}
	Let $D$ be a digraph and $F\subseteq \E{\Complement{D}}$ a set of edges not in $D$.
	Then $\dtw{D+F}\leq\dtw{D}+\Abs{F}$.
\end{observation}

\begin{proof}
	Let $\Brace{T,\beta,\gamma}$ be a directed tree decomposition for $D$ of optimal width and let $S\subseteq\V{D}$ be the set of tails of the edges in $F$.
	Now add $S$ to every guard of $\Brace{T,\beta,\gamma}$.
	Clearly this increases the width of our decomposition by at most $\Abs{S}\leq\Abs{F}$ and the result is a directed tree decomposition for $D+F$.
\end{proof}

\begin{lemma}\label{lemma:safedecompositions}
	Let $B$ be a bipartite graph, $\mathcal{I}=\Set{\Brace{s_1,t_1},\dots,\Brace{s_k,t_k}}$ a distinct family of terminal pairs, and $W\subseteq\E{B}$ an extendable set covering all terminals such that if $e\in W$, then an endpoint of $e$ is a terminal.
	Let $\pmw{B}\leq w$ and $n\coloneqq\Abs{\V{B}}$.
	There exists an algorithm with running time $2^{\Fkt{\mathcal{O}}{\Brace{w^2+k}\log (w^2+k)}}n^{\Fkt{\mathcal{O}}{1}}$ that produces a $\Brace{\mathcal{I},W}$-decomposition of width at most $432w^2+864w+22+6k$ for $B$.
\end{lemma}

\begin{proof}
	Let us first compute a perfect matching $M$ extending $W$ and a $W$-completing set $F$.
	Clearly both can be done in polynomial time.
	Now let $D\coloneqq\DirM{B}{M}$, $D'\coloneqq\DirM{B+F}{M}$, and $F'\coloneqq\E{D'}\setminus\E{D}$.
	Then $F'$ corresponds to the edges in $F$ added to $B$.
	By \cref{sec:pmwanddtw} we obtain $\dtw{D}\leq 72\pmw{B}^2+144\pmw{B}+9$.
	With \cref{obs:addadgestodtw} this means $\dtw{D+F}\leq72\pmw{B}^2+144\pmw{B}+9+k$.
	Observe that $D+F'=D'$.
	Now let $\Brace{T,\beta,\gamma}$ be the prepared proto-directed tree decomposition of width at most $216w^2+432w+10+3k$ obtained from \cref{lemma:preparedtw}.
	
	In the next step, we show how to obtain a nice perfect matching decomposition for $B+F$ from $\Brace{T,\beta,\gamma}$.
	First of all, since $D'=\DirM{B+F}{M}$, every vertex $v$ of $D'$ corresponds to an edge $e_v$ of $M$, let us denote the endpoint of $e_v$ in $V_1$ by $a_v$ and the other endpoint by $b_v$.
	Since $T$ is an arborescence, it already is rooted at some vertex, say $r$.
	In what follows, we explain how to manipulate this tree $T$ and how we define a bijection $\delta'$ step by step in order to create a cycle decomposition for $D'$ of bounded width.
	This cycle decomposition is then translated into a perfect matching decomposition with the required properties.
	
	Let $t\in\V{T}$ be any vertex with $\Fkt{\beta}{t}\neq\emptyset$ and let $d$ be its predecessor.
	We only discuss the case in which $t$ is not the root, but the other case can be solved in a similar way.
	There are three possible cases, depending on the number of successors $t$ has in $T$.
	
	\textbf{Case 1:} Vertex $t$ is a leaf in $T$.
	
	In this case, let us construct a rooted cubic tree $T'$ with root $t$ and otherwise disjoint from $T$ such that $T'$ has exactly $\Abs{\Fkt{\beta}{t}}$ many leaves.
	Add $T'$ to $T$ and extend the bijection $\delta'$ such that the restriction of $\delta'$ to the leaves of $T'$ is a bijection between said leaves and $\Fkt{\beta}{t}$.
	Then the edge $\Brace{d,t}$ induces an edge cut $\CutG{D}{X}$ with $X=\Fkt{\beta}{t}$, and thus $\CycPor{\CutG{D}{X}}\leq \Abs{\Fkt{\beta}{t}}\leq 216w^2+432w+11+3k$ and every edge of $T'$ induces an edge cut with even smaller cycle porosity.
	Mark every non-leaf vertex in $T'$ as a guard.
	This mark does not hold a special significance for this decomposition, but will be used in the second half of this proof to show that we can construct a nice perfect matching decomposition.
	
	\textbf{Case 2:} Vertex $t$ has a unique successor $d'$ in $T$.
	
	Let $T'$ be a rooted cubic tree with root $t'$, completely disjoint from $T$ and exactly $\Abs{\Fkt{\beta}{t}}$ leaves, add $T'$ to $T$ together with the edge $\Brace{t,t'}$ and extend $\delta'$ for the leaves of $T'$ as above.
	With the same arguments we obtain bounds on the cycle porosity of every edge of $T'$ and the edge $\Brace{t,t'}$.
	Mark every non-leaf vertex in $T'$ as a guard.
	
	\textbf{Case 3:}  Vertex $t$ has two successors in $T$.
	
	Here we first subdivide the edge $\Brace{d,t}$, i.\@e.\@ we replace it by the directed path $\Brace{d,t',t}$ where $t'$ is a vertex newly introduced to $T$.
	Then we create a rooted cubic tree $T'$ with exactly $\Abs{\Fkt{\beta}{t}}$ leaves, rooted at $t''$ that is disjoint from the modified tree $T$ and introduce the edge $\Brace{t',t''}$.
	Afterwards, we extend $\delta'$ to the leaves of $T'$ as before and again obtain bounds on the cycle porosity of the edge cuts induced by the edges of $T'$ and $\Brace{t',t''}$.
	Mark $t'$ and every non-leaf vertex in $T'$ as guards.
	
	Let $\Brace{T',\delta'}$ be the cycle decomposition for $D'$ obtained by applying the above constructions to all vertices of $\Brace{T,\beta,\gamma}$ with non-empty bags.
	Since $\Brace{T,\beta,\gamma}$ was a proto-directed tree decomposition of width at most $216w^2+432w+10+3k$ it is straight forward to prove that all edges of $T'$ that were not discussed in the construction induce, with respect to $\delta'$, edge cuts of cycle porosity at most $432w^2+864w+22+6k$ in $D'$.
	Hence $\Width{T',\delta'}\leq 432w^2+864w+22+6k$.
	
	Now let us create a new rooted tree $T''$ from $T'$ by introducing for every leaf $t$ of $T'$ two new successors $t_{V_1}$ and $t_{V_2}$ and defining $\Fkt{\delta}{t_{V_1}}\coloneqq a_{\Fkt{\delta'^{-1}}{t}}$ and $\Fkt{\delta}{t_{V_2}}\coloneqq b_{\Fkt{\delta'^{-1}}{t}}$.
	The result is a perfect matching decomposition $\Brace{T'',\delta}$ for $B+F$.
	The bound $\Width{t,\delta}\leq 432w^2+864w+22+6k$ follows from \cref{obs:matporandcycpor}.
	
	Additionally, since we started out with a prepared proto-directed tree decomposition, it is relatively straight forward to check that $\Brace{T'',\delta}$ is nice.
	For the sake of completion we discuss this in the following paragraph.
	
	Let $t\in\Fkt{V}{T''}$ be any non-root vertex.
	
	\textbf{Case A:} Vertex $t$ is a leaf.
	
	Here we are done immediately since $t\in\Leaves{T''}$.
	
	\textbf{Case B:} Vertex $t$ is adjacent to a leaf.
	
	In this case, by construction of $T''$ from $T'$, $t$ must have exactly two successor which both are leaves and thus $t\in\Basic{T''}$.
	
	\textbf{Case C:} Vertex $t$ is not adjacent to a leaf, but has been marked as a guard in the construction of $\Brace{T',\delta'}$.
	
	First let us assume $t\in\V{T}$, in this case, $t$ must have been a leaf of $T$ and thus $\Fkt{\beta}{t}$ directly corresponds to $\Fkt{\delta}{T''_t}$ and $t$ is indeed a guard of $T''$.
	Otherwise, $t$ must have been introduced during the construction of $T'$ from $T$ and thus there must be a vertex $d\in\V{T}$ with $\Fkt{\beta}{d}\neq\emptyset$ that is responsible for the introduction of $t$.
	In case $d$ has a unique successor in $T$, $t$ belongs to the newly introduced rooted cubic tree and thus $\Fkt{\delta}{T''_t}$ has at most $2w$ vertices and is conformal.
	Thus $t$ is indeed a guard.
	Hence we may assume $d$ to have two successors in $T$.
	Then we subdivided the incoming edge at $d$ with a new vertex, say $d'$ and added a rooted cubic tree $R$ with root $d''$ as a new successor of $d'$.
	If $t\in\V{R}$ we are done by the same argument as above.
	If $t=d'$, then the successors of $d'$ are $d$ and $d''$.
	We have already seen that $d''$ is a guard of $T$ and since $\Brace{T,\beta,\gamma}$ is a prepared proto-directed tree decomposition, $d$ must be a join as $d$ has two successors $t_1$ and $t_2$ satisfying the appropriate requirements.
	
	\textbf{Case D:} Vertex $t$ is a vertex of the original $T$ but has not been marked as a guard during the construction of $T'$.
	
	This means in particular that $t$ is not a leaf of $T$ and does not have a unique successor.
	Indeed, in this case, $t$ must have exactly two successors $t_1$ and $t_2$.
	We may assume $t_1$ and $t_2$ to be ordered such that there is no edge from $\Fkt{\beta}{T_{t_2}}$ to $\Fkt{\beta}{T_{t_1}}$.
	Let $t'_1$ and $t'_2$ be the two successors of $t$ in $T''$, then it follows that there is no edge from $V_2\cap\Fkt{\delta}{T''_{t_2}}$ to $V_1\cap\Fkt{\delta}{T''_{t_1}}$.
	Moreover, with the same argument $\Fkt{\beta}{T_{t_1}}$ is strongly connected and thus $\Fkt{\delta}{T''_{t_1}}$ is elementary.
	Hence $t$ is a join.
	
	This completes the argument and thus $\Brace{T'',\delta}$ is nice.
	What is left to show is that $\Brace{T'',\delta}$ is safe for $\mathcal{I}$ and $W$.
	Note that the width of $\Brace{T'',\delta}$ cannot increase by deleting $F$ and thus it also is a perfect matching decomposition for $B$ with the same bound on its width.
	
	We claim that $\Brace{T'',\delta}$ is safe for $W$ and $\mathcal{I}$.
	Suppose there exists a solution $M',\mathcal{P}$ and an edge $e\in\E{T''}$ such that
	\begin{align*}
		\Abs{\CutG{B}{e}\cap\bigcup_{P\in\mathcal{P}}\E{P}}>864w^2+1728w+44+12k.
	\end{align*}
	With $F$ being $W$-completing, $F\cup W\cup\bigcup_{P\in\mathcal{P}}\E{P}$ induces a family $\mathcal{C}$ of pairwise disjoint $M'$-conformal cycles in $B+F$.
	Then let $\mathcal{C}'$ be the collection of all cycles in $\mathcal{C}$ with edges in $\CutG{B}{e}$.
	Let $\E{\mathcal{C}'}\coloneqq\bigcup_{C\in\mathcal{C}'}\E{C}$.
	Since $\Width{T'',\delta}\leq 432w^2+864w+22+6k$, at most $432w^2+864w+22+6k$ of the edges in $\E{\mathcal{C}'}\cap\CutG{B}{e}$ can belong to $M'$.
	Hence $\Abs{\Brace{\E{\mathcal{C}}\cap\CutG{B}{e}}\setminus M'}>432w^2+864w+22+6k$.
	Consider the perfect matching $M''\coloneqq M'\Delta\E{\mathcal{C}'}$ of $B+F$.
	Note that this is the point where we need $\mathcal{I}$ to be a distinct family.
	Then $\Abs{\CutG{B}{e}\cap M''}>432w^2+864w+22+6k$ contradicting our assumption.
	So our claim follows.
\end{proof}

Note that in case $t=0$, \cref{lemma:safedecompositions} implies the existence of an $\FPT$-approximation algorithm that produces a nice perfect matching decomposition for bipartite graphs $B$ with a perfect matching.
Hence we have established \cref{thm:approximatepmw}.

\subsection{The Matching Linkage Problem}\label{subsec:linkages}

With \cref{lemma:safedecompositions} we have fixed
\begin{align*}
	\Fkt{f_1}{\pmw{B},k,\Abs{\V{B}}}\coloneqq 2^{\Fkt{\mathcal{O}}{\Brace{\pmw{B}^2+k}\log (\pmw{B}^2+k)}}\Abs{\V{B}}^{\Fkt{\mathcal{O}}{1}},
\end{align*}
so from now on we will only be concerned with the dynamic programming on $\Brace{\mathcal{I},W}$-decompositions.

In most parts, we lean on the algorithm for the directed disjoint paths problem developed by Johnson et al.\@ for digraphs of bounded directed treewidth \cite{johnson2001directed}.
However, we face several challenges here.
The first one is that we cannot assume that there is no perfect matching $M$ for which some internally $M$-conformal path $P$ exists with $\Abs{\CutG{B}{e}\cap\E{P}}\gg2\Width{T,\delta}$.
The only thing we can be sure of is that no such path can be part of our solution.
Second, while we are exclusively interested in perfect matchings of $B$ that extend $W$, there might still be an exponential number of them, and thus we must store additional information in order to cope with this fact.

Let $B$ be a bipartite graph with a perfect matching, $\mathcal{I}$ a distinct family of $k$ terminal pairs for the $k$-DAPP in $B$, $W\subseteq\E{B}$ an extendable set matching all terminals such that if $e\in W$, then an endpoint of $e$ is a terminal, and $F$ a $W$-completion.
A subgraph $L$ of $B$ is called a \emph{linkage} if there exists a perfect matching $M$ and a family of pairwise internally disjoint internally $M$-conformal paths $\mathcal{P}$ such that $L=\bigcup_{P\in\mathcal{P}}P$, and $L$ has exactly $\Abs{\mathcal{P}}$ components.
A linkage $L$ is a \emph{$\Brace{\mathcal{I},W}$-linkage} if there exists a solution $(M,\mathcal{P})$ for $\mathcal{I}$ in $B$ extending $W$ such that $L=\bigcup_{P\in\mathcal{P}}P$.
Please note that for a $\Brace{\mathcal{I},W}$-linkage $L$ the corresponding $W$-extending solution $\Brace{M,\mathcal{P}}$ is uniquely determined apart from the edges of $M\cap\Brace{\E{B}\setminus\bigcup_{P\in\mathcal{P}}\E{P}}$.
A \emph{part} of $L$ is a subgraph $L'\subseteq L$ such that some path $P\in\mathcal{P}$ exists with $L'\subseteq P$.
Let $X\subseteq\V{B}$, a \emph{part of $L$ in $X$} is a component of $\InducedSubgraph{L}{X\cap\V{L}}$, we denote the set of all parts of $L$ in $X$ by $\Parts{L}{X}$.

We say that a linkage $L$ in $G$ is \emph{$\Brace{k,w}$-limited} in $X$ and $G$, for some integer $w$, if for every set $Y\subseteq X$ with $\MatPor{\CutG{B}{Y}}\leq w$ we have $\Abs{\Parts{L}{Y}}\leq k+w$.

\begin{lemma}\label{lemma:limited}
	Let $B$ be a bipartite graph with a perfect matching, $\mathcal{I}$ a distinct family of $k$ terminal pairs, $W\subseteq\E{B}$ an extendable set covering all terminals such that if $e\in W$, then an endpoint of $e$ is a terminal, and $F$ a $W$-completion.
	Let $X\subseteq\V{B}$ and $L$ a $\Brace{\mathcal{I},W}$-linkage in $B$ as well as $w$ a positive integer, then $L$ is $\Brace{k,w}$-limited in $X$ and $B+F$.
\end{lemma}

\begin{proof}
	The proof is similar to the safety-part in the proof of \cref{lemma:safedecompositions}.
	Let $Y\subseteq X$ be any set with $\MatPor{\CutG{B+F}{Y}}\leq w$ and suppose $\Abs{\Parts{L}{Y}}\geq k+w+1$.
	Let $L'$ be any component of $L$ and consider $\ell'\coloneqq\Abs{\CutG{B+F}{Y}\cap \E{L'}}$.
	If $\ell'\geq 1$, then 
	\begin{align*}
		\Ceil{\frac{\ell'}{2}}\leq\Abs{\Parts{L'}{Y}}\leq1+\frac{\ell'}{2}.
	\end{align*}
	Hence we obtain the following:
	\begin{align*}
		k+w+1\leq\Abs{\Parts{L}{Y}} &= \sum_{L'\text{ component of } L}\Abs{\Parts{L'}{Y}}\\
		&\leq \sum_{L'\text{ component of } L}1+\frac{\Abs{\CutG{B+F}{Y}\cap\E{L'}}}{2}\\
		&=k+\frac{\Abs{\CutG{B+F}{Y}\cap\E{L}}}{2}
	\end{align*}
	Therefore $2w+2\leq \Abs{\CutG{B+F}{Y}\cap\E{L}}$ and thus, with $F$ being $W$-completing, there must exist a $W$-extending perfect matching $M$ of $B$ such that $L$ is a family of internally $M$-conformal paths and thus, $L+F$ is a family of $M$-alternating cycles.
	Hence there exists a perfect matching of $B+F$ with at least $2w+2$ edges in $\CutG{B+F}{Y}$ contradicting the choice of $Y$.
\end{proof}

Since we will be working on a $\Brace{\mathcal{I},W}$-decomposition of bounded width, from now on the case where linkages are not $\Brace{k,w}$-limited will be ignored, as it wont occur in our algorithm.

Let $B$ be a bipartite graph with a perfect matching, $W\subseteq\E{B}$ an extendable set, $k,w\in\N$ two integers, $X\subseteq\V{B}$, and $U\subseteq\CutG{B}{X}$ a set such that $W\cup U$ is extendable and $W\cap\CutG{B}{X}\subseteq U$.
A \emph{$\Brace{k,w}$-$U$-itinerary for $X$} is a mapping $f_U$ that assigns every tuple $\Brace{\ell,\mathcal{J},J}$, where 
\begin{itemize}
	
	\item $\ell\in[1,\Abs{X}]$ is an integer,
	
	\item $\mathcal{J}$ is a distinct family of $j\in[0,k+w]$ terminal pairs from $X\setminus\V{U\setminus J}$, and
	
	\item $J\subseteq\E{B}$ is a matching covering all terminals of $\mathcal{J}$ and every edge of $J$ covers some terminal of $\mathcal{J}$ such that $W\cup U\cup J$ is extendable, and $J\cap\CutG{B}{X}=U\cap\CutG{B}{X}$,
	
\end{itemize}
a value $0$ or $1$ such that the following is guaranteed:
\begin{enumerate}
	\item If $\Fkt{f_U}{\ell,\mathcal{J},J}=0$, then there exists no $\Brace{\mathcal{J},J}$-linkage $L$ in $\InducedSubgraph{B}{X\setminus\V{U\setminus J}}$ with $\Abs{\V{L}}=\ell$ such that a $J$-extending solution $M,\mathcal{Q}$ exists with $W\cup J\cup U\subseteq M$, which is $\Brace{k,w}$-limited in $X$.
	\item If $\Fkt{f_U}{\ell,\mathcal{J},J}=1$, then there exists a $\Brace{\mathcal{J},J}$-linkage $L$ in $\InducedSubgraph{B}{X\setminus\V{U\setminus J}}$ with $\Abs{\V{L}}=\ell$ such that a $J$-extending solution $M,\mathcal{Q}$ exists with $W\cup J\cup U\subseteq M$.
\end{enumerate}

\begin{lemma}\label{lemma:mergetruejoin}
	Let $B$ be a bipartite graph with a perfect matching, $W\subseteq\E{B}$ an extendable set, and $k,w\in\N$ two integers.
	Furthermore let $X,Y\subseteq\V{B}$ be two disjoint subsets such that there is no edge between $V_1\cap Y$ and $V_2\cap X$ and let $U\subseteq\CutG{B}{X\cup Y}$ be an extendable set with $W\cap\CutG{B}{X\cup Y}\subseteq U$.
	Assume that for every $Z\in\Set{X,Y}$ and every extendable $U_Z\subseteq \CutG{B}{Z}$ with $\Abs{U_Z}\leq w$ and $W\cap\CutG{B}{Z}\subseteq U_Z$ we are given a $\Brace{k,w}$-$U_Z$-itinerary $f^Z_{U_Z}$.
	Then there exists an algorithm with running time $\Fkt{\mathcal{O}}{\Brace{k+w}!\Brace{2k+3w}^{4\Brace{k+w}}\Abs{X\cup Y}^{4k+12w+2}}$ that produces a $\Brace{k,w}$-$U$-itinerary for $X\cup Y$.
\end{lemma}

\begin{proof}
	Let $\ell\in[1,\Abs{X\cup Y}]$ and $j\in[0,k+w]$, let $\mathcal{J}=\Set{\Brace{s_1,t_1},\dots,\Brace{s_j,t_j}}$ be a distinct set of $j\in[0,k+w]$ terminal pairs in $X\cup Y$ and $J\subseteq\E{B}$ an extendable set such that every edge in $J$ covers a terminal of $\mathcal{J}$ and every terminal is covered by some edge in $J$.
	We need to choose some additional sets of edges before we can start, in order to determine the value of $\Fkt{f_U}{\ell,\mathcal{J},J}$.
	We iterate over all choices of sets $R$, $H$, $R_X$, and $R_Y$ satisfying the following requirements.
	\begin{enumerate}
		\item $R\subseteq\CutG{B}{X}\cap\CutG{B}{Y}$ such that
		\begin{itemize}
			\item $W\cap\CutG{B}{X}\cap\CutG{B}{Y}\subseteq R$,
			\item $\Abs{\Brace{U\cap\CutG{B}{X}}\cup R}\leq w$ and $\Abs{\Brace{U\cap\CutG{B}{Y}}\cup R}\leq w$, and
			\item $W\cup U\cup J\cup R$ is extendable.
		\end{itemize}
		\item $H\subseteq\CutG{B}{X}\cap\CutG{B}{Y}$ such that
		\begin{itemize}
			\item $H$ is a matching of size at most $w$,
			\item no edge of $H$ is incident with an edge of $U\cup R$, and
			\item the set of endpoints of the edges in $H$ in $Z\in\Set{X,Y}$ is denoted by $Z_H$.
		\end{itemize}
		\item For $Z\in\Set{X,Y}$, $R_Z\subseteq\E{\InducedSubgraph{B}{Z}}$ such that
		\begin{itemize}
			\item every vertex in $Z_H$ is covered by some edge of $R_Z$,
			\item every edge of $R_Z$ covers a vertex in $Z_H$, and
			\item $W\cup U\cup J\cup R\cup R_Z$ is extendable.
		\end{itemize}
	\end{enumerate}
	In what follows let $R$, $H$, and the $R_Z$ be fixed.
	For $Z\in\Set{X,Y}$ let $U_Z\coloneqq R\cup\Brace{U\cap\CutG{B}{Z}}$ and note that, by choice and our assumption, we are given a $\Brace{k,w}$-$U_Z$-itinerary $f^Z_{U_Z}$ for $Z$.
	Let us denote the set of endpoints in $V_2$ of the edges in $R_X$ by $V_{2,R_X}$ and the set of endpoints in $V_1$ of the edges in $R_Y$ by $V_{1,R_Y}$.
	
	There may exist some paths that belong to a linkage we are interested in which start in a vertex of $V_1\cap Y$ and end in a vertex of $V_2\cap X$.
	However, each such path must necessarily use an edge of $R$ and in total, since we are still only interested in $\Brace{k,w}$-limited linkages, we cannot cross the cut between $X$ and $Y$ too often.
	Still, we need to address this problem by possibly considering additional terminals not belonging to those we were given by $\mathcal{J}$.
	We approach the problem of merging the two itineraries with respect to the chosen sets above by constructing an auxiliary digraph $\InducedSubgraph{D_{W,U,X,Y}}{\mathcal{J},J,R,H,R_X,R_Y}$ of constant size.
	
	For the vertices of $\InducedSubgraph{D_{W,U,X,Y}}{\mathcal{J},J,R,H,R_X,R_Y}$ we define the following sets:
	\begin{align*}
		V_X&\coloneqq\CondSet{s_i\in X}{i\in[1,j]}\cup\CondSet{t_i\in X}{i\in[1,j]}\cup V_{2,R_X}\cup\CondSet{v_e}{e\in R}\\
		V_Y&\coloneqq\CondSet{s_i\in Y}{i\in[1,j]}\cup\CondSet{t_i\in Y}{i\in[1,j]}\cup V_{1,R_Y}\cup\CondSet{v_e}{e\in R}.
	\end{align*}
	And for the edges let
	\begin{align*}
		E_X\coloneqq\CondSet{\Brace{s_i,v}}{s_i\in V_X\text{ and }v\in V_{2,R_X}}&\cup \CondSet{\Brace{v_e,t_i}}{t_i\in V_X\text{ and }e\in R}\\
		&\cup\CondSet{\Brace{v_e,u}}{e\in R\text{ and }u\in V_{2,R_X}}\text{, and}
	\end{align*}
	\begin{align*}
		E_Y\coloneqq\CondSet{\Brace{s_i,v_e}}{s_i\in V_Y\text{ and }e\in R}&\cup \CondSet{\Brace{v,t_i}}{t_i\in V_Y\text{ and } v\in V_{1,R_Y}}\\
		&\cup\CondSet{\Brace{u,v_e}}{e\in R\text{ and }u\in V_{1,R_Y}}.
	\end{align*}
	Then 
	\begin{align*}
		\InducedSubgraph{D_{W,U,X,Y}}{\mathcal{J},J,R,H,R_X,R_Y}\coloneqq&\Brace{V_X,E_X}\cup\Brace{V_Y,E_Y}\\&+ \CondSet{\Brace{u,v}}{uw\in R_X,~wz\in H\text{, and }zv\in R_Y}.
	\end{align*}
	Let $L$ be a directed $\mathcal{J}$-linkage in $\InducedSubgraph{D_{W,U,X,Y}}{\mathcal{J},J,R,H,R_X,R_Y}$ such that $L$ has at most $t+w$ components in $\Brace{V_Z,E_Z}$ for both $Z\in\Set{X,Y}$.
	Then from $L$ we can derive two instances of the linkage problem for the matching case, one in $\InducedSubgraph{B}{X}$ and the other in $\InducedSubgraph{B}{Y}$, namely $\mathcal{J}_{L,X}\coloneqq\E{L}\cap E_X$ and $\mathcal{J}_{L,Y}\coloneqq\E{L}\cap E_Y$.
	Additionally we define for $Z\in\Set{X,Y}$
	\begin{align*}
		U_{L,Z}&\coloneqq \Brace{\CutG{B}{Z}\cap U}\cup R \text{, and}\\
		J_{L,Z}&\coloneqq \CondSet{e\in J}{e\in\E{\InducedSubgraph{B}{Z}}\cup\CutG{B}{Z}}\cup\CondSet{e\in R_Z}{e\text{ covers a terminal in }\mathcal{J}_{L,Z}}.
	\end{align*}
	If there now exist integers $\ell_1$ and $\ell_2$ with $\ell=\ell_1+\ell_2$ such that
	\begin{align*}
		\Fkt{f^X_{U_{L,X}}}{\ell_1,\mathcal{J}_{L,X},J_{L,X}}=\Fkt{f^Y_{U_{L,Y}}}{\ell_1,\mathcal{J}_{L,Y},J_{L,Y}}=1,
	\end{align*}
	the two solutions in $\InducedSubgraph{B}{X}$ and $\InducedSubgraph{B}{Y}$ can be combined and we may set $\Fkt{f_U}{\ell,\mathcal{J},J}\coloneqq 1$.
	
	In total, since we iterate over all possible choices and combinations, this process correctly computes a $\Brace{k+w}$-$U$-itinerary for $X\cup Y$.
	The running time follows from the number of possible choices we need to consider and the size and construction of $\InducedSubgraph{D_{W,U,X,Y}}{\mathcal{J},J,R,H,R_X,R_Y}$.
	Please note that the bound given in the statement of the lemma is probably not optimal, but it suffices for our purposes.
\end{proof}

\Cref{lemma:mergetruejoin} describes how to merge partial solutions at join-vertices of a $\Brace{\mathcal{I},W}$-decomposition, once a set $U$ has been fixed.
The next lemma addresses the same problem at guard-vertices.
Indeed for our purposes, it suffices to only consider guard- and join-vertices in a bottom-up fashion in order to find the desired solution.

\begin{lemma}\label{lemma:mergeguard}
	Let $B$ be a bipartite graph with a perfect matching, $W\subseteq\E{B}$ an extendable set, and $k,w\in\N$ two integers.
	Furthermore let $X,Y\subseteq\V{B}$ be two disjoint subsets such that $\MatPor{\CutG{B}{X}}\leq w$ and $\Abs{Y}\leq w$, and let $U\subseteq\CutG{B}{X\cup Y}$ be an extendable set with $W\cap\CutG{B}{X\cup Y}\subseteq U$.
	Assume that for every extendable $U_X\subseteq \CutG{B}{X}$ with $\Abs{U_X}\leq w$ and $W\cap\CutG{B}{X}\subseteq U_X$ we are given a $\Brace{k,w}$-$U_X$-itinerary $f^X_{U_X}$.
	Then there exists an algorithm with running time $\Fkt{\mathcal{O}}{\Brace{w+k}!w^{\frac{1}{2}w}\Brace{\Abs{X}+w}^{4k+10w+2}}$ that produces a $\Brace{k,w}$-$U$-itinerary for $X\cup Y$.
\end{lemma}

\begin{proof}
	Let $\ell\in[1,\Abs{X\cup Y}]$, $j\in[0,k+w]$, and $\mathcal{J}=\Set{\Brace{s_1,t_1},\dots,\Brace{s_j,t_j}}$ be a distinct set of $j\in[0,k+w]$ terminal pairs in $X\cup Y$ and $J\subseteq\E{B}$ an extendable set such that every edge in $J$ covers a terminal of $\mathcal{J}$ and every terminal is covered by some edge in $J$.
	Next we iterate over all possible choices for the sets $R$, $H$, and $R_X$ defined analogously to those in the proof of \cref{lemma:mergetruejoin}.
	Additionally, we iterate over all possible choices of $W\cup U\cup J\cup R\cup R_X$-extending perfect matchings $M_{J,R,R_X}$ of the graph $B_Y\coloneqq\InducedSubgraph{B}{Y\cup\V{W}\cup\V{U}\cup\V{J}\cup\V{R}\cup\V{R_X}}$.
	Since $\Abs{Y}\leq w$ there are at most $w^{\frac{w}{2}}$ such perfect matchings of $B_Y$.
	We need six additional vertex sets in order to construct another auxiliary digraph that will be used similarly to the one in \cref{lemma:mergetruejoin}.
	However, since we do not know which colour the endpoints of an edge in $\CutG{B}{X}\cap\CutG{B}{Y}$ have, with respect to the shore we are interested in, the construction is slightly more complicated.
	\begin{enumerate}
		\item $V_{1,R}\coloneqq\bigcup_{e\in R}e\cap V_1\cap X$
		\item $V_{2,R}\coloneqq\bigcup_{e\in R}e\cap V_2\cap X$
		\item $V_{1,H}\coloneqq\bigcup_{e\in R_X} e\cap V_1\setminus\bigcup_{e\in H}e$
		\item $V_{2,H}\coloneqq\bigcup_{e\in R_X} e\cap V_2\setminus\bigcup_{e\in H}e$
		\item $V_{1,X}\coloneqq\CondSet{s_i\in X}{\Brace{s_i,t_i}\in\mathcal{J}}$
		\item $V_{2,X}\coloneqq\CondSet{t_i\in X}{\Brace{s_i,t_i}\in\mathcal{J}}$
	\end{enumerate}
	As a first component we need the digraph $D_1\coloneqq\DirM{B_Y}{M_{J,R,R_X}}$.
	Second let
	\begin{align*}
		V_X&\coloneqq \CondSet{u'}{u\in V_{1,R}\cup V_{1,H}\cup V_{1,X}}\cup\CondSet{v'}{v\in V_{2,R}\cup V_{2,H}\cup V_{2,X}},\\
		E_X&\coloneqq \CondSet{\Brace{u',v'}}{u\in V_{1,R}\cup V_{1,H}\cup V_{1,X}\text{ and }v\in V_{2,R}\cup V_{2,H}\cup V_{2,X}}\text{, and}
	\end{align*}
	\begin{align*}
		E'\coloneqq &\CondSet{\Brace{u',v_e}}{u\in V_{1,H}\cup V_{1,R}\text{, }u\in e\in R\cup R_X\text{, and }v_e\in\V{D_1}}\\&\cup\CondSet{\Brace{v_e,v'}}{v\in V_{2,H}\cup V_{2,R}\text{, }v\in e\in R\cup R_X\text{, and }v_e\in\V{D_1}}.
	\end{align*}
	In total these definitions give rise to the digraph
	\begin{align*}
		\InducedSubgraph{D_{W,U,X,Y}}{\mathcal{J},J,R,H,R_X,M_{J,R,R_X}}\coloneqq&\Brace{V_X,E_X}\cup D_1 + E'.
	\end{align*}
	Now let 
	\begin{align*}
		\mathcal{J}'\coloneqq\CondSet{\Brace{u_i,w_i}}{u_i=s_i\text{ if }s_i\in Y\text{ if not, }u_i=s_i',~w_i=t_i\text{ if }t_i\in Y\text{ if not, }w_i=t_i'}.
	\end{align*}
	For every $e\in J$ we identify any endpoint of $J$ that is a terminal of $\mathcal{J}'$ with the vertex $v_e\in\V{D_1}$.
	Then, by construction, every solution $\mathcal{P}$, $M$ for $\mathcal{J}'$, $W$ in $\InducedSubgraph{B}{X\cup Y}$ such that $M$ extends $M_{J,R,R_X}$ naturally corresponds to a family of pairwise internally disjoint directed paths in $\InducedSubgraph{D_{W,U,X,Y}}{\mathcal{J},J,R,H,R_X,M_{J,R,R_X}}$ that links $\mathcal{J}'$.
	On the other hand, let $\mathcal{P}$ be a family of pairwise internally disjoint directed paths linking $\mathcal{J}'$, such that the following requirements are met:
	\begin{enumerate}
		\item Let $Q\coloneqq\bigcup_{P\in\mathcal{P}}P$, then the total number, over all paths $P\in\mathcal{P}$, of subgraphs of $Q$ that are a maximal directed subpaths of $\InducedSubgraph{P}{V_X\cup \CondSet{v_e}{e\in R_X}}$ does not exceed $k+w$.
		
		\item If $v\in\V{P}$ such that $v=v_e$ for some $e\in M_{J,R,R_X}$ and there is some $u'\in V_X$ such that $u\in e$, then $u'$ does not occur in any other path of $\mathcal{P}$.
		Similarly, if $u'\in\V{P}\cap V_X$ such that some $e\in M_{J,R,R_X}$ exists with $u\in e$, then $v_e$ does not occur in any path of $\mathcal{P}$ besides possibly $P$.
	\end{enumerate}
	Let $P'$ be a subgraph of $Q$ that is a maximal directed subpath of $\InducedSubgraph{P}{V_X\cup \CondSet{v_e}{e\in R_X}}$ for some $P\in\mathcal{P}$ and let $u'_{P'}$ be the starting point of $P$ and $v'_{P'}$ its end.
	We define a terminal pair $\Brace{u,v}$ in $\InducedSubgraph{B}{X}$ as follows:
	\begin{itemize}
		\item If $u_{P'}\in V_{1,R}\cup V_{1,H}\cup V_{1,X}$ set $u\coloneqq u_{P'}$, otherwise there must be some $e\in R_X$ such that $u'_{P'}=v_e$.
		In this case let $u_e$ be the endpoint of $e$ in $V_1$ and set $u\coloneqq u_e$.
		
		\item Similarly, if $v_{P'}\in V_{2,R}\cup V_{2,H}\cup V_{2,X}$ set $v= v_{P'}$, otherwise there must be some $e\in R_X$ such that $v'_{P'}=v_e$.
		In this case let $u_e'$ be the endpoint of $e$ in $V_2$ and set $v\coloneqq u_e'$.
	\end{itemize}
	Let $\mathcal{J}_{\mathcal{P}}$ be the collection of all terminal pairs $\Brace{u,v}$ defined as above.
	Then no vertex of $X$ occurs in two different terminal pairs of $\mathcal{J}_{\mathcal{P}}$ and every terminal is covered by an edge of $J\cup R_X$.
	We define two additional sets as before:
	\begin{align*}
		U_{\mathcal{P},X}&\coloneqq \Brace{\CutG{B}{X}\cap U}\cup R \text{, and}\\
		J_{\mathcal{P},X}&\coloneqq \CondSet{e\in J}{e\in\E{\InducedSubgraph{B}{X}}\cup\CutG{B}{X}}\cup\CondSet{e\in R_X}{e\text{ covers a terminal in }\mathcal{J}_{L,X}}.
	\end{align*}
	If there now exist integers $\ell_1$ and $\ell_2$ with $\ell=\ell_1+\ell_2$ such that 
	\begin{align*}
		\ell_2=2\Abs{\V{Q}\cap \CondSet{v_e}{e\in M_{J,R,R_X}\setminus R_X}}\text{, and }\Fkt{f^X_{U_{\mathcal{P},X}}}{\ell_1,\mathcal{J}_{\mathcal{P}},J_{\mathcal{P},X}}=1,
	\end{align*}
	We can combine the parts of $\mathcal{P}$ in $\InducedSubgraph{D_1}{\CondSet{v_e}{e\in M_{J,R,R_X}\setminus R_X}}$ and a solution for $\mathcal{J}_{\mathcal{P}}$ to obtain a solution for $X\cup Y$ and $U$.
	Hence we may set $\Fkt{f_U}{\ell,\mathcal{J},J}\coloneqq 1$.
	By iterating over all possible choices for the various sets we are sure to produce a complete $\Brace{k,w}$-$U$-itinerary for $X\cup Y$.
\end{proof}

Using \cref{lemma:mergetruejoin,lemma:mergeguard}, we are now able to merge partial solutions at all join- and guard vertices.
For basic vertices obtaining partial solutions is straight forward, since we may only choose the edges of the perfect matchings covering the two singular vertices that lie in the two subtrees beneath.
In order to obtain a $\Brace{k,w}$-$U$-itinerary for every possible $U$, we just have to call the corresponding merge operation for every possible choice of $U$.
At any given time there are $\Fkt{\mathcal{O}}{\Abs{\V{B}}^w}$ such choices, which overall implies the following:

\begin{corollary}\label{cor:disjointpathsforfixedW}
	Let $B$ be a bipartite graph with a perfect matching, $\mathcal{I}$ a distinct set of $k$ terminal pairs, $W$ an extendable set covering all terminals such that every edge in $W$ covers a terminal and $\Brace{T,\delta}$ a $\Brace{\mathcal{I},W}$-decomposition of width $w$ for $B$.
	There exists an algorithm that decides in time $\Fkt{\mathcal{O}}{\Abs{\V{B}}^{4k+13w+3}}$ whether there exists a solution for $\mathcal{I}$, $W$ or not.
\end{corollary}

\Cref{cor:disjointpathsforfixedW}, together with the approximation factor for our $\Brace{\mathcal{I},W}$-decomposition from \cref{lemma:safedecompositions}, fixes 
\begin{align*}
	\Fkt{f_2}{\pmw{B},k,\Abs{\V{B}}}\coloneqq\Fkt{\mathcal{O}}{\Abs{\V{B}}^{5616\pmw{B}+11232\pmw{B}^2+189+73k}}.
\end{align*}
Together with our previous results this completes the proof of \cref{thm:disjointpaths}.

\subsection{Counting Perfect Matchings}\label{subsec:counting}

To count the number of perfect matchings in a graph of bounded perfect matching width is another relatively straight forward application of dynamic programming.
We describe the algorithm for general, so not necessarily bipartite, graphs, however, for non-bipartite graphs there is currently no algorithm known to compute a perfect matching decomposition of bounded width.
Hence we consider the decomposition itself to be part of the input.
This means that \cref{thm:countmatchings} will follow from the results presented in this subsection by an application of \cref{thm:approximatepmw}.

Let $G$ be a graph with a perfect matching and $\Brace{T,\delta}$ be a perfect matching decomposition of width $w\in\N$ for $G$.
Let us select a root $r\in\V{T}$ and let $\vec{T}$ be the orientation of $T$ obtained by orienting every edge of $T$ away from $r$.
Moreover, for every $\Brace{d,t}\in\E{\vec{T}}$ let $T_t$ denote the component of $\vec{T}-d$ that contains $t$.

For every vertex $t\in\V{T}\setminus\Set{r}$ with unique incoming edge $\Brace{d,t}$ we will compute a value $\Fkt{\mu}{t,F}\in\N$ where $F\subseteq\CutG{G}{dt}$ and $\Abs{F}\leq w$ such that $\Fkt{\mu}{t,F}=0$ if and only if $F$ is not extendable, and otherwise $\Fkt{\mu}{t,F}$ is the number of perfect matchings of $\InducedSubgraph{G}{\Fkt{\delta}{t}\cup \V{F}}$ that extend $F$.
Since $\Width{T,\delta}=w$ no set $F\subseteq\CutG{G}{dt}$ with $\Abs{F}\geq w+1$ can be extendable and thus for every $t\in\V{T}\setminus\Set{r}$ we only have to consider $\Abs{\V{G}}^{2w}$ many such sets $F$.
In an additional step we will compute the value $\Fkt{\mu}{r,\emptyset}$ using similar techniques as before to obtain the total number of perfect matchings in $G$.
In what follows we write $\Fkt{\mu}{t,\cdot}$ as a placeholder for $\Fkt{\mu}{t,F}$ for every $F\subseteq\CutG{G}{dt}$ with $\Abs{F}\leq w$.

\begin{proposition}\label{thm:countmatchingsgeneral}
Let $G$ be a graph with a perfect matching and $\Brace{T,\delta}$ be a perfect matching decomposition of width $w\in\N$ for $G$.
There exists an algorithm that computes in time $\Fkt{\mathcal{O}}{\Abs{\V{G}}^{4w+1}}$ the number of perfect matchings in $G$.
\end{proposition}

\begin{proof}
Let us assume that the root $r$ is not a leaf.
Throughout the proof let us fix the convention that, given $t\in\V{T}\setminus\Set{r}$, $d\in\V{T}$ is the unique vertex with $\Brace{d,t}\in\E{\vec{T}}$.
	
Let us assume $t\in\V{T}$ to be a leaf.
In this case the only extendable sets $F\subseteq\CutG{G}{dt}$ are of cardinality one.
Indeed, we have to test at most $\Abs{\V{G}}-1$ edges incident with $t$ whether they are contained in a perfect matching or not.
This can clearly be done in polynomial time \cite{edmonds1965paths}.

Next let us assume $t\in\V{T}\setminus\Set{r}$ is not a leaf and has successors $t_1$ and $t_2$.
Moreover, assume that $\Fkt{\mu}{t_i,\cdot}$ have already been computed for both $i\in[1,2]$.
Observe that $\CutG{G}{dt}=\Brace{\CutG{G}{tt_1}\cup\CutG{G}{tt_2}}\setminus \Brace{\CutG{G}{tt_1}\cap\CutG{G}{tt_2}}$.
Let $F\subseteq\CutG{G}{dt}$ be an extendable set of edges and let $F_i\coloneqq \CutG{G}{tt_i}$ for both $i\in[1,2]$.
Let $\mathcal{W}$ be the collection of all sets $W\subseteq \CutG{G}{tt_1}\cap\CutG{G}{tt_2}$ such that $F\cup W$ is extendable.
Note that $W\cap F=\emptyset$, and $\Abs{F_i\cup W}\leq w$ for both $i\in[1,2]$ by definition and the width of $\Brace{T,\delta}$.
We claim that every perfect matching $M$ of $G$ with $M\cap \CutG{G}{dt}=F$ contains such a set $W$.
To see this simply observe that every edge of $M$ which belongs to $\CutG{G}{tt_1}\cup\CutG{G}{tt_2}$ but not to $\CutG{G}{dt}$ must belong to $\CutG{G}{tt_1}\cap\CutG{G}{tt_2}$.
Hence with $M\cap \CutG{G}{tt_1}\cap\CutG{G}{tt_2}$ we have found the set we wanted.
Therefore the number
\begin{align*}
	\Fkt{\mu}{t,F}\coloneqq\sum_{W\in\mathcal{W}}\Fkt{\mu}{t_1,F_1\cup W}\cdot\Fkt{\mu}{t_2,F_2\cup W}
\end{align*}
is exactly the number of perfect matchings of $\InducedSubgraph{G}{\V{F}\cup\Fkt{\delta}{T_t}}$ as intended.
To compute this number we have to consider every set $W\subseteq\CutG{G}{tt_1}\cap\CutG{G}{tt_2}$ of size at most $w-\Abs{F}$ and test, whether $W\cup F$ is extendable or not.
So in total we perform $\Fkt{\mathcal{O}}{\Abs{\V{G}}^{2w}}$ such tests.
Since we also have $\Fkt{\mathcal{O}}{\Abs{\V{G}}^{2w}}$ such sets $F$ we need to consider, $\Fkt{\mu}{t,\cdot}$ can be computed in $\Fkt{\mathcal{O}}{\Abs{\V{G}}^{4w}}$ steps.

At last let us consider the root $r$ with its three successors $t_1$, $t_2$, and $t_3$ and suppose we are given $\Fkt{\mu}{t_i,\cdot}$ for every $i\in[1,3]$.
Let $M$ be any perfect matching of $G$, then $M\cap \bigcup_{i=1}^3\CutG{G}{rt_i}$ can be decomposed into two sets: A set $F_M\subseteq\CutG{G}{rt_1}$ and a set $W_M\subseteq \bigcup_{i=1}^3\CutG{G}{rt_i}\setminus\CutG{G}{rt_1}=\CutG{G}{rt_2}\cap\CutG{G}{rt_3}$.
Hence we may use the same method as for ne non-root inner vertices to compute $\Fkt{\mu}{r,\emptyset}$ by simply considering all extendable sets $F\subseteq\CutG{G}{rt_1}$ and then for each such $F$ every possible set $W\subseteq \CutG{G}{rt_2}\cap\CutG{G}{rt_3}$ such that $F\cup W$ is extendable.
Let $\mathcal{F}_r$ be the set of all extendable subsets $F$ of $\CutG{G}{rt_1}$ and let $\mathcal{W}_{r,F}$ be the set of all sets $W\subseteq \CutG{G}{rt_2}\cap\CutG{G}{rt_3}$ such that $F\cup W$ is extendable.
For each $i\in[2,3]$ let $F_i\coloneqq F\cap \CutG{G}{rt_i}$.
We set
\begin{align*}
	\Fkt{\mu}{r,\emptyset}\coloneqq\sum_{F\in\mathcal{F}_r}\sum_{W\in\mathcal{W}_{r,F}}\Fkt{\mu}{t_1,F}\cdot\Fkt{\mu}{t_2,F_2\cup W}\cdot\Fkt{\mu}{t_3,F_3\cup W}.
\end{align*} 
It follows from the discussion above that $\Fkt{\mu}{r,\emptyset}$ is the number of perfect matchings of $G$.
Moreover, since $T$ is a cubic tree with $\Abs{\V{G}}$ leaves, we have $\Abs{\V{T}}\in\Fkt{\mathcal{O}}{\Abs{\V{G}}}$, and thus, the total running time of the procedure is $\Fkt{\mathcal{O}}{\Abs{\V{G}}^{4w+1}}$.
\end{proof}

\subsection{Consequences for Matching Minor Checking and Bipartite Graphs Excluding a Planar Matching Minor}\label{subsec:rest}

The importance of the disjoint paths problem in the Graph Minors series by Robertson and Seymour is due to the fact that checking for minor containment can be reduced to certain instances of the disjoint paths problem.
For bipartite graphs with perfect matchings, this is also true.

\begin{proof}[Proof of \cref{thm:matchingminors}]
	By \cref{cor:Mmodels}, if $H$ is a matching minor of $B$, then there exists a perfect matching $M$ of $B$ such that there exists an $M$-model of $H$ in $B$.
	Every vertex $v\in\V{H}$ is represented by a barycentric tree in $B$, and it is not hard to see that we may always choose such a barycentric tree such that the number of vertices of degree at least $3$ is at most $\DegG{H}{v}$.
	Let $\mu\colon H\rightarrow B$ be such a model.
	By \cref{lemma:matchingsofmodels} we may further assume that $M$ corresponds to a perfect matching $M_H$ of $H$ and $\Fkt{\mu}{uv}$ is internally $M$-conformal if and only if $uv\notin M_H$.
	Moreover, if $uv\in M_H$, then $\Fkt{\mu}{uv}$ is $M$ conformal.
	Hence for every $v$, it suffices to guess the at most $\DegG{H}{v}$ many edges of $M$ and ask for pairwise internally disjoint internally $M$-conformal paths connecting them in an appropriate way.
	Additionally, we need an internally $M$-conformal path representing every $uv\in\E{H}\setminus M_H$ and for each of those, we need to find an edge of $M$ for each of the two endpoints.
	Since we also guessed the edges of $M$ covering the only vertex of $\Fkt{\mu}{x}$ not covered by $\E{\Fkt{\mu}{x}}\cap M$ for both $x\in\Set{u,v}$ if $uv\in M_H$, the endpoints of these two edges, not belonging to their respective vertex models must also be linked by paths.
	Since it is not feasible to check for $M$-models of $H$ for every perfect matching $M$ of $B$, we instead check all possible choices of extendable sets $F$ of size at most $2\Abs{\E{H}}+\sum_{v\in\V{H}}\DegG{H}{v}=4\Abs{\E{H}}$.
	In fact, since we also do not know which edge of our set $F$ belongs to the model of which vertex or edge of $H$, we also need to try all possible configurations.
	But this only worsens our running time by a factor depending exclusively on the size of $F$.
	Hence in total we need to call the algorithm from \cref{thm:disjointpaths} at most $\Fkt{\mathcal{O}}{\Abs{\V{B}}^{4\Abs{\E{H}}}}$ times with $k\leq 4\Abs{\E{H}}\leq4\Abs{\V{H}}^2$, and thus our claim follows.
\end{proof}

\Cref{cor:countpmsexcludingplanar,cor:planarmatchingminors} now both follows by applying \cref{thm:boundedclasses} to the findings from this section.
Additionally an algorithmic version of \cref{thm:matchingEP} can be achieved using similar arguments.
Moreover, it follows from \cref{lemma:excludingantichains} that we can obtain analogous results for the testing for members of fundamental anti-chains as butterfly minors on digraphs of bounded directed treewidth and, in particular, for every strongly connected strongly planar digraph $H$ there exists a polynomial-time algorithm that decides whether any given digraph $D$ contains a member of $\Antichain{H}$ as a butterfly minor.
This might be especially surprising since $\Antichain{H}$ may be infinite.

\section{Conclusion}

So far the only bipartite graphs with perfect matchings for which we were able to test for matching minor containment in bipartite graphs were $C_4$, the cube, and $K_{3,3}$.
Moreover, it was part of the original motivation for the study of $K_{3,3}$ matching minor free bipartite graphs that on these graphs the  number of perfect matchings can be computed efficiently.
In this paper we have established perfect matching width as a useful parameter for the study of matching minors, at least in bipartite graphs.
Moreover, we have shown both the recognition problem of classes of bipartite graphs excluding a planar matching minor $H$, and the problem of counting the number of perfect matchings in bipartite graphs in $H$ matching minor free graphs to be in $\Poly$.
A natural question to ask is, whether or not these observations can be extended to bipartite graphs with perfect matchings that exclude some non-planar matching covered graph as a matching minor.
Towards this goal we formulate three questions, which pose as the main motivation of our research.

\begin{question}
	What is the computational complexity of the $t$-DAPP?
\end{question}

\begin{question}
	Let $H$ be a non-planar bipartite matching covered graph.
	What is the computational complexity of deciding whether a given bipartite graph $B$ with a perfect matching contains $H$ as a matching minor?
\end{question}

\begin{question}
	Let $H$ be any bipartite matching covered graph.
	What is the computational complexity of computing the number of perfect matchings in a bipartite graph $B$ which does not contain $H$ as a matching minor?
\end{question}

Finally, we conclude this work with a question regarding infinite anti-chains of butterfly minors.
Let $\mathcal{F}$ be a, possibly infinite, family of digraphs.
We say that $\mathcal{F}$ can be \emph{covered} by a finite number of fundamental anti-chains if there exist digraphs $D_1,\dots,D_m$ such that $\mathcal{F}\subseteq\bigcup_{i=1}^m\Antichain{D_i}$.

\begin{question}
	Is there an infinite family $\mathcal{F}$ of strongly $2$-connected digraphs such that $\mathcal{F}$ cannot be covered by a finite number of fundamental anti-chains?
\end{question}

\bibliographystyle{alphaurl}
\bibliography{literature}

\end{document}